\documentclass[10pt,reqno]{amsart}
\usepackage[utf8]{inputenc}
\usepackage[margin=1in]{geometry}
\usepackage{amsfonts,amssymb,amsmath,amsthm,tikz,comment,mathtools,setspace,stmaryrd,enumitem,bbm}
\usepackage[hidelinks]{hyperref}
\usepackage{scalerel,stmaryrd} 
\numberwithin{equation}{section}
\usepackage[normalem]{ulem}

\allowdisplaybreaks

\newcommand{\N}{\mathbb{N}}
\newcommand{\R}{\mathbb{R}}
\newcommand{\Q}{\mathbb{Q}}
\newcommand{\Z}{\mathbb{Z}}
\newcommand{\F}{\mathcal{F}}

\newcommand{\A}{\mathcal{A}}
\newcommand{\X}{\mathcal{X}}
\newcommand{\Y}{\mathcal{Y}}
\newcommand{\B}{\mathcal{B}}

\newcommand{\D}{\mathcal{D}}

\newcommand{\Hh}{\mathcal{H}}

\newcommand{\li}{\;{\le}_{\rm inc}\;}

\newcommand{\sli}{\;{<}_{\rm inc}\;}
\newcommand{\sgi}{\;{>}_{\rm inc}\;}
\newcommand{\CRpin}{C_{\text{pin}}(\R)}
\newcommand{\Ll}{\mathcal L}
\newcommand{\h}{\mathfrak h}
\newcommand{\wt}{\widetilde}
\newcommand{\Ss}{\mathcal{S}}
\newcommand{\T}{\mathcal{T}}
\newcommand{\deq}{\overset{d}{=}}

\newcommand{\ve}{\varepsilon}

\newcommand{\f}{\frac}

\newcommand{\mbf}{\mathbf}

\newcommand{\Pp}{\mathbb P}
\newcommand{\Ee}{\mathbb E}

\newcommand{\KPZH}{\operatorname{KPZH}}

\newcommand{\kpzb}{h_{Z_\beta}}
\newcommand{\bb}{b_\beta}
\newcommand{\Rup}{\R_{\uparrow}^4}

\newcommand{\dir}{\lambda}

\newcommand{\KH}{F}
\newcommand{\OCY}{Z^{\rm sd}}

\newcommand{\ind}{\mathbf 1}

\newcommand{\erfc}{\operatorname{erfc}}

\newcommand{\Dp}{D}
\newcommand{\Qp}{Q}
\newcommand{\Rp}{R}
\newcommand{\sDp}{\mathbf{D}}

\newcommand{\be}{\begin{equation}}
\newcommand{\ee}{\end{equation}}
\newcommand{\fl}[1]{\lfloor{#1}\rfloor} 

\newcommand{\sig}{{\scaleobj{0.8}{\boxempty}}} 
\newcommand{\sigg}{{\scaleobj{0.9}{\boxempty}}}

\newtheorem{theorem}{Theorem}[section]
\newtheorem{proposition}[theorem]{Proposition}
\newtheorem{corollary}[theorem]{Corollary}
\newtheorem{lemma}[theorem]{Lemma}

\theoremstyle{definition}
\newtheorem{definition}[theorem]{Definition}

\theoremstyle{remark}
\newtheorem*{remark}{Remark}

\RequirePackage{color}
\definecolor{darkblue}{rgb}{0.0,0.0,0.7}
\definecolor{darkred}{rgb}{0.5,0.0,0.0}
\definecolor{darkgreen}{rgb}{0.0,0.5,0.0}
\definecolor{indigo}{rgb}{0.3,0,0.5}

\newcommand{\pathsp}{\mathbb{X}}

\def\tsp{\hspace{0.5pt}}  
\def\tspa{\hspace{0.7pt}}
\def\tspb{\hspace{0.9pt}}

\def\viiva{\hspace{.7pt}\vert\hspace{.8pt}}
\def\bviiva{\hspace{.8pt}\big\vert\hspace{.8pt}}   
\def\Bviiva{\hspace{.7pt}\Big\vert\hspace{.7pt}}

\newcommand\aabullet{{\tspb\raisebox{2pt}{\scaleobj{0.5}{\bullet}}\tspb}}  
\newcommand\aaabullet{{\tspb\raisebox{1pt}{\scaleobj{0.5}{\bullet}}\tspb}}  

\def\ddd{\displaystyle}

\title{Jointly invariant measures for the Kardar-Parisi-Zhang equation}

\author{Sean Groathouse}
\address{Sean Groathouse, University of Utah, Mathematics Department, 155 South 1400 East, JWB 233
Salt Lake City, UT 84112, USA.}
\email{sean.groathouse@gmail.com}

\author{Firas Rassoul-Agha}
\address{Firas Rassoul-Agha, University of Utah, Mathematics Department, 155 South 1400 East, JWB 233
Salt Lake City, UT 84112, USA.}
\email{firas@math.utah.edu}

\author{Timo Sepp{\"a}l{\"a}inen}
\address{Timo Sepp{\"a}l{\"a}inen, University of Wisconsin-Madison, Mathematics Department, Van Vleck Hall, 480
Lincoln Dr., Madison WI 53706, USA.}
\email{seppalai@math.wisc.edu}
\author{Evan Sorensen}
\address{Evan Sorensen (corresponding author) Columbia University, Mathematics Department,  Room 624, MC 4432,
2990 Broadway,
New York, NY 10027, USA.}
\email{es4203@columbia.edu}


\begin{document}
\maketitle

\begin{abstract}
    We give an explicit description of a family of jointly invariant measures of the KPZ equation singled out by asymptotic slope conditions. These are couplings of Brownian motions with drift, and can be extended to a cadlag process indexed by all real drift parameters. We name this process the KPZ horizon (KPZH). As a corollary, we resolve a recent conjecture by showing the existence of a random, countably infinite dense set of drift values at which the Busemann process of the KPZ equation is discontinuous. This signals instability, and shows the failure of the one force--one solution principle and the existence of at least two extremal semi-infinite polymer measures in the exceptional directions. The low-temperature limit of the KPZH is the stationary horizon (SH), the unique jointly invariant measure of the KPZ fixed point under the same slope conditions. The high-temperature limit of the KPZH is a coupling of Brownian motions that differ by linear shifts, which is jointly invariant under the Edwards-Wilkinson fixed point. 
\end{abstract}
\tableofcontents

\section{Introduction}
\subsection{Invariant measures of the KPZ equation}
For $t > s$, consider the KPZ equation
\be \label{eqn:KPZ}
\begin{aligned}
\partial_t h(t,x) = \f{1}{2} \partial_{xx} h(t,x) + \f{\beta}{2}(\partial_x h(t,x))^2 +  W(t,x),\quad  h(s,x) = h_s(x),
\end{aligned}
\ee
with inverse temperature $\beta>0$, initial condition $h_s$ at time $s$ and 
 space-time white noise $W$ as driving force.
Classically, this equation is ill-posed, but formally, one can solve the KPZ equation via the Cole-Hopf transformation $h(t,x) = \f{1}{\beta}\log Z(t,x)$, where $Z$ solves the stochastic heat equation (SHE) with multiplicative noise:  
\be \label{eqn:SHE}
\partial_t Z(t,x) = \f{1}{2} \partial_{xx}Z(t,x) + \beta Z(t,x)W(t,x),\quad Z(s,x) = e^{\beta h_s(x)}.
\ee
Rigorous solutions to this equation have been discussed in \cite{Bertini-Cancrini-1995,Bertini-Giacomin-1997,Chen-Dalang-2014,Chen-Dalang-2015}. Recently, great progress has been made in understanding solutions of the KPZ equation in the work on Martin Hairer \cite{Hairer-2013,Hairer-2014} on regularity structures. Another perspective through paracontrolled distributions has been studied in \cite{Gubinelli-Perkowski-2017,Perkowski-Rosati-2019}. 

It is well-known that Brownian motion with diffusivity $1$ and arbitrary drift is an invariant measure for \eqref{eqn:KPZ}. 
The notion of invariance requires the caveat that invariance only holds up to a global height shift. That is, if we let $h(t,x \viiva B)$ denote the solution to \eqref{eqn:KPZ} at time $t > 0$ when $h(0,x) = B(x)$ is a Brownian motion, then,
\[
\{h(t,x \viiva B) - h(t,0 \viiva B):x \in \R \} \deq B.
\]
See \cite{Janjigian-Rassoul-Seppalainen-19} and the references therein for a detailed discussion of this height shift.
Using the work of \cite{Alberts-Khanin-Quastel-2014a,Alberts-Khanin-Quastel-2014b,Alb-Janj-Rass-Sepp-22}, one can construct solutions to the KPZ equation with the same driving noise $W$ but started from different initial conditions. The present paper is concerned with \textit{jointly invariant measures}, namely couplings of Brownian motions  $F^1,\ldots,F^k$ with different drifts such that, on $C(\R)^k$,   we have for all $t>0$ the distributional invariance 
\be \label{joint_invar}
\bigl(h(t,\aabullet \viiva F^1) - h(t,0 \viiva F^1),\ldots,h(t,\aabullet \viiva F^k) - h(t,0 \viiva F^k)\bigr)  \deq \bigl(F^1(\aabullet),\ldots,F^k(\aabullet)\bigr).
\ee
The existence, uniqueness, and ergodicity of such jointly invariant measures, up to an asymptotic slope condition, was established in \cite{Janj-Rass-Sepp-22} (see Section 3.4 of that paper for a detailed discussion). We state this condition as follows:
\be \label{eqn:SHE_drift_cond_intro}
\begin{aligned}
    -\infty \le \limsup_{x \to -\infty} \f{F(x)}{|x|} &< \lambda = \lim_{x \to \infty} \f{F(x)}{x} \qquad\qquad\qquad\ \text{ if }\lambda > 0 \\
    \lim_{x \to -\infty} \f{F(x)}{|x|} &= |\lambda| > \limsup_{x \to \infty} \f{F(x)}{x} \ge -\infty \qquad \text{ if }\lambda < 0 \\
    -\infty &\le \limsup_{|x| \to \infty} \f{F(x)}{|x|} \le 0\qquad \qquad\qquad\text{ if }\lambda = 0.
\end{aligned}
\ee
Our first theorem gives an explicit description of these measures.
\begin{theorem} \label{thm:Busedrift}
Let $\lambda_1 < \cdots < \lambda_k$ be real. Let $Y^1,\ldots,Y^k$ be independent two-sided Brownian motions with diffusivity $1$ and drifts $\lambda_1,\ldots,\lambda_k$, respectively. 
Set $F_\beta^1=Y^1$ and then for 
 $j =2, \dotsc, k$ define $F_\beta^j$ through
\begin{align*}
\exp\bigl[\beta F_\beta^j(y)\bigr] 
&=e^{\beta Y^1(y)} \cdot 
\f{
\ddd\int\limits_{-\infty < x_{j - 1} < \cdots <x_{1} < y} \prod_{i = 1}^{j-1} e^{\beta( Y^{i + 1}(x_i) - Y^i(x_i))} d x_{ i}}{\ddd\int\limits_{-\infty < x_{j - 1} < \cdots <x_{1} < 0} \prod_{i = 1}^{j-1} e^{\beta (Y^{i + 1}(x_i) - Y^i(x_i)) } d x_{ i}}\,, \quad y\in\R.
\end{align*}
Then, $(F_\beta^1,\ldots,F_\beta^k)$ is distributed as the unique jointly stationary and ergodic measure for the KPZ equation \eqref{eqn:KPZ} such that, for $1 \le j \le k$, each $F_\beta^j$ satisfies almost surely  the asymptotic slope condition \eqref{eqn:SHE_drift_cond_intro} for $\lambda = \lambda_j$. In particular, $F_\beta^j$ is a two-sided Brownian motion with diffusivity $1$ and drift $\lambda_j$.
\end{theorem}
In Section \ref{sec:KPZH_cons_sec}, we extend the measures of Theorem \ref{thm:Busedrift} to a process $\{\KH_\beta^\lambda\}_{\lambda \in \R}$, which we name the \textit{KPZ horizon} with inverse temperature $\beta$ ($\KPZH_\beta$ for short, or sometimes simply $\KPZH$). The path space of this process is the Skorokhod space $D(\R,C(\R))$ of functions $\R \to C(\R)$ that are right-continuous with left limits. $C(\R)$ is endowed with its Polish topology of uniform convergence on compact sets. The term KPZ horizon is introduced  in analogy to its zero-temperature counterpart, the stationary horizon (SH), introduced by Busani in \cite{Busani-2021} and studied by Busani and the third and fourth authors in \cite{Seppalainen-Sorensen-21b,Busa-Sepp-Sore-22a,Busa-Sepp-Sore-22b,Busa-Sepp-Sore-23}. 
The KPZ fixed point is the large-time scaling limit of the KPZ equation under the $1:2:3$ scaling \cite{KPZ_equation_convergence,heat_and_landscape,Wu-23}. This is discussed more in Section \ref{sec:SH-KPZH} below.  
The SH gives the unique jointly invariant measure of the KPZ fixed point under the same asymptotic slope conditions \eqref{eqn:SHE_drift_cond_intro}.  This picture is completed by the convergence, as  $\beta \nearrow \infty$, of the projections of $\KPZH_\beta$ on $C(\R,\R^k)$   to those of SH (Theorem \ref{thm:SHlimit} below). 

Theorem \ref{thm:Busedrift} gives rise to the following description of the difference  of two marginals of KPZH. 
\begin{theorem} \label{thm:lambda-F}
Let $\beta > 0$ and  $\{\KH_\beta^\lambda\}_{\lambda \in \R}$ be the $\KPZH_\beta$. For $\lambda_1 < \lambda_2$ with $\lambda=\lambda_2 - \lambda_1$, 
\[
\{\KH_\beta^{\lambda_2}(y) - \KH_\beta^{\lambda_1}(y): y \in \R \} \deq \Biggl\{\beta^{-1} \log\Biggl(\f{\int_{-\infty}^y \exp\bigl(\sqrt 2\tsp \beta B(x) + \lambda \beta x\bigr)\,dx}{\int_{-\infty}^0 \exp\bigl(\sqrt 2\tsp \beta B(x) + \lambda \beta x\bigr)\,dx}\Biggr): y \in \R\Biggr\}.
\]
In particular, 
\[
\{\KH_\beta^{\lambda_2}(y) - \KH_\beta^{\lambda_1}(y): y \ge 0 \} \deq  \{{\beta}^{-1} \log\bigl(1 + X_{\lambda,\beta} Y_{\lambda,\beta}(y)\bigr):y \ge 0 \}  
\]
where $X_{\lambda,\beta}\sim{\rm Gamma}(\lambda \beta^{-1}, \beta^{-2})$, independent of the process $\{Y_{\lambda,\beta}(y):y \ge 0\}$. The law of this latter process is given  by 
\[
\{ Y_{\lambda,\beta}(y): y \ge 0 \} \;\deq \; \Bigl\{\int_0^y \exp\bigl(\sqrt 2\tsp \beta B(x) + \lambda \beta x\bigr)\,dx :y \ge 0 \Bigr\} ,
\]
where $B$ is a standard Brownian motion.
\end{theorem}
\noindent 
The analogous description of the process $\{\KH_\beta^{\lambda_2}(y) - \KH_\beta^{\lambda_1}(y): y \le 0 \}$ can be obtained through the symmetry in Theorem \ref{thm:dist_invar}\ref{itm:reflinv}. Qualitatively, a  key feature of the description of these increments is that the extension to the  full KPZH process inherently produces  discontinuities: 
\begin{theorem} \label{thm:jumps_every_edge}
Let $\KH_\beta = \{\KH_\beta^\lambda\}_{\lambda \in \R}$ be the $\KPZH_\beta$ and $\Pp_\beta$ its distribution on the space $D(\R,C(\R))$.
Then, $\Pp_\beta$-almost surely there exists a random countably infinite dense subset $\Lambda_{\beta}$ of $\R$ such that  whenever $x \neq y$, $\alpha \mapsto \KH_\beta^{\alpha}(y) - \KH_\beta^\alpha(x)$ is discontinuous at $\alpha = \lambda$ if and only if $\lambda \in \Lambda_{\beta}$.
\end{theorem}

\subsection{Discontinuities of the Busemann process in the continuum directed random polymer}
The work of \cite{Alberts-Khanin-Quastel-2014a,Alberts-Khanin-Quastel-2014b,Alb-Janj-Rass-Sepp-22} constructs a strictly positive, continuous  four-parameter random field $\{Z_\beta(t,y\viiva s,x):x,y \in \R, s  < t \}$ on an appropriate probability space of the white noise so that, for each $s \in \R$ and suitable  initial data $h_s$, 
\[
(t,y) \mapsto \int_\R e^{\beta h_s(x)}Z_\beta(t,y\viiva s,x)\,dx
\]
solves the SHE \eqref{eqn:SHE} at times $t\in(s,\infty)$ and agrees with the notion of solution from \cite{Bertini-Cancrini-1995,Bertini-Giacomin-1997,Chen-Dalang-2014,Chen-Dalang-2015}.  This four-parameter field defines random probability measures $Q_\beta^{(s,x) \to (t,y)}$ on paths $g:[s,t] \to \R$ from $(x,s)$ to $(y,t)$ whose time-$r$ distribution is given by  
\[
Q_\beta^{(s,x) \to (t,y)}(g(r) \in dz) = \f{Z_\beta(t,y\viiva r,z)Z_\beta(r,z\viiva s,x)}{Z_\beta(t,y\viiva s,x)}\, dz 
\qquad\text{for } s<r<t.
\]

In this sense, we say that $Z_\beta$ is the partition function for the \textit{continuum directed random polymer} (CDRP) first introduced in \cite{Alberts-Khanin-Quastel-2014a}. The measures 
$Q_\beta^{(s,x) \to (t,y)}$ extend in a Gibbsian sense to measures $Q_\beta^{(t,y)}$ on 
 semi-infinite backward  paths $g:(-\infty,t] \to \R$ rooted at $(t,y)$. 
 The Gibbs property is that, conditional on the path passing through $(s,x)$ at time $s \in(-\infty,t)$, the portion of the path between $(t,y)$ and $(s,x)$ is distributed as $Q_\beta^{(s,x) \to (t,y)}$. See \cite[Section 9]{Janj-Rass-Sepp-22} for a more precise definition and detailed discussion. The infinite-path measure is said to be \textit{strongly $\lambda$-directed}  if 
\[
Q_\beta^{(t,y)}\biggl(\lim_{r \to -\infty} \f{g(r)}{|r|} = \lambda\biggr) = 1.
\]


To study this collection of infinite-path measures, Janjigian and the second and third authors \cite{Janj-Rass-Sepp-22} constructed \textit{Busemann functions} for the SHE. For a fixed $\lambda \in \R$, these satisfy the almost sure locally uniform limits \cite[Theorem 3.16]{Janj-Rass-Sepp-22}
\be \label{eqn:buselim}
\bb^\lambda(s,x,t,y) = \lim_{r \to -\infty} \log \f{Z_\beta(s,x\viiva r,z_r)}{Z_\beta(t,y\viiva r,z_r)},
\ee
simultaneously for all paths    $\{z_r: r<s\wedge t\}$  that satisfy $\lim_{r \to -\infty} \f{z_r}{|r|} = \lambda$. Furthermore, the article \cite{Janj-Rass-Sepp-22} constructs the  Busemann process
\be\label{Bpr10}
\{\bb^{\lambda \sig}(s,x,t,y): (s, x,t,y)\in\R^4, \,\lambda \in \R, \,\sigg \in \{-,+\}\}
\ee
on a single event of full probability.    The sign parameter $\sigg \in \{-,+\}$ is a necessary ingredient of the construction.   In general, $\lambda \mapsto \bb^{\lambda-}(s,x,t,y)$ is left-continuous, while $\lambda \mapsto \bb^{\lambda+}(s,x,t,y)$ is right-continuous. A fixed value $\lambda\in\R$ is almost surely not a discontinuity of this process, that is, 
\be\label{fixed_not}
\Pp(\bb^{\lambda-}=\bb^{\lambda+})=1 \quad \forall \lambda\in\R.
\ee
(See Theorem \ref{thm:SHEbuse}\ref{itm:SHEBuseP0} below.)  But the existence of random  discontinuities across the uncountably many values $\lambda$  was left open  \cite[Open Problem 2]{Janj-Rass-Sepp-22}. Define the  set
of exceptional  directions at which jumps occur as  
\be \label{Lambda}
\Lambda_{\bb} := \{\lambda \in \R: \bb^{\lambda -}(s,x,t,y) \neq \bb^{\lambda +}(s,x,t,y) \text{ for some }(s,x,t,y) \in \R^4 \}.
\ee

The set $\Lambda_{\bb}$ is exactly the set of directions $\lambda$ at which  the semi-infinite Gibbs measure supported on $\lambda$-directed paths is not unique \cite[Theorems 3.35, 3.38]{Janj-Rass-Sepp-22}. The Busemann process is an \textit{eternal solution} to the KPZ equation, meaning that started from any initial time, the Busemann process evolves forward in time via the KPZ equation. When the Busemann process is discontinuous at $\lambda$,   the \textit{one force--one solution} principle fails because there are two eternal solutions to the equation satisfying the same asymptotic slope conditions. This is manifested in the dynamic programming principle proved in \cite{Janj-Rass-Sepp-22} and recorded in the present paper as Theorem \ref{thm:SHEbuse}\ref{itm:SHEBuseevolve}.  Theorem 3.5 of \cite{Janj-Rass-Sepp-22}, recorded as Theorem \ref{thm:SHEbuse}\ref{itm:SHEBusedichot} in the present paper, established the following dichotomy: either $\Pp(\Lambda_{\bb} = \varnothing) = 1$ or $\Pp(\Lambda_{\bb}\text{ is countable and dense in }\R) = 1$.  Our next theorem resolves this question. 

\begin{theorem} \label{thm:SHEjumps}
Let $\beta > 0$. Then, $\Pp(\Lambda_{\bb}\text{ is countable and dense in }\R) = 1$.
\end{theorem}

Theorem \ref{thm:SHEjumps} is equivalent to Theorem \ref{thm:jumps_every_edge}. 
In particular, we deduce from Theorem \ref{thm:Busedrift} that the $\KPZH_\beta$ is equal in law to the marginal of the Busemann process \eqref{Bpr10} obtained by fixing $s=t$ and $x=0$  (see Corollary \ref{cor:SHEBuse=KPZH}). The presence of the discontinuity set for any interval of space is derived from the earlier work \cite{Janj-Rass-Sepp-22}.

The proof of the existence of discontinuities of the $\KPZH_\beta$ comes in Corollary \ref{cor:SHEjumps}. 
Our proof exploits the explicit description of the distribution of $\KH_\beta^\lambda(y) - \KH_\beta^0(y)$ given in Theorem \ref{thm:lambda-F}. It would be interesting to see if there is a proof of the condition above that uses softer properties of Busemann functions and can be generalized to other models. However, as a counterexample, consider the deterministic approximation of the Green's function for the KPZ equation with $\beta = 1$ (see \cite[Section 1.5, Theorem 3.8]{Janj-Rass-Sepp-22})
\[
\wt {\mathcal H}(t,y \viiva s,x) = -\f{t-s}{24} - \f{(y-x)^2}{2(t-s)}.
\]
In this setting, the Busemann function is equal to
\[
\wt b^\lambda(s,x,t,y) = \lim_{r \to -\infty} \wt {\mathcal H}(t,y \viiva r,-r\lambda) -  \wt {\mathcal H}(s,x\viiva r,-r\lambda) =  \f{(12\lambda^2 - 1)(t-s)}{24} + (y-x)\lambda,
\]
which is  continuous in the parameter $\lambda$. Thus, any more general condition to prove the existence of discontinuities would need deeper information about the noise present in the model and cannot rely only on curvature or strict convexity of the shape function.

\subsection{High and low temperature limits of the KPZ horizon} 
\label{sec:SH-KPZH}
The KPZ equation interpolates between two so-called \textit{universality classes}. This phenomenon was first mathematically observed from explicit formulas calculated in \cite{Amir-Corwin-Quastel-2011} and is explicitly noted in \cite[Theorem 1.1]{Corwin-survey}. Setting $\beta = 1$ for simplicity (noting that the general equation can be obtained from this one by scaling, see \cite[Equation (3)]{Corwin-survey}), we let $Z(T,X) = Z_1(T,X \viiva 0,0)$, and set 
\[
F_T(s) = P\Bigl(\log Z(T,X) + \f{X^2}{2T} + \f{T}{24} \le s\Bigr).
\]
Theorem 1.1 and Corollary 1.2 of \cite{Corwin-survey} state that the probability above does not depend on $X$ and gives these long and short time limits:  
\be\label{acq}
\lim_{T \to \infty} F_T(2^{-1/3} T^{1/3} s) = F_{\text{GUE}}(s)\quad \text{and}\quad \lim_{T \searrow 0} F_T(2^{-1/2} \pi^{1/4} T^{1/4}(s-\log \sqrt{2\pi T})) = \Phi(s),
\ee
where $F_{\text{GUE}}$ is the Tracy-Widom GUE distribution, and $\Phi$ is the standard Gaussian distribution. The Tracy-Widom distribution is central to the KPZ universality class, while the Gaussian distribution is central to the Edwards-Wilkinson class \cite{Edwards-Wilkinson,Corwin-survey}. On the KPZ side of things, much recent work has been devoted to stronger convergence on the level of the process \cite{KPZ_equation_convergence,heat_and_landscape,Wu-23,Das-Zhu-22a,Das-Zhu-22b}. See Section \ref{sec:KPZintro} for a more detailed discussion of the relevant literature.

The scaling relations for $Z_\beta$ proved in \cite{Alberts-Khanin-Quastel-2014b,Alb-Janj-Rass-Sepp-22} (recorded in the present paper as Theorem \ref{thm:Zinvar}) imply that $Z(
T,0) = Z_1(T,0 \viiva 0,0) \deq \f{1}{\sqrt T} Z_{T^{1/4}}(1,0 \viiva 0,0)$. Hence, large times $T$ correspond to high inverse temperatures $\beta$, while short times $T$ correspond to small values of $\beta$. In this same spirit, the results of this section show that the $\KPZH_\beta$ interpolates between the jointly invariant measures in the KPZ and Edwards-Wilkinson universality classes, seen in the limits as $\beta \nearrow \infty$ and $\beta \searrow 0$, respectively.

Figure \ref{fig:KH_sim} shows a simulation (using R \cite{R}) of the $\KPZH_\beta$ for three different values of $\beta$, namely $0.1,1$, and $20$. In each case, we use the drift values $\lambda = -5,-2.5,0,2.5,5$. For small $\beta$,   the trajectories tend to look like affine shifts of one another. For large $\beta$, the trajectories appear to stick very closely together in a neighborhood of the origin.  In fact, before the limit the paths do not actually touch outside the origin,   but at $\beta=\infty$, each pair of paths coincide  in a nondegenerate interval around the origin. 
\begin{figure}
    \centering
    \includegraphics[height = 3in]{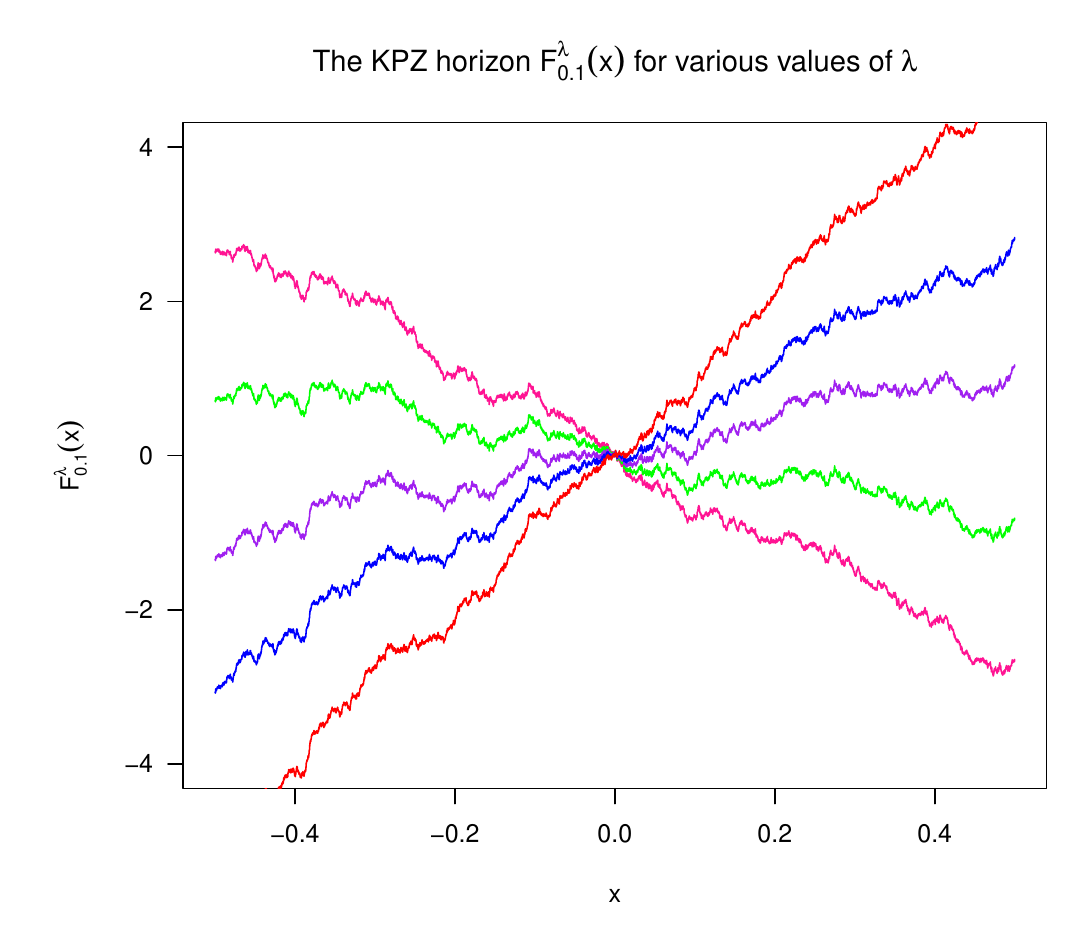}
    \includegraphics[height = 3in]{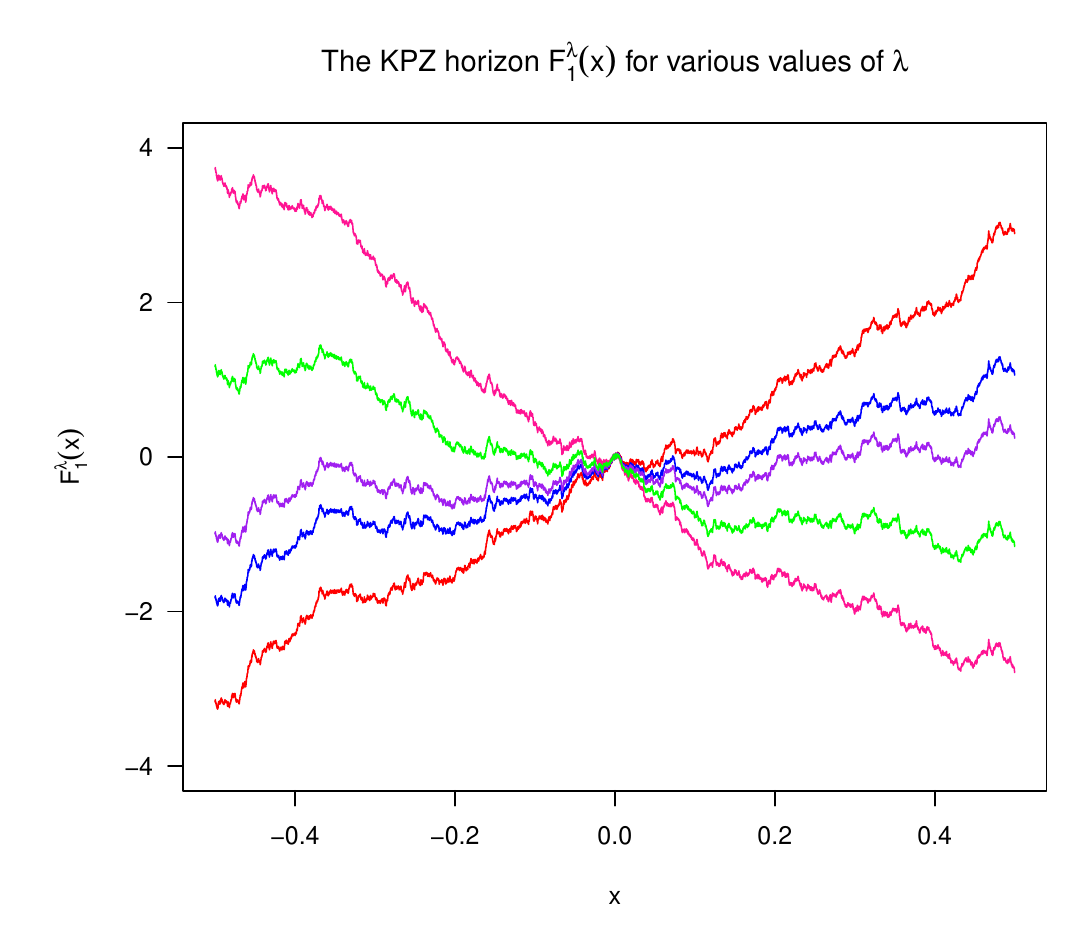}
    \includegraphics[height = 3in]{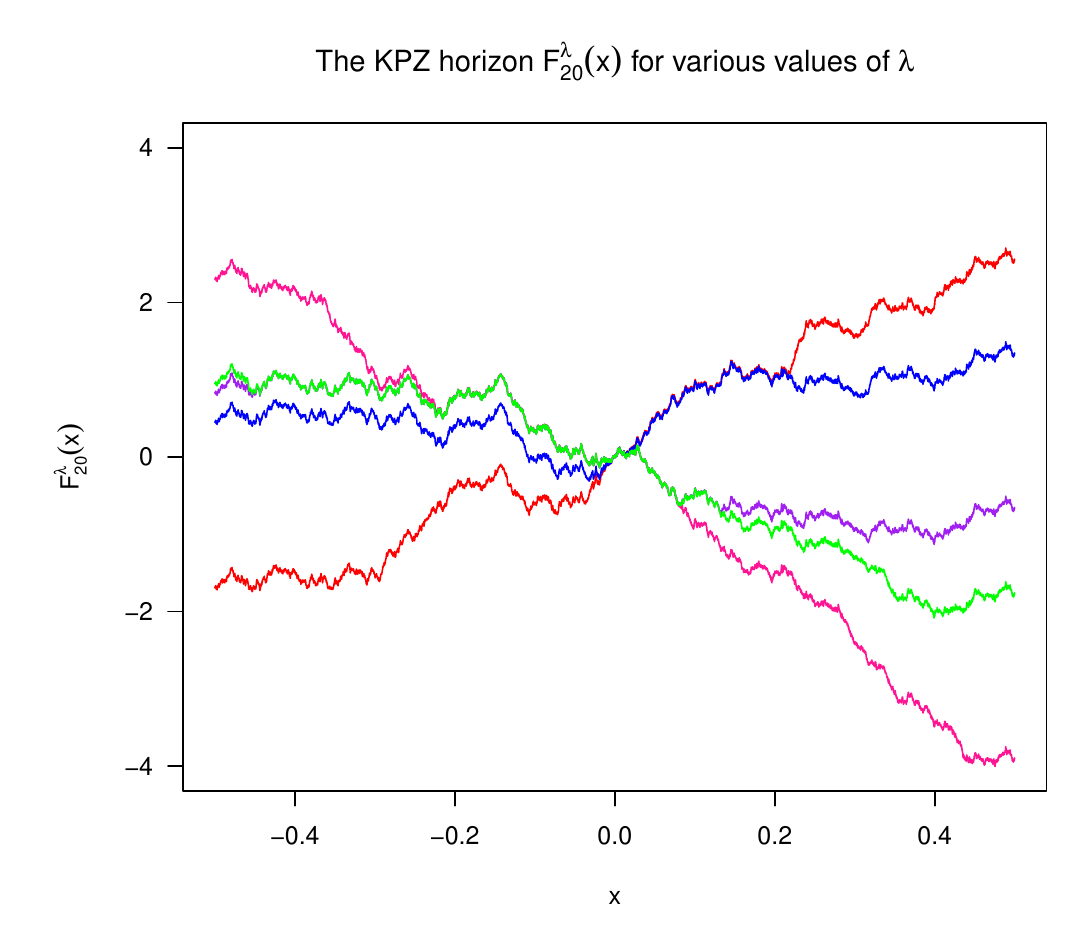}
    \caption{\small KPZH$_\beta$ for three inverse temperature values $\beta= 0.1$, $1$, and $20$ from top to bottom,  and in each frame for the drift  values   $\lambda = -5$ (pink), $\lambda = -2.5$ (green), $\lambda = 0$ (purple), $\lambda = 2.5$ (blue), and $\lambda = 5$ (red).}
    \label{fig:KH_sim}
\end{figure}

As already touched upon above, the next Theorem \ref{thm:SHlimit}  establishes the low-temperature/long-time limit of   $\KPZH_\beta$ as its zero-temperature counterpart, the \textit{stationary horizon}. 
The stationary horizon (SH) is a stochastic process $G = \{G^\lambda\}_{\lambda \in \R}$ with path space $D(\R,C(\R))$. Marginally, each $C(\R)$-valued component  $G^\lambda$ is a Brownian motion with diffusivity $\sqrt 2$ and drift $2\lambda$.  See Appendix \ref{sec:SH} for a precise definition of the SH.  
As $\beta \searrow 0$ we see a simple  limit from the EW class, as we remark below in Section \ref{sec:EW}, and one 
that is consistent with the short-time limit in \eqref{acq}.


The mode of convergence proved is weak convergence of  projections on the spaces $C(\R,\R^k)$ for $k \ge 1$. We expect that convergence on the full path space $D(\R,C(\R))$ also holds, as is proved for exponential LPP in \cite{Busani-2021} and for the TASEP speed process in \cite{Busa-Sepp-Sore-22b}. However, the topology of convergence on the space $D(\R,C(\R))$ needs to be adjusted because the set of discontinuities for the prelimiting object is not isolated in a compact window of space, as is the case in \cite{Busani-2021,Busa-Sepp-Sore-22b}. We leave the investigation of tightness on $D(\R,C(\R))$ to future work. The convergence of  parts \ref{tempscal} and \ref{KPZscal} below are equivalent by the scaling relations of Theorem \ref{thm:dist_invar}\ref{itm:KHscale} followed by the change of variable $\gamma \mapsto \gamma \beta$.  

\begin{theorem} \label{thm:SHlimit}
Let $ \{G^\lambda\}_{\lambda \in \R}$ be the {\rm SH} \rm{(}defined in Appendix \ref{sec:SH}\rm{)} and   $\{\KH_\beta^\lambda\}_{\lambda \in \R}$  the $\KPZH_\beta$.  Fix  two real  parameters $\beta > 0$ and $\alpha \in \R$.  For any finite increasing vector $\lambda_1 < \cdots < \lambda_k$, $\{G^{\lambda_i}\}_{1 \le i \le k}$ is the limit in distribution on $C(\R,\R^k)$, as $\gamma \to \infty$, of the following two processes: 
\begin{enumerate} [label={\rm(\roman*)}, ref={\rm(\roman*)}]   \itemsep=3pt 
    \item \label{tempscal} $\{\KH_\gamma^{\lambda_i}(2\,\aabullet)\}_{1 \le i \le k}$. 
    \item \label{KPZscal}$\Bigl\{\gamma^{-1} F_\beta^{\alpha +  \gamma^{-1}\lambda_i}(2\gamma^2 \, \aabullet) - 2\gamma \alpha  \aabullet\Bigr\}_{1 \le i \le k}$. 
\end{enumerate}
Furthermore, let $B$ be a standard two-sided Brownian motion {\rm(}diffusivity $1$ and zero drift{\rm)}. Then as $\gamma \searrow 0$, the processes in parts \ref{tempscal} and \ref{KPZscal} above converge in distribution, on $C(\R,\R^k)$, to 
$
\{B(2 \aabullet) + 2\lambda_i \aabullet\}_{1 \le i \le k}.
$
\end{theorem}

For large $\gamma > 0$, the scaling in Item \ref{KPZscal} above fixes a temperature $\beta$ and considers a direction perturbed from the drift $\alpha$. This is the scaling of the  initial data in the convergence of the KPZ equation to the KPZ fixed point, as in \cite{KPZ_equation_convergence,heat_and_landscape,Das-Zhu-22a,Das-Zhu-22b,Wu-23}.

Setting $\gamma = 2^{-1/3}T^{1/3}$, the sequence in part   \ref{KPZscal} becomes
\[
\Bigl\{2^{1/3} T^{-1/3}  \KH_\beta^{\alpha +  2^{1/3}T^{-1/3}\lambda_i}(2^{1/3}T^{2/3} \aabullet) - 2^{2/3}T^{1/3}\alpha  \aabullet\Bigr\}_{1 \le i \le k},
\]
which, as $T \to \infty$, demonstrates the $1:2:3$ scaling in convergence to the KPZ fixed point. There are only two scaling parameters now because we are scaling initial data, so there is no time parameter. However, we can in fact strengthen our result to a process-level convergence using the recent results of Wu \cite{Wu-23}. Let $\Ll = \{\Ll(x,s;y,t):x,y \in \R, s < t\}$ be the directed landscape (DL). For upper semicontinuous initial data $\h:\R \to \R \cup \{-\infty\}$ satisfying, for some $a,b > 0$, $\h(x) \le a + b|x|$ for all $x \in \R$ and $\h(x) > -\infty$ for some $x$, define the KPZ fixed point by 
\be \label{eqn:KPZfixed}
h_\Ll(t,y|s,\h) = \sup_{x \in \R}\{\h(x) + \Ll(x,s;y,t)\}.
\ee

Let $\kpzb(t,y\viiva s,f)$ be the solution of the KPZ equation \eqref{eqn:KPZ} started at time $s$ with initial data $f$:
\be \label{eqn:KPZeqn}
\kpzb(t,y\viiva s,f) = \f{1}{\beta}\log \int_\R e^{\beta f(x)} Z_\beta(t,y\viiva s,x)\,dx. 
\ee

\begin{corollary} \label{cor:KPZeqtofp}
    Let $\beta > 0$, $\alpha,s \in \R$, and $\lambda_1 < \cdots < \lambda_k$. Then, as processes in $C(\R_{> s} \times \R,\R^k)$ equipped with the uniform-on-compacts  topology, 
    \begin{align*}
    &\Biggl\{2^{1/3} T^{-1/3} \Biggl[\beta h_{Z_{\beta}}\Bigl( \f{T t}{\beta^4} ,\f{2^{1/3} T^{2/3} y}{\beta^2} \Bviiva \f{Ts}{\beta^4}, F_\beta^{\alpha + 2^{1/3} T^{-1/3}\lambda_i}(\aabullet) - \alpha \aabullet\Bigr) + \f{T(t - s)}{24}\Biggr]: \\
    &\qquad\qquad  (t,y) \in \R_{>s} \times \R \Biggr \}_{1 \le i \le k}  
    \quad \overset{T \to \infty}{\Longrightarrow}  \quad
    \{h_{\Ll}(t,y \viiva s,G^{\lambda_i}):(t,y) \in \R_{>s} \times \R\}_{1 \le i \le k}.
    \end{align*}
    Furthermore, let $\{\mathcal B^{\lambda \sig}: \lambda \in \R, \sigg \in \{-,+\}\}$ be the Busemann process for the DL discussed in Appendix \ref{sec:SH}. Then, for any $\beta > 0$ and $\lambda_1 < \cdots < \lambda_k$, as processes in $C(\R^4,\R^k)$ equipped with the uniform-on-compacts topology,
    \begin{align*}
    &\Biggl\{2^{1/3}T^{-1/3}\Bigl[b_{\beta}^{\beta 2^{1/3} T^{-1/3} \lambda_i}\Bigl(\f{Ts}{\beta^4},\f{2^{1/3} T^{2/3} x}{\beta^2}, \f{Tt}{\beta^4}, \f{2^{1/3} T^{2/3} y}{\beta^2}   \Bigr) + \f{T(t - s)}{24}\Bigr]: (x,s;y,t) \in \R^4 
   \Biggr\}_{1 \le i \le k}  \\
   &\qquad\qquad\qquad\qquad \overset{T \to \infty}{\Longrightarrow} \{\mathcal B^{\lambda_i}(y,-t;x,-s):(x,s;y,t) \in \R^4 \}_{1 \le i \le k}.
    \end{align*}
\end{corollary}
The temporal reflection in the process $\{\mathcal B^{\lambda_i}(y,-t;x,-s):(x,s;y,t) \in \R^4 \}_{1 \le i \le k}$ is a manifestation of the fact that in \cite{Janj-Rass-Sepp-22}, the infinite paths travel south, while the infinite geodesics in \cite{Rahman-Virag-21} and \cite{Busa-Sepp-Sore-22a} travel north.

\subsubsection{Jointly invariant measures for the Edwards-Wilkinson fixed point} 
\label{sec:EW}
In contrast with the $\gamma \to \infty$ limit to the SH and in light of the $\gamma \searrow 0$ limit in Theorem \ref{thm:SHlimit}, it is natural to ask whether $\{B(\aabullet) + \lambda_i \aabullet\}_{1 \le i \le k}$ is a jointly invariant measure for the Edwards-Wilkinson fixed point \cite{Edwards-Wilkinson,Corwin-survey}.  The Edwards-Wilkinson fixed point is governed by the $1$-dimensional additive stochastic heat equation $\partial_t u = \frac12 u_{xx} + W$. It is well-known that this equation, started from initial data $f$ at time $0$, is solved as 
\be \label{SHEadd}
u(t,x \viiva f) = \int_\R \rho(t,x - y) f(y)\,dy + \int_0^t \int_{\R}  \rho(t - s,x - y)  W(ds\,dy).
\ee
It is also well-known that the increments of two-sided Brownian motion $B$ are invariant in time for $u$. That is, $u(t,\aabullet;B) - u(t,0;B) \deq B$. From \eqref{SHEadd}
it follows that, for any appropriate function $f:\R \to \R$ and $\lambda \in \R$,
\[
u(t,x \viiva f(\aabullet) + \lambda \aabullet) = u(t,x \viiva f) + \lambda x.
\]
Hence, in the sense of \eqref{joint_invar}, $\{B(\aabullet) + \lambda_1 \aabullet,\ldots, B(\aabullet) + \lambda_k \aabullet\}$ is a jointly invariant measure for the SHE with additive noise, where the common noise $W$ drives the equation from the different initial conditions. Indeed, this can be expected from Theorems \ref{thm:Busedrift} and \ref{thm:SHlimit}, as the $\beta \searrow 0$ limit of \eqref{eqn:KPZ} is precisely the additive SHE.

\subsection{Methods and related literature}

\subsubsection{The O'Connell-Yor polymer and intertwining}
The  proof of Theorem \ref{thm:Busedrift} comes from first showing that the $\KPZH_\beta$ describes jointly invariant measures for the semi-discrete  O'Connell-Yor (OCY) polymer  introduced in \cite{brownian_queues}. We show that the $\KPZH_\beta$ satisfies certain distributional invariances under scaling to initial data for the SHE. Then, we show that the $\KPZH_\beta$ is jointly invariant for the SHE and use a uniqueness result from \cite{Janj-Rass-Sepp-22} to conclude the proof. 

The proof of invariance of the KPZH for the O'Connell-Yor polymer comes from an intertwining argument that was originally developed for TASEP and the Hammersley process in \cite{Ferrari-Martin-2007,Ferrari-Martin-2005,Ferrari-Martin-2009}. The main idea is to find the invariant measure for a different Markov chain with a more tractable invariant distribution, then prove that this simpler process intertwines with the process of interest via mappings from queuing theory.  Since then, \cite{CGM_Joint_Buse} adapted this method to discrete last-passage percolation, \cite{Seppalainen-Sorensen-21b} extended this to the semi-discrete model of Brownian last-passage percolation, and \cite{Bates-Fan-Seppalainen} extended this to the positive temperature inverse-gamma polymer. The present paper is the first to extend this method to a positive temperature semi-discrete model. While one can take limits of the positive temperature model to get the zero temperature model, the opposite direction is not a straightforward task. The key inputs needed to complete the intertwining argument in this setting are found in Section \ref{sec:queue}. Other details of the argument that bear close resemblance to previous work are relegated to Appendix \ref{appx:queue}. 

The convergence step to the SHE requires a substantial amount of nontrivial work. Convergence of the OCY polymer (with the initial point fixed) to narrow wedge solutions of the SHE was established in the sense of finite-dimensional distributions by Nica \cite{Nica-2021}. Also relevant is the work of Jara and Moreno-Flores \cite{Jara-MorenoFlores-20}, which showed convergence of the discrete gradients from the O'Connell-Yor polymer to the stochastic Burgers equation. In Section \ref{sec:OCYSHE} of the present paper, we prove in full detail, using different methods than those in \cite{Nica-2021}, the convergence of the four-parameter field of the OCY polymer to the Green's function of the SHE (in the sense of finite-dimensional distributions) and prove convergence of solutions from appropriate initial data.

Similar items to  Lemmas \ref{lem:Yrep}, \ref{lem:moment_bd}, \ref{lem:L1conv_bd}, and Theorem \ref{thm:OCYtoSHE} in Section \ref{sec:SHE} appeared in an unfinished manuscript
of Moreno Flores, Quastel, and Remenik. As no proofs for the precise results we need appear in the literature, we provide them in Section \ref{sec:SHE}. We develop several new ideas to complete the technical details of these results. This work in Section \ref{sec:SHE} may have independent interest.

\subsubsection{Discontinuities of the Busemann process and one force--one solution principle}
At this point it is reasonable to expect that discontinuities appear universally  in the Busemann processes of 1+1 dimensional KPZ models on noncompact spaces. As evidence for this, a dense set of discontinuities has been established in discrete, semi-discrete and fully continuous settings, in both positive and zero temperature, and in the putative universal limit (DL). An overview of this recent development appears later in this section.  

  The \textit{one force--one solution principle} (1F1S) states that, for a given realization of the driving noise and a given value of the conserved quantity in a  stochastically forced conservation law, there is a unique eternal solution that is measurable with respect to the history of the noise.  That the  failure of 1F1S is associated with the discontinuities  of the Busemann process has now been observed in both discrete and continuous settings \cite{Bates-Fan-Seppalainen, janj-rass-19-1f1s, Janjigian-Rassoul-Seppalainen-19, Janj-Rass-Sepp-22}.

 Busemann functions and the one force--one solution principle have been studied in the past for the Burgers equation with discrete random forcing, both in compact and noncompact settings \cite{Sinai-1991,GIKP-2005,Iturriaga-Khanin-2003,Bakhtin-2007,Kifer-1997,Dirr-Souganidis-2005,Kifer-1997,Bakhtin-2013,Bakhtin-Cator-Konstantin-2014,Bakhtin-16chapter,Bakhtin-16,Bakhtin-Khanin-18,Bakhtin-Li-18,Bakhtin-Li-19,Hoang-Khanin-2003,Drivats-2022,Dunlap-Graham-Ryzhik-21}. Specifically, in the works of Bakhtin and coauthors \cite{Bakhtin-16,Bakhtin-Cator-Konstantin-2014,Bakhtin-16chapter,Bakhtin-2013,Bakhtin-Li-18,Bakhtin-Li-19}, one sees analogous results for Busemann functions and semi-infinite geodesics, which are the zero-temperature counterpart of semi-infinite polymer measures. The failure of 1F1S did not arise in this earlier work because the focus was on    a fixed, nonrandom value of the conserved quantity.  As mentioned above in  \eqref{fixed_not}, there are no fixed discontinuities.  This is a general fact about polymer models with differentiable  limit shapes. Jointly invariant measures for the Burgers equation with smooth forcing have recently been studied in \cite{Dunlap-Graham-Ryzhik-21,Dunlap-Ryzhik-2020}. 

We give a  brief summary of the history. 
The first observation of random discontinuities of the Busemann process was completed by Fan and the third author \cite{CGM_Joint_Buse} for the exactly solvable exponential corner growth model. Across a single horizontal edge, they showed that the Busemann process, indexed by the direction, can be described by a compound Poisson process. Across all edges, the union of the discontinuities is countably infinite and dense. This result was used in \cite{Janjigian-Rassoul-Seppalainen-19} to characterize the set of directions with non-unique semi-infinite geodesics as the same as the set of discontinuities of the Busemann process. 

Similar studies were carried out for Brownian last-passage percolation (BLPP) by the third and fourth authors \cite{Seppalainen-Sorensen-21b} and for the directed landscape (DL) \cite{Busa-Sepp-Sore-22b} by Busani and the third and fourth authors. Here, the characterization of exceptional directions of semi-infinite geodesics is exactly analogous to that of \cite{Janjigian-Rassoul-Seppalainen-19}, but additional non-uniqueness of initial segments of geodesics appears due to the semi-discrete setting in these models. As a result, new methods of proof were developed to achieve these results. The studies \cite{Seppalainen-Sorensen-21b,Busa-Sepp-Sore-22a} used a description of the Busemann process for Brownian LPP developed in \cite{Seppalainen-Sorensen-21b} and the remarkable fact that the Busemann process along a horizontal line for BLPP agrees with that of the DL. Unlike in the exponential corner growth model, the drift-indexed Busemann process of BLPP along a single horizontal interval is not a compound Poisson process, nor does it have independent increments. Thus, obtaining a full  description of this process remains an open problem. However, the characterization  of the process in terms of coupled Brownian motions permits some distributional calculations,  enough to show that the drift-indexed  Busemann process along an interval  is a step function. 

The recent work \cite{Bates-Fan-Seppalainen} studies the Busemann process for the inverse-gamma polymer and discovers a similar explicit description as in \cite{CGM_Joint_Buse}. However,  the jumps of the Busemann process across a single horizontal interval are now dense, unlike in the zero temperature case where they are isolated. Likewise, the paper \cite{Janj-Rass-Sepp-22} showed that, for the KPZ equation, if the set \eqref{Lambda} is nonempty, the jumps are present along each horizontal interval and are therefore dense. In the present work, we obtain a description of the Busemann process for the SHE in terms of coupled Brownian motions with drift. But, just as in the zero-temperature cases of BLPP and the DL, the drift-indexed process along a horizontal interval does not have an explicit description that we know. However, we can compute the distribution of an increment of this process.  Then in Corollary \ref{cor:SHEjumps}, we apply a novel condition developed in Lemma \ref{lem:jumpscond} to show the existence of jumps. Our work demonstrates that the corresponding phenomenon in  \cite{Bates-Fan-Seppalainen} is \textit{not} simply a manifestation of discrete lattice effects. 

After the first version of this paper was posted, Dunlap and Gu \cite{Dunlap-Gu-23} studied jointly invariant measures for the KPZ/Burgers equation on the torus. Without obtaining an explicit description of these measures, they showed that in this compact setting, there are no random discontinuities in the drift parameter.

\subsubsection{Stationary horizon and KPZ universality} \label{sec:KPZintro} SH was first constructed by Busani \cite{Busani-2021} as the scaling limit of the Busemann process of the exponential corner growth model.  \cite{Busani-2021} conjectured SH to be the universal scaling limit of Busemann processes of models in the KPZ universality class. Shortly afterwards, SH was independently discovered in the context of Brownian last-passage percolation by the third and fourth authors \cite{Seppalainen-Sorensen-21b}. 
A brief introduction to the SH is given in Appendix \ref{sec:SH}. 

In \cite{Busa-Sepp-Sore-22a,Busa-Sepp-Sore-22b,Busa-Sepp-Sore-23}, the third and fourth authors, together with Busani, studied the role of SH in the KPZ class and established further evidence of its universality:   

(i)  Given appropriate conditions on the asymptotic slope of the initial data, the SH is the unique multi-type stationary distribution of the KPZ fixed point \eqref{eqn:KPZfixed} that  evolves in the  environment given by the directed landscape.

(ii) As a consequence, the SH gives the distribution of the fixed-time-level Busemann process of the directed landscape.  In this representation, the parameter $\lambda$ corresponds to the space-time slope of semi-infinite geodesics.  

(iii) The suitably scaled TASEP speed process introduced by \cite{Amir_Angel_Valko11} converges to the SH. In the limit, $\lambda$ represents  the scaled and centered values of the speed process. This suggests that  SH is a general scaling limit of   multitype invariant distributions,  beyond the  Busemann functions of stochastic growth models. 

(iv) A framework is given in the work \cite{Busa-Sepp-Sore-23} to show convergence to the SH under conditions that are expected to hold in great generality. These conditions are (a) convergence of the point-to-point LPP process to the DL, (b) marginal convergence of individual Busemann functions to Brownian motion with drift, and (c) tightness of exit points on the scale $N^{2/3}$ under stationary boundary conditions. As an application of the general framework,  it is shown that the Busemann process of six solvable LPP models converge to the SH in the sense of finite-dimensional distributions. 

The high-level analogy  between $\KPZH_\beta$ and SH is that they both describe  unique jointly invariant distributions,  $\KPZH_\beta$ for the KPZ equation and SH for the KPZ fixed point. Additionally, $\KPZH_\beta$ and SH share certain  properties. Both are couplings of Brownian motions with drift  whose increments are ordered. Both are translation-invariant and have a reflection symmetry (Theorem \ref{thm:dist_invar}\ref{itm:shinv} and \ref{itm:reflinv}). However, the two processes are not the same in law. One way to see this is Theorem \ref{thm:jumps_every_edge}). While the full SH process $\lambda \mapsto G^\lambda\in C(\R)$ has a dense set of discontinuities $\lambda\in\R$,   for  given $x < y$  the points of discontinuity of the restricted process  $\lambda \mapsto G^\lambda(y)-G^\lambda(x)$ are isolated. In contrast, for \textit{any} $x < y$, the process $\lambda \mapsto \KH_\beta^\lambda(y) - \KH_\beta^\lambda(x)$ contains the full countable dense set of discontinuities of the process $\lambda \mapsto \KH_\beta^\lambda \in C(\R)$. 

There has been much recent work on the   convergence of the KPZ equation to the KPZ fixed point. This was first accomplished in two independent works of Quastel and Sarkar \cite{KPZ_equation_convergence} and  Vir\'ag \cite{heat_and_landscape}. Recently, Wu \cite{Wu-23} proved that the Green's function of the KPZ equation converges to the directed landscape. Combined with the previous work of Das and Zhu \cite{Das-Zhu-22a,Das-Zhu-22b}, who showed localization of polymer path measures in the CDRP, this establishes that the annealed polymer measures of the CDRP converge in distribution to the geodesics of the DL (See \cite[Theorem 1.9]{Das-Zhu-22b}).

\subsection{Organization of the paper} 
Section \ref{sec:KPZHbig} constructs the KPZ horizon. The  mappings that define the projections of this process onto $C(\R,\R^k)$ are developed in Section \ref{sec:queue}. In Section \ref{sec:KPZH_cons_sec}, we construct the KPZH as a process of Brownian motions indexed by the drift $\lambda \in \R$. The remaining subsections of Section \ref{sec:KPZHbig} establish properties of this process, including the proof of Theorem \ref{thm:lambda-F} in Section \ref{sec: x_Desc}. In Section \ref{sec:disc}, we show the existence of discontinuities in the $\lambda$ parameter.  The main technical Section \ref{sec:SHE} begins with background on the stochastic heat equation from \cite{Alberts-Khanin-Quastel-2014a,Alb-Janj-Rass-Sepp-22,Janj-Rass-Sepp-22}. Then we prove the  weak-noise limit   of the O'Connell-Yor polymer to the stochastic heat equation.  The paper culminates in the proofs of the main theorems in Section \ref{sec:proofs}, except for Theorem \ref{thm:lambda-F}, which is proved earlier.  The invariance of KPZH under the KPZ equation is established through the limit from OCY to SHE. A uniqueness theorem for invariant distributions completes the characterization of the Busemann process as KPZH.  Appendix \ref{appx:queue} contains additional technical proofs for the queuing mappings. Appendix \ref{sec:SH} contains the necessary background information for the stationary horizon.

\subsection{Notation and conventions}
\begin{itemize}
    \item $\CRpin$ denotes the space of continuous functions $f:\R\to \R$ such that  $f(0) = 0$. 
    \item Increments of a single-variable function $F$ are denoted by $F(x,y) = F(y) - F(x)$.  Increment ordering between functions $f,g:\R \to \R$:    $f \li g$ if $f(x,y) \le g(x,y)$ for all $x < y$, and $f \sli g$ if $f(x,y)< g(x,y)$ for all $x < y$. 
    \item For random variables $X$ and $Y$ and probability measures $\mu$, $X\deq Y$ and $X\sim Y$ both mean that $X$ and $Y$ are equal in distribution, and  $X\sim\mu$ means that $X$ has probability distribution $\mu$.   
    \item Random variable $X$ has the gamma distribution with shape parameter $\alpha>0$  and rate $\beta>0$, abbreviated $X\sim$ Gamma$(\alpha,\beta)$, if $X$ has density function $f(x)=\Gamma(\alpha)^{-1} \beta^\alpha x^{\alpha-1}e^{-\beta x}$ on $\R_+$. 
    \item A two-sided standard Brownian motion is a continuous random process $\{B(x): x \in \R\}$ such that $B(0) = 0$ almost surely and   $\{B(x):x \ge 0\}$ and $\{B(-x):x \ge 0\}$ are two independent standard Brownian motions on $[0,\infty)$.
    \item\label{def:2BMcmu} If $B$ is a two-sided standard Brownian motion, then 
    $\{c B(x) + \mu x: x \in \R\}$ is a two-sided Brownian motion with diffusivity $c>0$ and drift $\mu\in\R$. 
    \item The complementary error function $\erfc$ is defined as 
$
\erfc(x) = \f{2}{\sqrt \pi} \int_{x}^\infty e^{-u^2}\,du.
$
\item The heat kernel is $\rho(t,x) = \f{1}{\sqrt{2\pi t}} e^{-\f{x^2}{2t}} \ind_{t > 0}$ for $(t,x)\in\R^2$.  
\item Ranges of indices in vectors and sequences are abbreviated as in $x_{m:n}=(x_m, x_{m+1}, \dotsc, x_n)$. 

\item The domain of pairs of space-time points with strictly ordered times is $\Rup=\{(s,x,t,y) \in\R^4: s<t\}$. 

\item In a $C(\R)$-valued stochastic process $\lambda\mapsto Y^\lambda(\aabullet)$, the bullet marks the missing real variable: $Y^\lambda(\aabullet)=(x\mapsto Y^\lambda(x))\in C(\R)$. 

\item Coordinatewise order on $\R^2$: $(x,y)\le(a,b)$ means that $x\le a$ and $y\le b$. 

\end{itemize}

\subsection{Acknowledgements and Funding} E.\ Sorensen wishes to thank Tom Alberts, Ofer Busani, Ivan Corwin, Sayan Das, Yu Gu, Chris Janjigian, Mihai Nica, and Xuan Wu for helpful pointers to the literature and insightful discussions. S.\ Groathouse and F.\ Rassoul-Agha were partially supported by National Science Foundation grants DMS-1811090 and
DMS-2054630. F.\ Rassoul-Agha was partially supported by MPS-Simons Fellowship grant 823136.
 T.\ Sepp\"al\"ainen  was partially supported by National Science Foundation grants DMS-1854619 and DMS-2152362, by Simons Foundation grant 1019133, and by the Wisconsin Alumni Research Foundation. E.\ Sorensen was partially supported by the Fernholz foundation. This work was partly performed while E.\ Sorensen was a Ph.D.\ student at the University of Wisconsin--Madison, where he was partially supported by T.\ Sepp\"al\"ainen under National Science Foundation grants DMS-1854619 and DMS-2152362.

\section{Construction and properties of the KPZ horizon} \label{sec:KPZHbig}

\subsection{Mappings defining finite-dimensional distributions} \label{sec:queue}
Let $\CRpin$ denote the space of continuous functions $f:\R \to \R$ satisfying $f(0) = 0$.
For $Y,B \in \CRpin$ satisfying
\be\label{Q400}
\limsup_{x \rightarrow -\infty} \f{Y(x) - B(x)}{x} > 0,
\ee
and for $\beta > 0$, define the following transformations: 
\be \label{QPDpRp}
\begin{aligned}
    \Qp_\beta(B,Y)(y) &= {\beta}^{-1}\log \int_{-\infty}^y \exp\left(\beta( B(x,y) - Y(x,y))\right)\,dx \\
    \Dp_\beta(B,Y)(y) &=  Y(y) + \Qp_\beta(B,Y)(y) -  \Qp_\beta(B,Y)(0), \\[3pt] 
    \Rp_\beta(B,Y)(y) &= B(y) + \Qp_\beta(B,Y)(0) - \Qp_\beta(B,Y)(y).
    \end{aligned}
\ee
Iterate the mapping $\Dp_\beta$ as follows:
\begin{align}
    \Dp^{(1)}_\beta(Y) &= Y, \quad\text{and}\quad
    \Dp^{(n)}_\beta(Y^1,Y^{2},\ldots,Y^n) = \Dp_\beta(Y^1, \Dp_\beta^{(n - 1)}(Y^2,\ldots,Y^{n}))\quad \text{ for } n \geq 2. \label{hat D iterated}
\end{align}
Given a Borel subset $A \subseteq \R$, we define three state spaces of $n$-tuples of functions. 
\be \label{Andef}
\A_n^A:= \Big\{\mathbf Y = (Y^1,\ldots, Y^n) \in \CRpin^n: 
\text{ for }1 \le i \le n,\;\;\lim_{x \rightarrow -\infty} \frac{Y^i(x)}{x} \text{ exists and lies in }A  \Big\}.
\ee
Note that if the components of $\mathbf Z \in \CRpin^n$ are Brownian motions with drifts in $A$, then $\mathbf Z \in \A_n^A$ almost surely. 
Next, set 
\be\label{Yndef}\begin{aligned} 
\ \;  \Y_n^A := \Bigg\{\mathbf Y = (Y^1,\ldots, Y^n) \in \CRpin^n: \text{ for } 1 \le i \le n, \lim_{x\rightarrow -\infty} \frac{Y^i(x)}{x} \text{ exists and lies in }A,  \\
\qquad\qquad\text{and for } 2 \leq i \leq n , \lim_{x \rightarrow -\infty} \frac{Y^i(x)}{x} > \lim_{x \rightarrow -\infty} \frac{Y^{i - 1}(x)}{x}    \Bigg\}
\end{aligned}\ee
and
\begin{align} \label{Xndef}
\X_n^A := \Bigg\{\boldsymbol\eta = (\eta^1,\ldots,\eta^n) \in \Y_n^A:  \eta^i \sgi \eta^{i - 1}  \text{ for } 2 \leq i \leq n 
\Bigg\}.
\end{align}
The most common choices for $A$ will be $\R_{>0}$ (to be used for the state space of invariant measures in the O'Connell-Yor polymer) and $\R$ (to be used as the state space of invariant measures in the KPZ equation). 
Section 7 of \cite{Seppalainen-Sorensen-21b} shows that these state spaces are Borel measurable subsets of the space $C(\R,\R^n)$.

Next, define a transformation  $\sDp_\beta^{(n)}$ on $n$-tuples of functions as follows. Let $A\subseteq \R$. For $\mathbf Y  = (Y^1,\ldots,Y^n)\in \Y_n^A$, the image $\boldsymbol \eta = (\eta^1,\ldots, \eta^n)=\sDp_\beta^{(n)}(\mbf Y) \in \X_n^A$ is defined by 
\begin{equation} \label{definition of script hat D}
\eta^i = \Dp^{(i)}_\beta(Y^1,\ldots,Y^i) \quad \text{ for } 1 \le i \le n.
\end{equation}
Lemma \ref{sDppres} below proves that $\sDp_\beta^{(n)}: \Y_n^A \to \X_n^A$.

For a finite increasing real vector $\bar \lambda = (\lambda_1 < \lambda_2 < \cdots < \lambda_n)$, define the measure $\nu^{\bar \lambda}$ on $\Y_n^\R$ as follows: $(Y^1,\ldots,Y^n) \sim \nu^{\bar \lambda}$ if $Y^1,\ldots,Y^n$ are mutually independent and $Y^i$ is a Brownian motion with drift $\lambda_i$. Define the measure $\mu_\beta^{\bar \lambda}$ on $\X_n^\R$ as
\be \label{mubeta}
\mu^{\bar \lambda}_\beta = \nu^{\bar \lambda} \circ (\sDp^{(n)}_\beta)^{-1}.
\ee
This is the key definition of the section.  
In each application of \eqref{definition of script hat D} the drifts satisfy $\lambda_1<\dotsm<\lambda_i$ and so the mappings are well-defined. 

We prove a series of lemmas about these measures. These measures and their properties have analogues in zero temperature (see \cite{Busani-2021,Seppalainen-Sorensen-21b,Busa-Sepp-Sore-22a}), but their extensions to positive temperature require a different perspective and the proofs are different.  The first result below derives a formula for $\Dp^{(n)}_\beta(Y^1,\ldots,Y^n)$. Once the first properties of the mappings and measures are established, some proofs go through just as they do for zero temperature in \cite{Seppalainen-Sorensen-21b}. For such results, we provide the full details in Appendix \ref{appx:queue}. 
\begin{lemma} \label{lem:pQalt}
Let $Y^1,\ldots,Y^n \in \CRpin$ be such that all the following integrals are finite. Then, for $n \ge 2$ and $\beta > 0$,
\be \label{Did}
\begin{aligned}
\exp\bigl(\beta \Dp_\beta^{(n)}(Y^1,\ldots,Y^n)(y)\bigr) 
&=e^{\beta Y^1(y)} \cdot 
\f{
\ddd\int\limits_{-\infty < x_{n - 1} < \cdots <x_{1} < y} \prod_{i = 1}^{n-1} e^{\beta (Y^{i + 1}(x_i) - Y^i(x_i))} d x_{ i}}{\ddd\int\limits_{-\infty < x_{n - 1} < \cdots <x_{1} < 0} \prod_{i = 1}^{n-1} e^{\beta( Y^{i + 1}(x_i) - Y^i(x_i)) } d x_{ i}}.
\end{aligned}
\ee
Furthermore,
\be \label{Rid}
\exp\bigl(\beta R_\beta(Y^1,Y^2)(y)\bigr) = \f{e^{\beta Y^2(y)}\int_{-\infty}^0 \exp\bigl[\beta(Y^2(x) - Y^1(x)\bigr]\,dx }{\int_{-\infty}^y \exp\bigl[\beta(Y^2(x) - Y^1(x)\bigr]\,dx}.
\ee
\end{lemma}
\begin{proof}
We prove this by induction on $n$. We start with the base case $n = 2$. From \eqref{QPDpRp},
\be\label{Did2}\begin{aligned}
\exp\bigl(\beta \Dp_\beta(Y^1,Y^2)(y)\bigr) &= \f{e^{\beta Y^2(y)}\int_{-\infty}^y \exp\bigl(\beta(Y^1(x,y) - Y^2(x,y)) \bigr)\,dx }{\int_{-\infty}^0 \exp\bigl(\beta(Y^1(x,0) - Y^2(x,0)) \bigr)\,dx} \\[0.5em]
&= \f{e^{\beta Y^1(y)}\int_{-\infty}^y \exp\bigl(\beta(Y^2(x) - Y^1(x)) \bigr)\,dx }{\int_{-\infty}^0 \exp\bigl(\beta(Y^2(x) - Y^1(x)) \bigr)\,dx}.
\end{aligned}\ee
The proof of \eqref{Rid} is analogous.
Now, assume that \eqref{Did} holds for $n \ge 2$. Then,
\begin{align*}
&\quad \exp\bigl(\beta \Dp_\beta^{(n)}(Y^1,\ldots,Y^n)(y)\bigr) \\[0.5em]
&= \exp\bigl(\beta \Dp_\beta(Y^1,\Dp_\beta^{(n - 1)}(Y^2,\ldots,Y^n))(y)\bigr) \\[0.5em]
&=\f{e^{\beta Y^1(y)}\int_{-\infty}^y \exp\bigl(\beta(\Dp_\beta^{(n - 1)}(Y^2,\ldots,Y^n)(x_{1}) - Y^1(x_1)\bigr)dx_1 }{\int_{-\infty}^0 \exp\bigl(\beta(\Dp_\beta^{(n - 1)}(Y^2,\ldots,Y^n)(x_1) - Y^1(x_1)\bigr) dx_1} \\[0.5em]
&= \f{e^{\beta Y^1(y)}\int_{-\infty}^y\bigl( e^{\beta Y^2(x_1)} \int_{-\infty < x_{n - 1} < \cdots < x_2 < x_1} \prod_{i = 2}^{n-1}\exp\bigl[\beta (Y^{i+1}(x_i) - Y^i(x_i))\bigr]\,dx_i\bigr) e^{- \beta Y^1(x_1)} dx_1 }{\int_{-\infty}^0\bigl( e^{\beta Y^2(x_1)} \int_{-\infty < x_{n - 1} < \cdots < x_2 < x_1} \prod_{i = 2}^{n-1}\exp\bigl[\beta (Y^{i+1}(x_i) - Y^i(x_i))\bigr]\,dx_i\bigr) e^{- \beta Y^1(x_1)} dx_1 } \\[0.5em]
&= \f{e^{\beta Y^1(y)} \int_{-\infty < x_{n - 1} < \cdots <x_{1} < y} \prod_{i = 1}^{n-1} \exp\bigl[\beta(Y^{i + 1}(x_i) - Y^i(x_i)) \bigr] dx_i }{\int_{-\infty < x_{n - 1} < \cdots <x_{1} < 0} \prod_{i = 1}^{n-1} \exp\bigl[\beta(Y^{i + 1}(x_i) - Y^i(x_i)) \bigr] dx_i}.
\end{align*}
The first equality used the definition of $\Dp^{(n)}$, the second the $n = 2$ case, and in the third the induction assumption. In the third equality,  an integral over the set $\{-\infty < x_{n - 1} < \cdots < x_2 < 0\}$ was cancelled from the numerator and the denominator. 
\end{proof} 

\begin{lemma} \label{lem:OCYlim}
Assume that $(B,Y) \in \Y_2^\R$ with
\[
\lim_{x \to -\infty} \f{B(x)}{x} = a < b = \lim_{x \to -\infty} \f{Y(x)}{x}.
\]
Then,
\[
\lim_{x \to -\infty} \f{\Rp_\beta(B,Y)(x)}{x} = a,\qquad\text{and}\qquad \lim_{x \to -\infty} \f{\Dp_\beta(B,Y)(x)}{x} = b.
\]
\end{lemma}
\begin{proof}
By Lemma \ref{lem:pQalt}, it suffices to show that 
\[
\lim_{y \to -\infty} \f{1}{y} \log \int_{-\infty}^y e^{\beta(Y(x) - B(x))}\,dx = \beta(b - a).
\]
Fix $\ve > 0$, and let $y < 0$ be such that $x\ve < B(x) - ax < -x\ve$  and $x\ve < Y(x) - bx < -x\ve$ for all $x < y$. Then, for such $y$,
\[
\f{e^{\beta(b - a+2\ve)y}}{\beta(b -a + 2\ve)} = \int_{-\infty}^y e^{\beta(b - a + 2\ve )x}\,dx \le \int_{-\infty}^y e^{\beta(Y(x) - B(x))}\,dx \le \int_{-\infty}^y e^{\beta(b - a - 2\ve )x}\,dx \le \f{e^{\beta(b - a - 2\ve)y}}{\beta(b -a - 2\ve)}.
\]
Taking the log of all sides and dividing by $y$ yields
\[
\beta(b - a - 2\ve) \le \liminf_{y \to -\infty} \f{1}{y} \log \int_{-\infty}^y e^{\beta(Y(x) - B(x))}\,dx  \le \limsup_{y \to -\infty} \f{1}{y} \log \int_{-\infty}^y e^{\beta(Y(x) - B(x))}\,dx \le \beta(b - a + 2\ve).
\]
Sending $\ve \searrow 0$ completes the proof.
\end{proof}

\begin{lemma} \label{lem:OCYmont}
Let $(B,Y),(B,Y') \in \Y_2^\R$ be such that $Y \li Y'$. Then, $B \sli \Dp_\beta(B,Y) \li \Dp_\beta(B,Y')$. If $Y \sli Y'$, then $\Dp_\beta(B,Y) \sli \Dp_\beta(B,Y')$ as well.
\end{lemma}
\begin{proof}
Let $x < y$. We use Lemma \ref{lem:pQalt} to write
\begin{align*}
&\quad \, \exp[\beta \Dp_\beta(B,Y)(x,y)] = e^{\beta B(x,y)}  \f{\int_{-\infty}^y e^{\beta(Y(z) - B(z))}\,dz}{\int_{-\infty}^x e^{\beta(Y(z) - B(z))}\,dz}  \\
&=  e^{\beta B(x,y)}\Biggl(1 + \f{\int_x^y e^{\beta(Y(z) - B(z))}\,dz }{\int_{-\infty}^x e^{\beta(Y(z) - B(z))}\,dz}\Biggr) =e^{\beta B(x,y)}\Biggl(1 + \f{\int_x^y e^{\beta(Y(x,z) - B(x,z))}\,dz }{\int_{-\infty}^x e^{\beta(Y(x,z) - B(x,z))}\,dz}\Biggr).
\end{align*}
All statements of the lemma now follow from the last equality. 
\end{proof}

\begin{lemma}\label{sDppres}
For $\beta > 0$ and $A \subseteq \R$, $\sDp^{(n)}_\beta:\Y_n^A \to \X_n^A$. 
\end{lemma}
\begin{proof}
Let $\mbf Y = (Y^1,\ldots,Y^n)$ and $\boldsymbol \eta = \sDp^{(n)}_\beta(\mbf Y)$. From  Lemma \ref{lem:OCYlim} and induction, it follows that for $1 \le i \le n$, 
\[
\lim_{x \to -\infty} \f{Y^i(x)}{x} = \lim_{x \to -\infty}\f{\eta^i(x)}{x}.
\]
It remains to show that $\eta_{i - 1} \sli \eta_i$, which we prove by induction. By Lemma \ref{lem:OCYmont},
$\eta^2 = \Dp_\beta(Y^1,Y^2) \sgi Y^1 = \eta^1$.
Now, we assume that 
\[
\eta^{i} = \Dp_\beta^{(i)}(Y^{1},\ldots,Y^i) \sgi \Dp_\beta^{(i - 1)}(Y^{1},\ldots,Y^{i-1}) = \eta^{i - 1}.
\]
We apply this assumption, replacing $Y^1,\ldots,Y^i$ with  $Y^2,\ldots,Y^{i + 1}$ along with Lemma \ref{lem:OCYmont} to get
\begin{align*}
\eta^{i + 1} &= \Dp_\beta^{(i + 1)}(Y^{1},\ldots,Y^{i+1}) = \Dp_\beta(Y^1,\Dp_\beta^{(i)}(Y^{2},\ldots,Y^{i+1})) \\
&\sgi \Dp_\beta(Y^1,\Dp_\beta^{(i - 1)}(Y^{2},\ldots,Y^{i})) = \Dp_\beta^{(i)}(Y^1,\ldots,Y^i) = \eta^i.  \qedhere
\end{align*}
\end{proof}

\begin{lemma} \label{lem:mubetaweak}
Let $\beta > 0$ and $\lambda_1 < \cdots < \lambda_k$.  Let $\lambda_i^N,\beta^N$ be sequences such that  $\lambda_i^N \to \lambda_i$ and $\beta^N \to \beta$ as $N\to\infty$. Set $\bar \lambda^N = (\lambda_1^N,\ldots,\lambda_k^N)$ and $\bar \lambda = (\lambda_1,\ldots,\lambda_k)$. Then, $\mu_{\beta_N}^{\bar \lambda^N} \to \mu_\beta^{\bar \lambda}$ weakly as measures on $C(\R,\R^k)$.
\end{lemma}

\begin{proof}
Realize the distributions in terms of $(\eta_N^1,\ldots,\eta_N^k)\sim\mu_{\beta_N}^{\bar \lambda^N}$ and $(\eta^1,\ldots,\eta^k)\sim\mu_{\beta}^{\bar \lambda}$, where 
$(\eta_N^1,\ldots,\eta_N^k) =  \sDp^{(k)}_{\beta^N}(Z_N^1,\ldots,Z_N^k)$ and 
$(\eta^1,\ldots,\eta^k) = \sDp^{(k)}_\beta(Y^1,\ldots,Y^k)$, 
and the Brownian motions 
$(Z_N^1,\ldots,Z_N^k) \sim \nu^{\bar \lambda^N}$
and 
$(Y^1,\ldots,Y^k) \sim \nu^{\bar \lambda}$  
are coupled so that 
$Z_N^i(x) = Y^i(x) + (\lambda_i^N - \lambda_i)x$.  By \eqref{Did}, 
\be \label{DZN} \nonumber
\begin{aligned}
\eta_N^j(y) &= Z^1_N(y) + \f{1}{\beta_N} \log \int_{-\infty < x_{j - 1} < \cdots < x_1 < y} \prod_{i = 1}^{j-1}\exp\Bigl\{\beta_N(Z_N^{i + 1}(x_i) - Z_N^i(x_i))\Bigr\} dx_i \\
&\qquad\qquad\qquad-\f{1}{\beta_N} \log \int_{-\infty < x_{j - 1} < \cdots < x_1 < 0} \prod_{i = 1}^{j-1}\exp\Bigl\{\beta_N(Z_N^{i + 1}(x_i) - Z_N^i(x_i))\Bigr\} dx_i.
\end{aligned}
\ee
Dominated convergence applied to the integrals gives 
 $(\eta_N^1,\ldots,\eta_N^k) \Rightarrow (\eta^1,\ldots,\eta^k)$ in the sense of finite-dimensional distributions. Each $\eta_N^i$ is a Brownian motion with drift $\lambda^N_i$, so each marginal is tight in $C(\R)$. Hence, the process $(\eta_N^1,\ldots,\eta_N^k)$ is tight in $C(\R,\R^k)$. 
\end{proof}

For $\gamma > 0$ and $\alpha \in \R$  define the mapping $T_{\gamma,\alpha}:C(\R) \to C(\R)$ as \[ T_{\gamma,\alpha} f(x) = \gamma^{-1} f(\gamma^2x) + \alpha x. \]
Extend it to a mapping  $T^n_{\gamma,\alpha}:C(\R,\R^n) \to C(\R,\R^n)$ of $n$-tuples componentwise: 
\[
T^n_{\gamma,\alpha} (f_1,\ldots,f_n) = (T_{\gamma,\alpha} f_1,\ldots, T_{\gamma,\alpha} f_n). 
\]
For $\alpha = 0$ use the shorthand notation $T_\gamma = T_{\gamma,0}$ and $T^n_\gamma = T^n_{\gamma,0}$.

\begin{lemma} \label{lem:pQscaling}
For $\beta, \gamma > 0$, $\alpha \in \R$, and $Y^1,\ldots,Y^n$ such that the following are all finite, we have
\be \label{diff}
T_{\gamma,\alpha} \Dp^{(n)}_\beta(Y^1,\ldots,Y^n) = \Dp^{(n)}_{\beta \gamma}(T_{\gamma,\alpha}^n( Y^1,\ldots, Y^n)),
\ee
and 
\be \label{Rscaling}
T_{\gamma,\alpha} R_\beta(Y^1,Y^2) = R_{\beta \gamma}(T_{\gamma,\alpha}^2( Y^1, Y^2)).
\ee
Consequently, for $\bar \lambda = (\lambda_1 < \cdots < \lambda_n)$,
\be \label{measmap}
\mu^{\bar \lambda}_\beta \circ (T^{n}_{\gamma,\alpha})^{-1} = \mu^{\gamma \bar \lambda + (\alpha,\ldots,\alpha)}_{\beta \gamma}.
\ee
In particular,
\be \label{temp1}
\mu^{\bar \lambda}_\beta = \mu^{\beta^{-1} \bar \lambda}_1 \circ (T^{n}_{\beta})^{-1}
\ee
\end{lemma}
\begin{remark}
Equation \eqref{temp1} allows us to perform computations for $\beta = 1$ and extend to general $\beta$. 
\end{remark}

\begin{proof}
\eqref{measmap} follows from \eqref{diff} because, if $(Y^1,\ldots,Y^n) \sim \nu^{\bar \lambda}$, then $T^{n}_{\gamma,\alpha}(Y^1,\ldots,Y^n) \sim \nu^{\gamma \bar \lambda + (\alpha,\ldots,\alpha)}$. We turn our attention to proving \eqref{diff}.  To do so, we use Lemma \ref{lem:pQalt}. For $y \in \R$,
\begin{align*}
&\quad T^{n}_{\gamma,\alpha}\Dp_\beta^{(n)}(Y^1,\ldots,Y^n)(y) \\
&= \gamma^{-1} Y^1(\gamma^2 y) + \alpha y  + \f{1}{\beta \gamma } \log \int_{-\infty < x_{n - 1} < \cdots < x_1 < \gamma^2 y} \prod_{i = 1}^{n-1}\exp\bigl(\beta(Y^{i + 1}(x_i) - Y^i(x_i))\bigr)dx_{i} \\
&\qquad\qquad\qquad\qquad- \f{1}{\beta\gamma}\log \int_{-\infty < x_{n - 1} < \cdots < x_1 < 0} \prod_{i = 1}^{n-1}\exp\bigl(\beta(Y^{i + 1}(x_i) - Y^i(x_i))\bigr)dx_{i} \\
&= T_{\gamma,\alpha} Y^1(y) + \f{1}{\beta \gamma } \log \int_{-\infty < w_{n - 1} < \cdots < w_1 < y} \prod_{i=1}^{n-1} \exp\bigl(\beta \gamma (T_{\gamma,\alpha} Y^{i + 1}(w_i) - T_{\gamma,\alpha} Y^i(w_i))\bigr)dw_{i} \\
&\qquad\qquad\qquad\qquad-  \f{1}{\beta \gamma } \log \int_{-\infty < w_{n - 1} < \cdots < w_1 < 0} \prod_{i=1}^{n-1} \exp\bigl(\beta \gamma (T_{\gamma,\alpha} Y^{i + 1}(w_i) - T_{\gamma,\alpha} Y^i(w_i))\bigr)dw_{i} \\
&= \Dp^{(n)}_{\beta \gamma}(T^{n}_{\gamma,\alpha}(Y^1,\ldots,Y^n)),
\end{align*}
where in the second equality, we made the change of variables $x_i = \gamma^2w_i$, with the Jacobian term cancelling in the difference of the logs of the two integrals. The proof of \eqref{Rscaling} is analogous. 
\end{proof}

\begin{lemma} \label{lem:intertwine}
Let $\beta > 0$ and let $Y^2,Y^1,B^1:\R \to \R$ be so that the following are well-defined. Set $B^2 = R_\beta(Y^1,B^1)$. Then, 
\be \label{fullid}
\Dp_\beta(\Dp_\beta(B^1,Y^1),\Dp_\beta(B^2,Y^2)) = \Dp_\beta^{(3)}(B^1,Y^1,Y^2).
\ee
\end{lemma}
\begin{proof}
Written out fully, the statement reads
\[
\Dp_\beta(\Dp_\beta(B^1,Y^1),\Dp_\beta(\Rp_\beta(B^1,Y^1),Y^2)) = \Dp_\beta^{(3)}(B^1,Y^1,Y^2).
\]
By applying Lemma \ref{lem:pQscaling} to each of the operations $\Dp$ and $\Rp$, this is equivalent to
\[
T_\beta \Dp_1(\Dp_1(T_{\beta^{-1}}B^1,T_{\beta^{-1}}Y^1),\Dp_1(\Rp_1(T_{\beta^{-1}} B^1,T_{\beta^{-1}}Y^1),T_{\beta^{-1}} Y^2))= T_\beta \Dp_1^{(3)}(T_{\beta^{-1}}B^1, T_{\beta^{-1}}Y^1,T_{\beta^{-1}}Y^2).
\]
Hence, it suffices to prove the $\beta = 1$ case. For this, we drop the subscript in the mappings $D,R$, and $\Dp^{(3)}$. We make repeated use of Lemma \ref{lem:pQalt}. By \eqref{Did}, 
\be \label{RHSint}
\begin{aligned}
\exp\bigl(D^{(3)}(B^1,Y^1,Y^2)(y)\bigr) &= \f{e^{B_1(y)}\int_{-\infty < w < x < y} \exp\bigl(Y^1(x) - B^1(x) + Y^2(w) - Y^1(w) \bigr)\,dx\,dw }{\int_{-\infty < w < x < 0} \exp\bigl(Y^1(x) - B^1(x) + Y^2(w) - Y^1(w) \bigr)\,dx\,dw}.
\end{aligned}
\ee
We turn to the left-hand side of \eqref{fullid} for $\beta = 1$.  We repeatedly use the $n = 2$ case of Lemma \ref{lem:pQalt} as follows:
\be \label{inteqp1}
\begin{aligned}
    \exp\bigl[D(D(B^1,Y^1),D(B^2,Y^2))(y)\bigr] 
    &= \f{e^{D(B^1,Y^1)(y)}\int_{-\infty}^y e^{D(B^2,Y^2)(x) - D(B^1,Y^1)(x)}\,dx }{\int_{-\infty}^0 e^{D(B^2,Y^2)(x) - D(B^1,Y^1)(x)}\,dx} \\[0.5em]
    &= \f{e^{B^1(y)} \int_{-\infty}^y e^{Y^1(x) - B^1(x)}\,dx  }{\int_{-\infty}^0 e^{Y^1(x) - B^1(x)}\,dx } \cdot \f{I_y}{I_0},
    \end{aligned}
\ee
where 
\[
\begin{aligned}
    I_y &= \int_{-\infty}^y e^{B^2(x) - B^1(x)} \f{\int_{-\infty}^x e^{Y^2(w) - B^2(w)}\,dw \int_{-\infty}^0 e^{Y^1(w) - B^1(w)}\,dw }{\int_{-\infty}^0 e^{Y^2(w) - B^2(w)}\,dw \int_{-\infty}^x e^{Y^1(w) - B^1(w)}\,dw }\,dx \\[0.5em]
    &= \int_{-\infty}^y e^{Y^1(x) - B^1(x)} \f{\bigl(\int_{-\infty}^0 e^{Y^1(w) - B^1(w)}\,dw\bigr)^2}{\bigl(\int_{-\infty}^x e^{Y^1(w) - B^1(w)}\,dw\bigr)^2} \f{\int_{-\infty}^x e^{Y^2(w) - B^2(w)}\,dw}{\int_{-\infty}^0 e^{Y^2(w) - B^2(w)}\,dw  } \,dx \\[0.5em]
    &= \int_{-\infty}^y e^{Y^1(x) - B^1(x)} \f{\bigl(\int_{-\infty}^0 e^{Y^1(w) - B^1(w)}\,dw\bigr)^2}{\bigl(\int_{-\infty}^x e^{Y^1(w) - B^1(w)}\,dw\bigr)^2} \f{\int_{-\infty}^x e^{Y^2(w) - Y^1(w)} \int_{-\infty}^w e^{Y^1(z) - B^1(z)}\,dz \,dw }{\int_{-\infty}^0 e^{Y^2(w) - Y^1(w)} \int_{-\infty}^w e^{Y^1(z) - B^1(z)}\,dz \,dw } \,dx.
    \end{aligned}
\]
Therefore, $I_y/I_0 = I_y'/I_0'$, where 
\be \label{Iyp}
I_y' = \int_{-\infty}^y e^{Y^1(x) - B^1(x)} \f{\int_{-\infty}^x\int_{-\infty}^w \exp\bigl(Y^2(w) - Y^1(w) + Y^1(z) - B^1(z)\bigr)\,dz \,dw}{\bigl(\int_{-\infty}^x e^{Y^1(w) - B^1(w)}\,dw\bigr)^2}\,dx.
\ee
Comparing \eqref{RHSint}, \eqref{inteqp1}, and \eqref{Iyp}, to prove \eqref{fullid}, it suffices to show that for each $y \in \R$, 
\be \label{intfin}
\begin{aligned}
 &\quad \int_{-\infty}^y e^{Y^1(x) - B^1(x)}\,dx \int_{-\infty}^y e^{Y^1(x) - B^1(x)} \f{\int_{-\infty}^x\int_{-\infty}^w e^{Y^2(w) - Y^1(w) + Y^1(z) - B^1(z)}\,dz \,dw}{\bigl(\int_{-\infty}^x e^{Y^1(w) - B^1(w)}\,dw\bigr)^2}\,dx \\[0.5em]
 &\qquad 
 = \int_{-\infty < w < x < y} e^{Y^1(x) - B^1(x) + Y^2(w) - Y^1(w)}\,dx\,dw. 
 \end{aligned}
\ee
On the left-hand side of \eqref{intfin}, integrate by parts in the second $dx$ integral over $(-\infty,y]$  with
\[
dv = e^{Y^2(x) - B^1(x)}\Bigl(\int_{-\infty}^x e^{Y^1(w) - B^1(w)}\,dw\Bigr)^{-2},\quad u = \int_{-\infty}^x\int_{-\infty}^w e^{Y^2(w) - Y^1(w) + Y^1(z) - B^1(z)}\,dz \,dw. 
\]
Then the left-hand side  of \eqref{intfin} equals
\[
\begin{aligned}
 &\int_{-\infty}^y e^{Y^1(x) - B^1(x)}\,dx \cdot\Biggl[-\Bigl(\int_{-\infty}^x e^{Y^1(w) - B^1(w)}\,dw \Bigr)^{-1} \int_{-\infty}^x\int_{-\infty}^w e^{Y^2(w) - Y^1(w) + Y^1(z) - B^1(z)}\,dz \,dw \Bigg |_{x = -\infty}^{y}  \\[0.5em] 
&\qquad\qquad\qquad\qquad\qquad\qquad
+ \int_{-\infty}^y \Bigl(\int_{-\infty}^x e^{Y^1(x) - B^1(x)}\,dx \Bigr)^{-1} \int_{-\infty}^x e^{Y^2(x) - Y^1(x) + Y^1(z) - B^1(z)}\,dz \,dx \Biggr] \\[0.5em]
&= \int_{-\infty}^y e^{Y^1(x) - B^1(x)}\,dx \int_{-\infty}^y e^{Y^2(x) - Y^1(x)}\,dx - \int_{-\infty}^y \int_{-\infty}^w e^{Y^2(w) - Y^1(w) + Y^1(z) - B^1(z)}\,dz\,dw \\[0.5em]
&= \int_{(x,w) \in (-\infty,y)^2} e^{Y^1(x) - B^1(x) + Y^2(w) - Y^1(w)}\,dx\,dw \\[0.5em]
&\qquad\qquad\qquad\qquad- \int_{-\infty < x < w <y} e^{Y^1(x) - B^1(x) + Y^2(w) - Y^1(w)}\,dx\,dw.
\end{aligned}
\]
One readily sees that the last right-hand side above equals the right-hand side of \eqref{intfin}. 
\end{proof}

\subsection{Consistency and invariance} \label{sec:consistent_invar}
With the needed inputs from Section \ref{sec:queue}, the following results follow analogously as for zero temperature in \cite{Seppalainen-Sorensen-21b,Sorensen-thesis}. The key technical input for both of these results is Lemma \ref{lem:intertwine} above, whose proof is much different than in zero temperature (see \cite[Lemma 7.6]{Seppalainen-Sorensen-21b}). The full proofs of the following two results are found in Appendix \ref{appx:queue}.

\begin{lemma} \label{lem:OCY_consis}
Let $\bar \lambda = (\lambda_1 < \lambda_2 < \cdots < \lambda_n)\in\R^n$. If $(\eta^1,\ldots,\eta^n)\sim \mu_\beta^{\bar \lambda}$, then for any subsequence $\lambda_{i_1} < \cdots < \lambda_{i_k}$, $(\eta^{i_1},\ldots,\eta^{i_k}) \sim \mu_\beta^{(\lambda_{i_1},\ldots,\lambda_{i_k})}$. 
\end{lemma}

\begin{theorem} \label{thm:OCYinvar}
For an increasing vector $\bar \lambda=(\lambda_1, \dotsc, \lambda_n)$ of strictly positive drifts $0 < \lambda_1 < \cdots <\lambda_n$, the measure $\mu_\beta^{\bar \lambda}$ is an invariant measure for the Markov chain on $\X_n^{\R_{>0}}$ whose time $m-1$ to time $m$ transition is defined as follows.  Let $(\eta_{m - 1}^1,\ldots,\eta_{m -1}^n)$ be the state at time $m - 1$, and let $B_m$ be a standard two-sided Brownian motion, independent of the Markov chain in the past. Then the state at time $m$ is 
\be \label{eqn:OCYMC}
(\eta_m^1,\ldots,\eta_m^n) = (D_\beta(B_m,\eta_{m-1}^1),\ldots,D_\beta(B_m,\eta_{m-1}^n)). 
\ee
\end{theorem}
\begin{remark}
    The strictly positive drifts ensure that condition \eqref{Q400} is satisfied, and the transformations above are well-defined almost surely.
\end{remark}

\subsection{Construction of the KPZ horizon} \label{sec:KPZH_cons_sec}
The Skorokhod space $D(\R,C(\R))$ consists of functions $\R \to C(\R)$ that are right-continuous with left limits. $C(\R)$ is endowed with the topology of uniform convergence on compact sets. A generic element of $D(\R,C(\R))$ is denoted by  $F = \{F^\lambda\}_{\lambda \in \R}$, where $F^\lambda \in C(\R)$ for each $\lambda$. The standard $\sigma$-algebra $\B_D$ on $D(\R,C(\R))$ is generated by the projections $\pi^{\lambda_1,\ldots,\lambda_k}:D(\R,C(\R))\to C(\R,\R^k)$ defined by  $\pi^{\lambda_1,\ldots,\lambda_k}(F) = (F^{\lambda_1},\ldots,F^{\lambda_k})$ (See, for example, \cite[Sections 12-13]{billing} and \cite[Page 101]{Schwartz-1973}.) Recall the measures $\mu_\beta^{\bar \lambda}$ defined in \eqref{mubeta}.

\begin{proposition} \label{prop:KPZH_cons}  
     On the space $(D(\R,C(\R)),\B_D)$, there exists a family of probability measures $\Pp_\beta$ indexed by the inverse temperature $\beta > 0$, satisfying the following properties. Let $\KH_\beta = \{\KH_\beta^\lambda\}_{\lambda \in \R}$ denote the random element of $D(\R,C(\R))$ under the measure $\Pp_\beta$.     
     \begin{enumerate} [label={\rm(\roman*)}, ref={\rm(\roman*)}]   \itemsep=3pt
     \item \label{itm:KPZHBM} For $\beta > 0$ and $\lambda \in \R$, $\KH_\beta^\lambda$ is a two-sided Brownian motion with diffusivity $1$ and drift $\lambda$. In particular, $\Pp_\beta$-almost surely,  $\KH_\beta^\lambda(0) = 0$ for each $\lambda \in \R$. 
     \item \label{itm:KPZH_dist} For $\beta > 0$ and an increasing  vector $\bar \lambda = (\lambda_1 < \cdots < \lambda_k)\in\R^k$ of drifts,  the $C(\R,\R^k)$-valued  $k$-tuple $(\KH_\beta^{\lambda_1},\ldots,\KH_\beta^{\lambda_k})$ has distribution $\mu_\beta^{\bar \lambda}$. Equivalently, in terms of projections, 
     $\Pp_\beta \circ (\pi^{\bar \lambda })^{-1} = \mu_\beta^{\bar \lambda }$.  
     In terms of the mapping $\sDp^{(k)}_\beta$  and independent Brownian motions  $Y^1,\ldots,Y^k$ with drifts $\lambda_1<\dotsm<\lambda_k$, 
     \be\label{FY500}  
(\KH_\beta^{\lambda_1},\ldots,\KH_\beta^{\lambda_k})
     \;\deq\;  
     (Y^1,\Dp^{(2)}_\beta(Y^1,Y^2),\ldots, \Dp_\beta^{(k)}(Y^1,\ldots,Y^k)).  
     \ee
       The measure $\Pp_\beta$ is the unique probability measure on $D(\R,C(\R))$ satisfying this finite-dimensional marginal condition. 
     \item \label{itm:KPZH_mont} For $\beta > 0$, $\Pp_\beta$-almost surely, for all $\lambda_1 < \lambda_2$, $\KH_\beta^{\lambda_1} \sli \KH_\beta^{\lambda_2}$.
    \end{enumerate}
\end{proposition}
\begin{remark}
With a nod to the stationary horizon (SH) discussed above in Section \ref{sec:SH-KPZH}, we call the process $\{\KH_\beta^\lambda\}_{\lambda \in \R}$ the
\textit{KPZ horizon at inverse temperature $\beta$}, abbreviated $\KPZH_\beta$.  
\end{remark}

\begin{proof}
The construction follows a similar procedure as the construction of the SH in \cite{Sorensen-thesis} (we note here that the SH was originally constructed as a limit of the Busemann process in exponential LPP in \cite{Busani-2021}). We start by recalling Lemma \ref{lem:OCY_consis}, which states that the measures $\mu^{\bar \lambda}_\beta$ are consistent. Thus, for $\bar \lambda = (\lambda_1 < \cdots < \lambda_k)$, if $(\eta^1,\ldots,\eta^k) \sim \mu^{\bar \lambda}$, each $\eta^i$ has distribution $\mu_\beta^{\lambda_i}$, which is the law of a two-sided Brownian motion with diffusion coefficient $1$ and drift $\lambda_i$. By Kolmogorov's extension theorem, there exists a unique measure $\mu^\Q_\beta$ on $C(\R)^\Q = \prod_{\Q}C(\R)$ under which, for $\{\wt \KH^\alpha\}_{\alpha \in \Q} \in C(\R)^\Q$ 
and any choice of $\bar \alpha = (\alpha_1,\dotsc,\alpha_k) \in \Q^k$ with $\alpha_1 < \cdots < \alpha_k$, $(\wt \KH^{\alpha_1},\ldots,\wt \KH^{\alpha_k}) \sim \mu^{\bar \alpha}_\beta$. In particular,  under $\mu_\beta^\Q$ each $\wt \KH^\alpha$ is a Brownian motion with drift $\alpha$. 

Because the measures $\mu_\beta^{\bar \lambda}$ are supported on the sets $\X_k^\R$ of \eqref{Xndef}, we have that
\be \label{mubetamont}
\mu^\Q_\beta\bigl(\wt \KH^{\alpha_1} \sli \wt \KH^{\alpha_2} \;\; \forall \alpha_1 < \alpha_2 \in \Q\bigr) = 1.
\ee
Hence, there is a full probability event for $\mu^\Q_\beta$, on which, for each $\lambda\in \R$ and $x \in \R$, the limits
\be \label{eqn:KPZHdef}
\KH^{\lambda}(x) := \lim_{\Q \ni \alpha \searrow \lambda}\wt \KH^\alpha(x) \qquad\text{and}\qquad  \KH^{\lambda -} \lim_{\Q \ni \alpha \nearrow \lambda} \wt \KH^\alpha(x)
\ee
exist. By construction,
\be \label{eqwtKHmont}
\mu^\Q_\beta(\KH^\lambda \li \wt \KH^\alpha \;\; \forall \lambda \in \R, \alpha \in \Q \text{ with } \alpha > \lambda) = 1.
\ee 
Then, on the event of \eqref{eqwtKHmont}, for $A < a < b < B$,  
\[
\KH^\lambda(A,a) + \KH^\lambda(b,B) \le  \wt \KH^\alpha(A,a) + \wt \KH^\alpha(b,B),
\]
or equivalently,
\[
0 \le \wt \KH^\alpha(a,b) -  \KH^\lambda(a,b) \le  \wt \KH^\alpha(A,B) - \KH^\lambda(A,B), 
\]
implying that the convergence is uniform on compact sets. The same holds for limits from the left. By monotonicity, $\mu^\Q_\beta( \KH^{\lambda -} \li \wt \KH^\lambda \li  \KH^\lambda \;\;\forall \lambda \in \Q) = 1$. Additionally, uniform convergence ensures that, for each $\lambda \in \R$,  $ \KH^{\lambda-}$ and $\KH^{\lambda}$ are both Brownian motions with drift $\lambda$. Hence, for each $\lambda \in \Q$, 
\[
\mu^\Q_\beta(\KH^{\lambda-}= \wt \KH^\lambda = \KH^\lambda )
= 1.
\]

In summary, we have defined a stochastic process $\{ \KH^\lambda\}_{\lambda \in \R}$ whose projection to the rationals agrees with $\{\wt \KH^\lambda\}_{\lambda \in \Q}$ under the measure $\mu_\beta^\Q$. Let $\lambda_1 < \lambda_2 < \lambda_3$ be real, and choose rational values $\alpha_1 < \lambda_1 < \alpha_2 < \lambda_2 <  \alpha_3 < \lambda_3 < \alpha_4$. Then,
\be \label{eqn:KHbig_mont}
\KH^{\alpha_1} \li  \KH^{\lambda_1} \li  \KH^{\alpha_2} \li  \KH^{\lambda_2} \li  \KH^{\alpha_3} \li  \KH^{\lambda_3} \li  \KH^{\alpha_4}.
\ee
This implies that, $\mu_\beta^\Q$-almost surely, simultaneously for every $\lambda \in \R$, the following limits exist uniformly on compact sets, and they agree with the limits along rational directions.
\[
\KH^\lambda = \lim_{\alpha \searrow \lambda} \KH^\alpha \qquad\text{and}\qquad \KH^{\lambda -} = \lim_{ \alpha \nearrow \lambda} \KH^\alpha.
\]

Therefore, the process $\{\KH^\lambda\}_{\lambda \in \R}$ lies in the space $D(\R,C(\R))$. Let $\Pp_\beta$ be the pushforward of the measure $\mu_\beta^\Q$ to $D(\R,C(\R))$ under the map defined by \eqref{eqn:KPZHdef}. Without reference to the measure, we use $\{\KH_\beta^\lambda\}_{\lambda \in \R}$ to denote the process. 

We check that $\{\KH^\lambda\}_{\lambda \in \R}$ satisfies the claims of the theorem. Item \ref{itm:KPZHBM} follows from the uniform convergence along rational directions. Item \ref{itm:KPZH_dist} follows because for rational directions the finite-dimensional distributions were defined to be $\mu_\beta^{\bar \lambda}$. The limits in \eqref{eqn:KPZHdef} and the weak convergence  of Lemma \ref{lem:mubetaweak} extend this property  to all real directions. Since the $\sigma$-algebra on $D(\R,C(\R))$ is generated by the projections, uniqueness of this process on $D(\R,C(\R))$ follows.  To verify Item \ref{itm:KPZH_mont} for real  $\lambda_1 < \lambda_2$, pick   rational  $\alpha_1, \alpha_2$ such that  $\lambda_1 < \alpha_1 < \alpha_2 < \lambda_2$.  Then  \eqref{eqn:KHbig_mont} and  \eqref{mubetamont} give $
\KH_\beta^{\lambda_1} \li \KH_\beta^{\alpha_1} \sli \KH_\beta^{\alpha_2} \li \KH_\beta^{\lambda_2}. 
$
\end{proof}

\subsection{Distributional invariances of the $\KPZH$}

We prove the following distributional invariances of $\KPZH_\beta$.  Item \ref{itm:reflinv} below is not needed elsewhere in this paper, but it is included here for future use.

\begin{theorem} \label{thm:dist_invar}
For $\beta > 0$, let $\KH_\beta$ be the $\KPZH_\beta$. 
\begin{enumerate} [label={\rm(\roman*)}, ref={\rm(\roman*)}]   \itemsep=3pt 

\item \label{itm:shinv} Translation invariance: for each $x\in\R$, $\{\KH_\beta^\lambda(x,x + \aaabullet)\}_{\lambda \in \R} \deq \KH_\beta$.

\item\label{itm:KHscale} Scaling invariance: for each $\beta >0$, $\gamma > 0$, and $\alpha \in \R$, $
\{\gamma^{-1} F^\lambda_\beta(\gamma^2 \aabullet) + \alpha \aabullet \}_{\lambda \in \R} \deq \{F^{\gamma \lambda + \alpha}_{\gamma \beta}\}_{\lambda \in \R}.
$
\item \label{itm:inc_stat} Stationarity of increments: for $\lambda_1 < \lambda_2 < \cdots < \lambda_n$ and $\lambda^\star\in \R$,
\[
(\KH_\beta^{\lambda_2 } - \KH_\beta^{\lambda_1},\ldots,\KH_\beta^{\lambda_{n}} - \KH_\beta^{\lambda_{n - 1}}) \deq (\KH_\beta^{\lambda_2 + \lambda^\star } - \KH_\beta^{\lambda_1 + \lambda^\star},\ldots,\KH_\beta^{\lambda_{n} + \lambda^\star} - \KH_\beta^{\lambda_{n - 1} + \lambda^\star}).
\]
\item \label{itm:reflinv} Reflection invariance: $\{\KH_\beta^{(-\lambda)-}(-\tspb\aaabullet)\}_{\lambda \in \R} \deq \KH_\beta$, where $F_\beta^{(-\lambda)-} = \lim_{\alpha \nearrow \lambda} F_\beta^{(-\alpha)}$.
\end{enumerate}
\end{theorem}
\begin{proof}
 For $\lambda_1 < \cdots < \lambda_k$, by definition of the $\KPZH_\beta$ (Proposition \ref{prop:KPZH_cons}\ref{itm:KPZH_dist}),
\be \label{KHY} 
(\KH_\beta^{\lambda_1},\ldots,\KH_\beta^{\lambda_k}) \deq 
(Y^1,D_\beta(Y^1,Y^2),\ldots,D^{(k)}_\beta(Y^1,\ldots,Y^k)),
\ee
where $Y^1,\ldots,Y^n$ are independent Brownian motions with drifts $\lambda_1,\ldots,\lambda_n$.
 It follows from Lemma \ref{lem:pQalt} that $\Dp_\beta^{(k)}(Y^1,\ldots,Y^k)(x,x + y) = \Dp_\beta^{(k)}(Y^1(x,x+\aabullet),\ldots,Y^k(x,x + \aabullet))(y)$, from which Item \ref{itm:shinv} follows. Item \ref{itm:KHscale} follows from Lemma \ref{lem:pQscaling}. 

  Now we note that Item \ref{itm:inc_stat} follows from Item \ref{itm:KHscale}: Setting $\gamma = 1$, we obtain 
  \[
(\KH_{\beta}^{\lambda_1},\ldots,\KH_\beta^{\lambda_n}) \deq (\KH_\beta^{\lambda_1 + \lambda^\star}(\aabullet) - \lambda^\star \aabullet,\ldots,\KH_\beta^{\lambda_n + \lambda^\star}(\aabullet) - \lambda^\star \aabullet),
  \]
  so 
  \[
(\KH_\beta^{\lambda_2 } - \KH_\beta^{\lambda_1},\ldots,\KH_\beta^{\lambda_{n}} - \KH_\beta^{\lambda_{n - 1}}) \deq (\KH_\beta^{\lambda_2 + \lambda^\star } - \KH_\beta^{\lambda_1 + \lambda^\star},\ldots,\KH_\beta^{\lambda_{n} + \lambda^\star} - \KH_\beta^{\lambda_{n - 1} + \lambda^\star}).
\]

Item \ref{itm:reflinv} follows from  Theorem \ref{thm:SHEbuse}\ref{itm:SHEBusereflect} and Corollary \ref{cor:SHEBuse=KPZH}, the first of which is proved in \cite{Janj-Rass-Sepp-22} and the second of which we prove later in this paper. One may notice that Item \ref{itm:shinv} also follows from Theorems \ref{thm:SHEbuse}\ref{itm:SHEbuseshift} and Corollary \ref{cor:SHEBuse=KPZH}, but we have proved this item here  to avoid circular logic because it is used  to prove Corollary \ref{cor:SHEBuse=KPZH}.
\end{proof}


 \subsection{Difference of two functions (Proof of Theorem \ref{thm:lambda-F})} 
\begin{proof}[Proof of Theorem \ref{thm:lambda-F}] \label{sec: x_Desc}

Let $Y^1,Y^2$ be two independent Brownian motions  with drifts $\lambda_1<\lambda_2$. 
By \eqref{FY500} and  \eqref{Did}, as processes indexed by  $y \in \R$, 
 \be \label{inc_dist}
 e^{\beta(\KH_\beta^{\lambda_2}(y) - \KH_\beta^{\lambda_1}(y))}
 \deq \f{\int_{-\infty}^y e^{\beta(Y^2(x) - Y^1(x))}\,dx   }{\int_{-\infty}^0 e^{\beta(Y^2(x) - Y^1(x))}\,dx} = 1 + \f{\int_{0}^y e^{\beta(Y^2(x) - Y^1(x))}\,dx   }{\int_{-\infty}^0 e^{\beta(Y^2(x) - Y^1(x))}\,dx}.
 \ee
 The process $\beta(Y^2(\aabullet) - Y^1(\aabullet))$ has the distribution of $\sqrt 2 \beta B(\aabullet) + \beta \lambda \,\aabullet$, where $B$ is a standard two-sided Brownian motion. When only considering the process for $y \ge 0$, by the independence of Brownian increments, the numerator (as a process in $y \ge 0$) and the denominator of the last ratio  in \eqref{inc_dist} are independent.   The distribution of the denominator is computed in Lemma \ref{lem:int_Gamma} using results from \cite{Dufresne-1990}.
 \end{proof}
 
 \subsection{Discontinuities of $\KPZH$ in the drift parameter} \label{sec:disc}
 
 \begin{lemma} \label{lem:jumpscond}
On a probability space $(\Omega,\F,\Pp)$, let $\lambda \mapsto X(\lambda)$ be an increment-stationary, nondecreasing, almost surely continuous process with $\Ee[X(1)-X(0)] < \infty$. Then, for every $\ve > 0$,
\be \label{eqn:cont_in_prob}
\lim_{\Z \ni n \to \infty} n\Pp(X(n^{-1}) - X(0) > \ve) = 0.
\ee
\end{lemma}
\begin{proof}
Partition $[0,1]$ into disjoint intervals of length $n^{-1}$, and let  
\[
J_n^\ve = \sum_{i = 1}^n \ind \{X(i/n) - X((i - 1)/n) > \ve\}. 
\]
By increment-stationarity,  $\Ee[J_n^\ve] = n \Pp(X(n^{-1}) - X(0) > \ve).$
By pathwise uniform continuity of  $X$ on $[0,1]$, $J_n^\ve = 0$ for large enough $n$.  The bound 
$\ve J_n^\ve \le  X(1)-X(0)$
and dominated convergence  complete the proof. 
\end{proof}
\begin{remark}
The condition \eqref{eqn:cont_in_prob} appears in Chapter 12 of \cite{Breiman-book} as a defining condition of Brownian motion. The process we will apply this to is of a different nature, as it is nondecreasing and does not have independent increments.
\end{remark}

\begin{corollary} \label{cor:SHEjumps}
For $y \in \R$ and $\beta > 0$, the process $\lambda \mapsto \KH_\beta^\lambda(y)$ is not almost surely continuous. 
\end{corollary}
\begin{proof}
We may take $y > 0$ because for $y < 0$, Theorem \ref{thm:dist_invar}\ref{itm:shinv} and $\KH_\beta^\lambda(0) = 0$ (Proposition \ref{prop:KPZH_cons}\ref{itm:KPZHBM}) implies 
\[
\{\KH_\beta^\lambda(y)\}_{\lambda \in \R} =  \{-\KH_\beta^\lambda(y,0)\}_{\lambda \in \R} \deq  \{-\KH_\beta^\lambda(0,-y)\}_{\lambda \in \R} 
= \{-\KH_\beta^\lambda(-y)\}_{\lambda \in \R}. 
\]
By the scaling relations of Theorem \ref{thm:dist_invar}\ref{itm:KHscale}, it suffices to take $\beta = 1$. We apply Lemma \ref{lem:jumpscond} to the process $\lambda \mapsto \KH_1^\lambda(y)$, which has stationary increments by Theorem \ref{thm:dist_invar}\ref{itm:inc_stat} and is strictly increasing by Proposition \ref{prop:KPZH_cons}\ref{itm:KPZH_mont}. Since $\KH_1^\lambda$ is a Brownian motion with drift $\lambda$ (Proposition \ref{prop:KPZH_cons}\ref{itm:KPZHBM}), we have $\Ee[\KH_1^1(y))] = y < \infty$. Lemma \ref{lem:jumpscond} reduces the problem to showing that for some $\ve > 0$,
\[
\liminf_{\lambda \searrow 0} \lambda^{-1} \Pp(\KH_1^\lambda(y) - \KH_1^0(y) > \ve) > 0.
\]
In fact, we show that this is true for all $\ve > 0$. 
For each $\lambda > 0$, let $X_\lambda,Y_\lambda(y)$ be the independent random variables of Theorem \ref{thm:lambda-F} with $\beta = 1$ so that  $\KH_1^\lambda(y) - \KH_1^0(y) \deq\log(1 + X_\lambda Y_\lambda(y))$. Observe that for a standard Brownian motion $B$,  
\[
Y_\lambda(y) \deq \int_0^y \exp\bigl(\sqrt 2  B(x) + \lambda x\bigr)\,dx > \int_0^y \exp\bigl(\sqrt 2 B(x)\bigr)\,dx =: Y,
\]
where $Y$ is taken as a new random variable  independent of $X_\lambda$. 
 By formula 1.8.4 on page 612 of \cite{BM_handbook}, $Y$ has a   density function $f_Y$  that  is 
strictly 
 positive on $(0,\infty)$.  For $\ve > 0$, let $\ve' = e^\ve - 1 > 0$. Since $X_\lambda\sim$ Gamma$(\lambda,1)$, 
\[
\begin{aligned}
&\quad \Pp(\KH_1^\lambda(y) - \KH_1^0(y) > \ve) = \Pp(X_\lambda Y_\lambda(y) > \ve') 
    \ge \Pp(X_\lambda Y > \ve') = \Pp(X_\lambda > \ve'/Y)  \\
    &= \int_0^\infty \int_{\ve'/w} f_Y(w) \f{1}{\Gamma(\lambda)} x^{\lambda -1} e^{-x}\,dx\,dw \ge \f{1}{\Gamma(\lambda)} \int_0^\infty \int_{1 \vee ({\ve'}/{w})} f_Y(w) x^{-1} e^{-x}\,dx\,dw = \f{C_\ve}{\Gamma(\lambda)},
    \end{aligned}
\]
where $C_\ve$ is a positive constant. Thus, 
$
\liminf_{\lambda \searrow 0} \lambda^{-1} \Pp(F_1^\lambda(y) > \ve) \ge C_\ve > 0
$ because $\lim_{\lambda \searrow 0} \lambda \Gamma(\lambda) = 1$.
\end{proof}

\section{Stochastic heat equation} \label{sec:SHE}

This section collects the necessary background on the SHE and KPZ equation. Section \ref{sec:SHEGreen} mainly summarizes results from \cite{Alb-Janj-Rass-Sepp-22,Janj-Rass-Sepp-22}. Section \ref{sec:OCYSHE} deals with convergence of the OCY polymer to the SHE.

\subsection{Green's function and Busemann process of the SHE} \label{sec:SHEGreen}

We briefly describe the construction of the four-parameter field $Z_\beta(\aabullet,\aabullet\viiva \aabullet,\aabullet)$ from \cite{Alberts-Khanin-Quastel-2014a,Alberts-Khanin-Quastel-2014b,Alb-Janj-Rass-Sepp-22}. We primarily adopt the notation of \cite{Alb-Janj-Rass-Sepp-22} and \cite{Janj-Rass-Sepp-22}. In the following, we give a brief overview of white noise. Some standard references are \cite{Janson,Nua06}.

On an appropriate probability space $(\Omega,\F,\Pp)$, a space-time white noise $W$ is a mean-zero Gaussian process whose index set is $L^2(\R^2)$, with Lebesgue measure.  It satisfies the almost sure linearity $W(af + bg) = aW(f) + bW(g)$ as well as the $L^2$ isometry property:
\[
\Ee[W(f)W(g)] = \int_{\R^2} f(t,x)g(t,x)\,dt\,dx.
\]
One immediate consequence is that, whenever $A$ and $B$ are disjoint, or more generally, their intersection has Lebesgue measure $0$, $W(\ind_A)$ and $W(\ind_B)$ are independent. As a point of notation, we often write
\[
W(f) = \int_{\R^2} f(t,x) W(dt\,dx) = \int_{\R^2} f(t,x)W(t,x)\,dt\,dx,
\]
where the second equality is formal because $W$ is a random distribution and not defined pointwise.  

We can also define multiple white noise integrals, as in \cite[Section 1.1.2]{Nua06}, denoted
\[
I_k(f) = \int_{\R^k}\int_{\R^k} f(t_1,\ldots,t_k,x_1,\ldots,x_k) \prod_{i = 1}^k W(dt_i,dx_i).
\]
These satisfy $E[I_k(f) I_j(g)] = 0$ for $k \neq j$ (orthogonality), and 
\begin{equation} \label{eq:Ikfbd}
E[I_k(f)^2]\le \|f\|_{L^2(\R^k \times \R^k)}^2.
\end{equation}
 Equality holds if we replace $f$ with the symmetrization of $f$, but in general, we have an inequality (see \cite[Section 1.1.2]{Nua06}).  

We define $Z_\beta$ as the following chaos expansion, where  convergence holds in $L^2(\Pp)$.
\be \label{SHEchaos}
Z_\beta(t,y\viiva s,x) =  \sum_{k = 0}^\infty \beta^k \int_{\R^k}\int_{\R^k} \prod_{i = 0}^k \rho(t_{i + 1} - t_i,x_{i +1} - x_i) \prod_{i = 1}^k W(dt_i,dx_i).
\ee
Here, $\rho(t,x) = \f{1}{\sqrt{2\pi t}} e^{-\f{x^2}{2t}} \ind(t > 0)$ is the  heat kernel, 
and the conventions in the integrals are $t_0 = s,x_0 = x,t_{k +1} = t$, and $x_{k + 1} = y$. 
For $f:\R \to \R_{>0}$ with sufficient decay at $\pm \infty$, and $t > s$, define
\be \label{ZZf}
Z_\beta(t,y\viiva s, f) = \int_\R Z_\beta(t,y\viiva s,x) f(x)\,dx.
\ee
When the value of $s$ is unspecified in \eqref{ZZf}, we take $s = 0$. Theorem 2.2 and Lemma A.5 of \cite{Alb-Janj-Rass-Sepp-22} prove that, in the rigorous sense of solutions in \cite{Bertini-Cancrini-1995,Bertini-Giacomin-1997,Chen-Dalang-2014,Chen-Dalang-2015}, the process \eqref{ZZf} solves
the SHE defined in \eqref{eqn:SHE}, for strictly positive functions $f = e^{h_s}$ satisfying
\[
\int_\R e^{-\alpha x^2}f(x)\,dx < \infty
\]
for all $\alpha > 0$. In fact, solutions can be defined for a class of measures which are not necessarily absolutely continuous with respect to Lebesgue measure, but in all applications of this paper, $f(x) = e^{B(x) + \lambda x}$, where  $B$ is a Brownian motion, and $\lambda \in \R$, so the necessary conditions are satisfied. We refer the reader to \cite[Appendix A]{Alb-Janj-Rass-Sepp-22} and the references therein for a more technical discussion on the solution of the SHE from measure-valued initial data.

\begin{theorem}   \cite[Proposition 2.3]{Alb-Janj-Rass-Sepp-22},\cite[Equation (18)]{Alberts-Khanin-Quastel-2014a} \label{thm:Zinvar} 
    Let $\beta > 0$. Then
  the following distributional equalities hold between random elements of $C(\Rup, \R)$.
    \begin{enumerate} [label={\rm(\roman*)}, ref={\rm(\roman*)}]   \itemsep=3pt
    \item \rm{(}Shift invariance\rm{)} \label{itm:Zshift} For given $u,z \in \R$, $Z_\beta(t,y\viiva s,x) \deq Z_\beta(t + u,y + z\viiva s + u,x + z).$
    \item \rm{(}Reflection invariance\rm{)} \label{itm:Zrefl}
    $
    Z_\beta(t,y \viiva s,x) \deq Z_\beta(t,-y \viiva,s,-x).
    $
    \item \rm{(}Rescaling\rm{)} \label{itm:Zscale} For given $\lambda > 0$,
    $
    Z_\beta(t,y\viiva s,x) \deq \lambda Z_{\beta/\sqrt \lambda}(\lambda^2 t,\lambda y \viiva \lambda^2 s,\lambda x).
    $
    \end{enumerate}
    Furthermore,
    \begin{enumerate} [resume,label={\rm(\roman*)}, ref={\rm(\roman*)}]   \itemsep=3pt
    \item \label{itm:SHEmoments}There exists a constant $C = C_\beta$ so that for all $t > s$ and $x,y \in \R$,
    \[
    \Ee[Z_\beta^2(t,y \viiva s,x)] \le C \rho^2(t - s,y - x).
    \]
    \end{enumerate}
\end{theorem}

We are particularly interested in the $\lambda = \beta^2$ case of Theorem \ref{thm:Zinvar}\ref{itm:Zscale}, in which the distributional equality becomes
\be \label{Zeq}
Z_\beta(t,y\viiva s,x) \deq \beta^2 Z_1(\beta^4 t,\beta^2 y \viiva \beta^4 s,\beta^2 x).
\ee
Jointly with the Busemann functions, by appeal to the  Busemann limits \eqref{eqn:buselim},  we have  this  distributional equality:
\be \label{Zbeq}
\begin{aligned}
& \bigl\{\bb^{\lambda \sig}(s,x,t,y), Z_\beta(t',y' \viiva s',x'): (s,x,t,y)\in\R^4 , \, (s',x',t',y')\in\Rup, \, \lambda \in \R, \, \sigg \in \{-,+\}\bigr\} \\
&\qquad\qquad\qquad
\deq \bigl\{b_1^{\lambda \beta^{-2} \sig}(\beta^4 s,\beta^2 x, \beta^4 t,\beta^2 y), \beta^2 Z_1(\beta^4t',\beta^2y' \viiva \beta^4s',\beta^2x'):  \\
&\qquad \qquad\qquad\qquad\qquad\qquad
(s,x,t,y)\in\R^4 , \, (s',x',t',y')\in\Rup, \, \lambda \in \R, \, \sigg \in \{-,+\}  \bigr\}.
\end{aligned}
\ee
In the following theorems, we use \eqref{Zbeq} to transfer the statements for $\beta = 1$ from \cite{Janj-Rass-Sepp-22} to general $\beta > 0$. We introduce the following class of functions, named $\mathbb F_\lambda$. Let $f :\R \to (0,\infty)$ be a Borel function that is locally bounded. Then, for $\lambda \in \R$, we say that $f \in \mathbb F_\lambda$ if 
\be \label{eqn:SHE_drift_cond}
\begin{aligned}
    -\infty \le \limsup_{x \to -\infty} \f{\log f(x)}{|x|} &< \lambda = \lim_{x \to \infty} \f{\log f(x)}{x} \qquad\qquad\qquad\ \text{ if }\lambda > 0 \\
    \lim_{x \to -\infty} \f{\log f(x)}{|x|} &= |\lambda| > \limsup_{x \to \infty} \f{\log f(x)}{x} \ge -\infty \qquad \text{ if }\lambda < 0 \\
    -\infty &\le \limsup_{|x| \to \infty} \f{\log f(x)}{|x|} \le 0\qquad \qquad\qquad\text{ if }\lambda = 0.
\end{aligned}
\ee

\begin{theorem}\cite[Theorems 3.1, 3.3, 3.5, 3.23, and Corollary 3.4]{Janj-Rass-Sepp-22} \label{thm:SHEbuse}
Let $\beta > 0$. Then, there exists a stochastic process $\{\bb^{\lambda \sig}(s,x,t,y): s,x,t,y,\lambda \in \R,\sigg \in \{-,+\}\}$ defined on the probability space $(\Omega,\F,\Pp)$ and satisfying the following properties. For this process, we define
\[
\Lambda_{\bb} = \{\lambda \in \R: \bb^{\lambda -}(s,x,t,y) \neq \bb^{\lambda +}(s,x,t,y) \text{ for some }(s,x,t,y) \in \R^4\}.
\]
When $\lambda \notin \Lambda_{\bb}$, we write $\bb^\lambda = \bb^{\lambda - } = \bb^{\lambda +}$.
\begin{enumerate} [label={\rm(\roman*)}, ref={\rm(\roman*)}]   \itemsep=3pt
\item \label{itm:BuseBM} For each $t,\lambda \in \R$, under $\Pp$, the process $y \mapsto \bb^{\lambda}(t,0,t,y)$ is a two-sided Brownian motion with diffusivity $\beta$ and drift $\lambda$. 
\item \label{itm:SHEbuseshift} \rm{(}Shift\rm{)} For $r,z \in \R$, as processes in $s,x,t,y,\lambda \in \R$, $\sigg \in \{-,+\}$,
\[
\bb^{\lambda \sig}(s,x,t,y) \deq \bb^{\lambda \sig}(s + r,x + z,t + r,y + z). 
\]
\item \label{itm:SHEBusereflect} \rm{(}Reflection\rm{)} As processes in $s,x,t,y,\lambda \in \R$, $\sigg \in \{-,+\}$,
\[
\bb^{\lambda \sig}(s,x,t,y) \deq \bb^{(-\lambda)(-\sig)}(s,-x,t,-y).
\]
\item \label{itm:SHEBuseP0} For each $\lambda \in \R$, $\Pp(\lambda \in \Lambda_{\bb}) = 0$.
\item \label{itm:SHEBusedichot} Either $\Pp(\Lambda_{\bb} = \varnothing)=1$ or $\Pp(\Lambda_{\bb} \text{ is countable and dense in } \R) = 1$.
\end{enumerate}
Furthermore, there exists an event of full probability on which the following hold:
\begin{enumerate}[resume, label={\rm(\roman*)}, ref={\rm(\roman*)}]   \itemsep=3pt
\item \label{itm:SHEBusecont} For each $\lambda \in \R$ and $\sigg \in \{-,+\}$, $\bb^{\lambda \sig} \in C(\R^4,\R)$. 
\item \label{itm:SHEBusemont} For all $x < y,t$, and $\alpha < \lambda$, 
\[
\bb^{\alpha-}(t,x,t,y) \le \bb^{\alpha +}(t,x,t,y) < \bb^{\lambda -}(t,x,t,y) \le \bb^{\lambda+}(t,x,t,y). 
\]
More specifically, whenever $\lambda \in \Lambda_{\bb}$, 
$
\bb^{\lambda -}(t,x,t,y) < \bb^{\lambda+}(t,x,t,y)
$, and consequently, for each $a \neq 0$,
    \[
    \Lambda_{\bb} =  \{\lambda \in \R: \bb^{\lambda-}(0,0,0,a) \neq \bb^{\lambda+}(0,0,0,a)\}.
    \]
\item \label{itm:SHEbuseadd} For all $r,x,s,y,t,z,\lambda$ and all $\sigg \in \{-,+\}$,
\[
\bb^{\lambda \sig}(r,x,s,y) + \bb^{\lambda \sig}(s,y,t,z) = \bb^{\lambda \sig}(r,x,t,z).
\]
\item \label{itm:SHEBusemontlim} For all $s,x,t,y,\lambda$ and all $\sigg \in \{-,+\}$, 
\[
\bb^{\lambda -}(s,x,t,y) = \lim_{\alpha \nearrow \lambda} \bb^{\alpha \sig}(s,x,t,y),\quad\text{and}\quad\bb^{\lambda+}(s,x,t,y) = \lim_{\alpha \searrow \lambda} \bb^{\alpha \sig}(s,x,t,y).
\]
\item \label{itm:SHEBuseevolve} For all $t > r$, all $s,x,y,\lambda$, and all $\sigg \in \{-,+\}$,
\[
e^{\bb^{\lambda \sig}(s,x,t,y)} = \int_\R e^{\bb^{\lambda \sig}(s,x,r,z)}Z_\beta(t,y\viiva r,z)\,dz.
\]
\item \label{itm:conv} For all $\lambda \notin \Lambda_{\bb}$ and $f \in \mathbb F_\lambda$, the following limit holds uniformly on compact sets of $(s,x,t,y) \in \R^4$:
\[
\lim_{r \to -\infty} \f{\int_\R f(z) Z_\beta(t,y\viiva r,z) \,dz}{\int_\R f(z) Z_\beta(s,x|r,z) \,dz} = e^{\bb^\lambda(s,x,t,y)}.
\]
\end{enumerate}
\end{theorem}

For later use, we 
  derive the following uniqueness result  from Theorem \ref{thm:SHEbuse}.
\begin{theorem} \label{thm:SHE_Buse_unique}
    Let $(f^1,\ldots,f^k)$ be a coupling of initial data with $f_i \in \mathbb F_{\lambda_i}$ almost surely for $i \in \{1,\ldots,k\}$. If, for all $t > 0$,
\[
\biggl\{\f{Z_\beta(t,\aabullet \viiva e^{f^i})}{Z_\beta(t,0 \viiva e^{f^i})} \biggr\}_{1 \le i \le k} \deq \{\exp(f^i)\}_{1 \le i \le k},
\]
then 
\[
\{\exp(f^i)\}_{1 \le i \le k} \deq \{\exp(\bb^{\lambda_i}(0,0,0,\aabullet\tspb))\}_{1 \le i \le k}.
\]
\end{theorem}
\begin{remark}
A stronger uniqueness property is true. The joint Busemann process is the unique stationary and ergodic jointly invariant distribution for the KPZ equation under more general conditions on the asymptotic slopes at $\pm\infty$. We refer the reader to Section 3.4 of \cite{Janj-Rass-Sepp-22} for a more detailed discussion. 
\end{remark}
\begin{proof}
The $r = s$ case of Theorem \ref{thm:SHEbuse}\ref{itm:SHEBuseevolve} along with the additivity of Theorem \ref{thm:SHEbuse}\ref{itm:SHEbuseadd} implies that 
\be \label{eqn:Buse_invar}
e^{\bb^{\lambda \sig}(t,0,t,y)} = \f{\int_\R e^{\bb^{\lambda \sig}(s,0,s,z)} Z_\beta(t,y \viiva s,z)\,dz }{\int_\R e^{\bb^{\lambda \sig}(s,0,s,z)} Z_\beta(t,0 \viiva s,z)\,dz },
\ee
and Theorem \ref{thm:SHEbuse}\ref{itm:SHEbuseshift} implies that for all $s < t$, 
\be \label{SHEbusehoriz}
\exp(\bb^{\lambda \sig}(t,0,t,\aabullet)) \deq \exp(\bb^{\lambda \sig}(s,0,s,\aabullet)).
\ee
The $s = t$ case of Theorem \ref{thm:SHEbuse}\ref{itm:conv} states that for $\lambda \notin \Lambda_{\bb}$ and $f \in \mathbb F_\lambda$, 
\[
\lim_{r \to -\infty} \f{\int_\R f(z) Z_\beta(t,y\viiva r,z) \,dz}{\int_\R f(z) Z_\beta(t,0|r,z) \,dz} = e^{b_\beta^\lambda(t,0,t,y)},
\]
uniformly on compact subsets of $y \in \R$. Then, using \eqref{eqn:Buse_invar}, Theorem \ref{thm:SHEbuse}\ref{itm:SHEBuseP0}, and the shift invariance of Theorem \ref{thm:Zinvar}\ref{itm:Zshift}, for any (deterministic or random) $k$-tuple of functions $(f^1,\ldots,f^k)$, so that, with probability one, each $f^i \in \mathbb F_{\lambda_i}$, as $t \to \infty$, we have the following distributional convergence on $C(\R^k,\R)$: 
\[
\biggl\{\f{Z_\beta(t,\aabullet \viiva e^{f^i})}{Z_\beta(t,0 \viiva e^{f^i})} \biggr\}_{1 \le i \le k} \Longrightarrow \{\exp(\bb^{\lambda_i}(0,0,0,\aabullet))\}_{1 \le i \le k}
\]
In particular, if for all $t > 0$,
\[
\biggl\{\f{Z_\beta(t,\aabullet \viiva e^{f^i})}{Z_\beta(t,0 \viiva e^{f^i})} \biggr\}_{1 \le i \le k} \deq \{\exp(f^i)\}_{1 \le i \le k},
\]
then $\{\exp(f^i)\}_{1 \le i \le k} \deq \{\exp(\bb^{\lambda_i}(0,0,0,\aabullet))\}_{1 \le i \le k}$.
\end{proof}

\subsection{Convergence of the O'Connell-Yor polymer to SHE} \label{sec:OCYSHE}
In his section, we show convergence of the O'Connell-Yor polymer to the Green's function of the SHE (Theorem \ref{thm:OCYtoSHE}) and prove a convergence result for the model started from initial data (Theorem \ref{lem:L1conv_bd}). The O'Connell-Yor polymer (alternatively, the Brownian polymer), first introduced in \cite{brownian_queues}, is defined as follows. On a probability space $(\Omega, \F,P)$, let $\mathbf B = (B_r)_{r \in \Z}$ be a sequence of independent, two-sided standard Brownian motions. For $(m,x) \le (n,y) \in \Z \times \R$,  define the path space
\[
\pathsp_{(m,x),(n,y)} := \{(x_{m - 1},x_m,\ldots,x_n) \in \R^{n - m + 2}: x = x_{m - 1} \le x_m \le \cdots \le x_n = y   \}.
\] For $(m,x),(n,y) \in \Z \times \R$ with $m < n$ and $x \le y$, the point-to-point partition function is defined as
\begin{equation} \label{OCYpart}
\OCY_\beta(n,y\viiva m,x)(\mathbf B) = \int_{\pathsp_{(m,x),(n,y)} }\exp \bigg\{\beta \sum_{r = m}^n B_r(x_{r - 1},x_r)\bigg\} d  x_{m:n-1}.
\end{equation}
Throughout the paper, $\beta>0$ is a positive inverse-temperature parameter.  
The superscript in $\OCY_\beta$ stands for \textit{semi-discrete}. 
 For $m = n$,  define
\[
\OCY_\beta(m,y\viiva m,x)(\mathbf B) = e^{\beta B_m(x,y)}.
\]
The argument $\mathbf B$ will often be omitted from the notation. From the definition, the Chapman-Kolmogorov equation holds: namely that, for $m < r \le n$ and $x < y$,
\be \label{Zcomp}
\OCY_\beta(n,y\viiva m,x) = \int_{x}^y \OCY_\beta(n,y\viiva r,w) \OCY_\beta(r - 1,w\viiva m,x)\,dw.
\ee

We also define a partition function with a boundary at level $m= -1$. For a random or deterministic  initial function $f:\R \to \R$ with $f(x) \to 0$ sufficiently fast as $x \to -\infty$, define, for $n \ge 0$ and $y \in \R$,
\be \label{Zf}
\OCY_\beta(n,y\viiva f) = \int_{-\infty}^y f(x) \tspb \OCY_\beta(n,y\viiva 0,x)\,dx.
\ee
For $n = -1$, define $\OCY_\beta(-1,y\viiva f) = f(y)$.

 Abbreviate the Poisson distribution as
\be\label{poi} 
q(n,y) = e^{-y} \f{y^n}{n!} \ind((n,y)\in \Z_{\ge 0} \times \R_{\ge 0}).  
\ee
For integers $n \ge m$, real numbers $y \ge x$, and $\gamma \in \R_{>0}$, set 
\be\label{Ydef}   Y_\gamma(n,y \viiva m,x) = e^{-(y - x) - \f{\gamma^2}{2} (y - x)}\OCY_\gamma(n,y \viiva m,x).\ee 
Next, define
\begin{align*}
 \delta_k(n \viiva m) &= \{m = n_0 \le n_1 \le \cdots \le n_k \le n_{k + 1} = n:  n_i \in \Z\}, \quad\text{and} \\
\Delta_k(y \viiva x) &= \{x = y_0 < y_1 < \cdots < y_k < y_{k + 1} = y: y_i \in \R\}.
\end{align*}

\begin{lemma} \label{lem:q2prod}
There exists a constant $C > 0$ so that, for all integers $k > 0$, $n \ge m$, and all real numbers $y > x$, whenever $\mbf n \in \delta_k(n \viiva m)$ and $\mbf y \in \Delta_k(y \viiva x)$,
\be \label{eq:prodqbd}
\begin{aligned}
\prod_{i = 0}^k q^2(n_{i+1} - n_i,y_{i+1} - y_i) &\le C^k g(\mbf n,\mbf y)  \\
&:= C^k \f{  e^{-2(y-x)} 2^{2(n-m)}}{\pi^{(k+1)/2}} \prod_{i = 0}^k \f{(y_{i+1} - y_i)^{2(n_{i+1} - n_i)}}{[2(n_{i+1} - n_i)]! \sqrt{(n_{i+1} - n_i) \vee 1}}.
\end{aligned}
\ee
Furthermore,  for each $\mbf n \in \delta_k(n \viiva m)$,
\be \label{eq:Dir_int}
\int_{\Delta_k(y \viiva x)} g(\mbf n,\mbf y) \prod_{i = 1}^k dy_i = \f{2^{2(n-m)} [(n-m)!]^2(y-x)^{k} q^2(n-m,y-x)}{\pi^{(k+1)/2}[2(n-m) + k]!} \prod_{i = 0}^k \f{1}{\sqrt{(n_{i+1} - n_i) \vee 1}}. 
\ee
\end{lemma}
\begin{proof}
From the definition, it follows that 
\be \label{231}
\prod_{i = 0}^k q^2(n_{i+1} - n_i,y_{i+1} - y_i) = e^{-2(y-x)} \prod_{i = 0}^k \f{(y_{i+1} - y_i)^{2(n_{i+1} - n_i)}}{[(n_{i+1} - n_i)!]^2},
\ee
and Stirling's approximation implies that for large $n$,
\[
[n!]^2 \sim \f{[2n]! \sqrt{\pi n}}{2^{2n}}
\]
In particular, there exists a constant $C$ so that $\f{1}{[n!]^2} \le C \f{2^{2n}}{[2n]!\sqrt{\pi n}}$ for $n \ge 1$. Inserting this bound into \eqref{231} proves \eqref{eq:prodqbd}. 
The integral \eqref{eq:Dir_int} is the computation of a Dirichlet integral after the change of variable $w_i = \f{y_i- y_{i-1}}{y-x}$.
\end{proof}

\begin{lemma} \label{lem:Yrep}
For $\ve \ge 0$, for all integers $n \ge m$ and real numbers $y \ge x + \ve$, the field $Y_\gamma(n,y \viiva m,x)$ satisfies the following It\^o integral equation:
\be \label{Yrep}
 Y_\gamma(n,y \viiva m,x) = \sum_{m \le k \le n} q(n-k,y-(x + \ve)) Y_\gamma(k,x+\ve \viiva m,x) + \gamma \int_{x+\ve}^y \sum_{m \le k \le n} q(n - k,y - w) \tspb Y_\gamma(k,w \viiva m,x)\,d B_k(w),
\ee
where $\{B_r\}_{r \in \Z}$ are the i.i.d.\ Brownian motions that define $\OCY_\gamma$.
\end{lemma}
\begin{remark}
We observe that in the $\ve = 0$ case, since $Y_\gamma(k,x \viiva m,x) = \ind(k = m)$, we obtain
\[
Y_\gamma(n,y \viiva m,x) = q(n-m,y-x) + \gamma \int_x^y \sum_{m \le k \le n} q(n - k,y - w) \tspb Y_\gamma(k,w \viiva m,x)\,d B_k(w).
\]
This further implies that 
\be \label{Ygamexp}
\Ee[Y_\gamma(n,y \viiva m,x)] = q(n-m,y-x).
\ee
\end{remark}
\begin{proof}
 With $Y_\gamma(n,y \viiva m,x)$ defined, let $\wt Y_\gamma(n,y \viiva m,x)$ denote the RHS of \eqref{Yrep}. We prove that $\wt Y_\gamma(n,y \viiva m,x) = Y_\gamma(n,y \viiva m,x)$ by induction on $n \ge m$. First, note that 
\begin{align*}
\wt Y_\gamma(m,y \viiva m,x) &= q(0,y - x - \ve)Y_\gamma(m,x + \ve \viiva m,x) + \gamma \int_{x+\ve}^y q(0,y - w) Y_\gamma(m,w \viiva m,x)\,dB_m(w) \\
&= e^{-(y - x) - \f{\gamma^2 }{2}\ve + \gamma B_m(x,x+\ve)}  + \gamma \int_{x+\ve}^y e^{-(y - x) - \f{\gamma^2}{2}(w - x) + \gamma B_m(x,w) }\, dB_m(w) \\
&= e^{-(y - x) + \f{\gamma^2}{2}x - \gamma B_m(x)} \Bigl(e^{-\f{\gamma^2}{2}(x+\ve) + \gamma B_m(x+\ve)} + \gamma \int_{x+\ve}^y e^{- \f{\gamma^2}{2}w + \gamma B_m(w)}\,dB_m(w)\Bigr),
\end{align*}
so the equality $Y_\gamma(m,y \viiva m,x) = \wt Y_\gamma(m,y \viiva m,x)$ reduces to 
\[
e^{-\f{\gamma^2}{2} y + \gamma B_m(y)} = e^{-\f{\gamma^2}{2} (x+\ve) + \gamma B_m(x+\ve)} +\gamma \int_{x+\ve}^y e^{- \f{\gamma^2}{2}w + \gamma B_m(w)}\,dB_m(w),
\]
which follows from It\^o's formula. Now, assume that for some $n > m$, $\wt Y_\gamma(n - 1,w \viiva m,x) = Y_\gamma(n - 1 ,w \viiva m,x)$ for all $w \ge x + \ve$.

From the definition \eqref{Ydef}, the Chapman-Kolmogorov equation \eqref{Zcomp}, and definition \eqref{Ydef} again,  
\begin{align*}
Y_\gamma(n,y \viiva m,x) &= e^{-(y - x) - \f{\gamma^2}{2} (y - x)}\int_x^y \OCY_\gamma(n - 1,w \viiva m,x) e^{\gamma B_n(w,y)}  \,dw \\[1em]
&= e^{-(y - x) - \f{\gamma^2}{2} (y - x) + \gamma B_n(y)} \int_x^y  \OCY_\gamma(n - 1,w\viiva m,x) e^{-\gamma B_n(w)}\,dw \\[1em]
&=   e^{\gamma B_n(y)} \cdot e^{-(1+\f{\gamma^2}{2})y } \int_x^y  
e^{(1+\f{\gamma^2}{2})w } \tspb Y_\gamma(n - 1,w\viiva m,x) e^{-\gamma B_n(w)}\,dw.
\end{align*}
Let $d$ denote differentiation in the real variable $y$, with $x$ fixed. An application of  It\^o's formula to the last line above gives 
\begin{align}
dY_\gamma(n,y \viiva m,x) &= [Y_\gamma(n - 1,y \viiva m,x)-Y_\gamma(n,y \viiva m,x) ]\,dy + \gamma Y_\gamma(n,y \viiva m,x)\,d B_n(y).  \label{Ydif}
\end{align}
Additionally, a simple computation shows
\be \label{qdif}
d q(n - k,y - x) = [ q(n - 1 - k,y - x)- q(n - k,y - x)  ]\,dy \qquad \text{for } n\in\Z_{\ge k} \text{ and } y\ge x, 
\ee
where we set $q(-1,w) = 0$ by convention.  Differentiate the right-hand side of \eqref{Yrep} and apply 
  \eqref{qdif} and $q(n,0) = \ind(n = 0)$  to obtain 
\begin{align}
d \wt Y_\gamma(n,y \viiva m,x) &= \sum_{m \le k \le n} dq(n - k, y - (x + \ve))Y_\gamma(k,x+\ve \viiva m,x)  + \gamma \sum_{m \le k \le n} q(n - k,0) Y_\gamma(k,y \viiva m,x)\,dB_k(y) \nonumber \\
&\qquad\qquad\qquad\qquad\qquad+ \gamma \int_{x+\ve}^y \sum_{m \le k \le n} dq(n - k,y - w)Y_\gamma(k,w \viiva m,x)\,dB_k(w) \\
&= [\wt Y_\gamma(n - 1,y \viiva m,x) - \wt Y_\gamma(n,y \viiva m,x)]\,dy + \gamma Y_\gamma(n,y \viiva m,x)\,dB_n(y).\label{Ytildeeq}
\end{align}
In the first equality above, we have used the stochastic Leibniz rule (see for example, \cite[Equation (6.2.25)]{Oskendal}). Then, comparing \eqref{Ydif} and \eqref{Ytildeeq}, the induction hypothesis implies that the process 
\[ X(y) := Y_\gamma(n,y \viiva m,x) - \wt Y_\gamma(n,y \viiva m,x)\] 
satisfies $dX = -Xdy$. Since $q(n,0) = \ind(n = 0)$, we observe further that $\wt Y_\gamma(n,x+\ve \viiva m,x) = Y_\gamma(n,x+\ve \viiva m,x)$ for $n \ge m$, so $X$ has the initial condition $X(x + \ve) = 0$. Thus, $X(y) = 0$ for all $y \ge x+\ve$, and $Y_\gamma(n,y \viiva m,x) = \wt Y_\gamma(n,y \viiva m,x)$, as desired. 
\end{proof}

We now use Lemma \ref{lem:Yrep} to write $Y_\gamma$ as an infinite series of iterated stochastic integrals.
\begin{lemma}
    Let $q$ and $Y_\gamma$ be defined as in \eqref{poi} and \eqref{Ydef}. For every $n \ge m$ and $y \ge x$, $Y_\gamma(n,y \viiva m,x)$ can be written as the following $L^2(\Pp)$-convergent infinite sum 
    \be \label{OCYchaos}
    \begin{aligned}
 Y_\gamma(n,y \viiva m,x) &= \sum_{k = 0}^\infty \gamma^k I_k(n,y \viiva m,x),\\
 I_k(n,y \viiva m,x) &=  \sum_{\delta_k(n \viiva m)} \int_{\Delta_k(y \viiva x)}  \prod_{i = 0}^{k} q(n_{i + 1} - n_{i},y_{i + 1} - y_{i}) \prod_{i = 1}^k\, dB_{n_i}(y_i),
 \end{aligned}
 \ee
 where, in the $k = 0$ case, we use this notation to mean $I_0(n,y \viiva m,x) = q(n - m, y - x)$. Furthermore, 
\be \label{Ygamma2}
\Ee[Y_\gamma(n,y \viiva m,x)^2] = \sum_{k = 0}^\infty \gamma^{2k}\Ee[I_k(n,y \viiva m,x)^2],
\ee
and there exists a universal constant $C > 0$ so that for all integers $n \ge m$ and $k \ge 0$, and real numbers $y \ge x$ and $\gamma > 0$,
\be \label{IgammaL2bd}
\Ee[I_k(n,y \viiva m,x)^2] \le  \f{C^k (y-x)^k q^2(n-m, y-x)  }{(n-m)^{k/2} \Gamma((k+1)/2)}.
\ee
\end{lemma}
\begin{proof}
Picard iteration of \eqref{Yrep} in Lemma \ref{lem:Yrep} in the case $\ve = 0$ gives the expansion \eqref{OCYchaos}, assuming that the series is convergent.  By independence of the $B_k$, the fact that It\^o integrals have mean $0$, and the It\^o  isometry, we have that 
\[
\Ee[I_k(n,y \viiva m,x) I_j(n,y \viiva m,x)] = \delta_{j = k} \sum_{\delta_k(n \viiva m)} \int \displaylimits_{\Delta_k(y \viiva x)} \prod_{i = 0}^{k} q^2(n_{i+1} - n_{i},y_{i+1} - y_{i}) \prod_{i = 1}^k dy_i,
\]
 Hence, as long as the sum on the right-hand side of \eqref{Ygamma2} is convergent, the expansion \eqref{OCYchaos} is $L^2(\Pp)$ convergent, and \eqref{Ygamma2} holds. For this, it suffices to show \eqref{IgammaL2bd}, and it further suffices to show the $m = x = 0$ case by translation invariance. 
For shorthand notation, set $Y_\gamma(n,y) = Y_\gamma(n,y \viiva 0,0)$. Then, by Lemma \ref{lem:q2prod} and Stirling's approximation, there exists a constant $C > 0$ (possibly changing from line to line) so that
\begin{equation} \label{eq:Ikbd1}
\begin{aligned}
    \f{\Ee[(I_{k}(n,y))^2]}{q^2(n,y)} &\le \f{1}{q^2(n,y)} \sum_{\delta_k(n)} \int \displaylimits_{\Delta_k(y)} \prod_{i = 0}^{k } q^2(n_{i+1} - n_{i},y_{i+1} - y_{i}) \prod_{i = 1}^k dy_i \\
    &\le C^k  \f{2^{2n} [n!]^2 y^k }{[2n + k]!} \sum_{\delta_k(n)} \prod_{i = 0}^k\f{1}{\sqrt{(n_{i+1} - n_i) \vee 1}} \\
    &\le C^k  y^k \f{n^{2n+1} 2^{2n}}{(2n + k)^{2n + k + 1/2}}\sum_{\delta_k(n)} \prod_{i = 0}^k\f{1}{\sqrt{(n_{i+1} - n_i) \vee 1}}\\
    &\le \f{C^k y^k n^{1/2}}{(2n+k)^{k}} \sum_{\delta_k(n)} \prod_{i = 0}^k\f{1}{\sqrt{(n_{i+1} - n_i) \vee 1}},
\end{aligned}
\end{equation}
where in the last line, we made the observation that $\f{n^{2n + 1/2}2^{2n}}{(2n + k)^{2n + 1/2}}\le 1$. 
Now, observe that
\begin{equation} \label{eq:Dirnm}
\begin{aligned}
&\quad \,\sum_{\delta_k(n)} \prod_{i = 0}^k\f{1}{\sqrt{(n_{i+1} - n_i) \vee 1}} = \sum_{m_i \ge 0, \sum_{i = 0}^{k} m_i = n} \prod_{i = 0}^k\f{1}{\sqrt{m_i\vee 1}} \\
&= \int_{\R^k} \prod_{i = 1}^k\Biggl(\f{\ind(t_i > 0) }{\sqrt{\lfloor t_i \rfloor \vee 1}}\Biggr) \f{\ind(\sum_{i = 1}^k \lfloor t_i \rfloor \le n)}{\sqrt{(n - \sum_{i = 1}^k \lfloor t_i \rfloor)\vee 1}}  \prod_{i = 1}^k dt_i.
\end{aligned}
\end{equation}
Make the change of variables $t_i = s_i(n+k)$ to obtain
\begin{equation} \label{eq:sqint}
\f{(n+k)^k}{n^{(k+1)/2}} \int_{\R^k} \prod_{i = 1}^k \Biggl(  \sqrt{\f{n \ind(s_i > 0)}{\lfloor s_i(n+k) \rfloor \vee 1}}\Biggr) \sqrt{ \f{n \ind( \sum_{i = 1}^k \lfloor s_i(n+k) \rfloor \le n)}{\bigl(n - \lfloor s_i(n+k) \rfloor\bigr) \vee 1}}  \prod_{i = 1}^k \, ds_i.
\end{equation}
We now bound the integrand above. Since $s_i > 0$, we have 
\begin{equation} \label{eq:sibd}
s_i(n + k) \vee 1 \ge (s_i n - 1) \vee 1 \ge \f{s_i n}{2},
\end{equation}
where in the last inequality, we have used $(x-\ell) \vee 1 \ge \f{x}{\ell + 1}$. 

Next, observe that 
\[
\sum_{i = 1}^k \lfloor s_i(n+k) \rfloor  \ge -k + \sum_{i = 1}^k s_i(n + k),
\]
so 
\begin{equation} \label{sin2}
\sum_{i = 1}^k  \lfloor s_i(n+k) \rfloor \le n \Longrightarrow \sum_{i = 1}^k s_i \le 1.
\end{equation}
When this occurs, we have 
\[
\sum_{i = 1}^k \lfloor s_i(n+k) \rfloor \le  \sum_{i = 1}^k s_i n + k \sum_{i = 1}^ks_i \le k + \sum_{i = 1}^k s_i n
\]
so
\begin{equation} \label{sin3}
(n - \lfloor s_i(n+k) \rfloor ) \vee 1 \ge (n -k -  \sum_{i = 1}^k s_i n) \vee 1 \ge \f{n - \sum_{i = 1}^k s_i n}{k+1},
\end{equation}
where we have again used the bound $(x - \ell) \vee 1 \ge \f{x}{\ell + 1}$ for $\ell = k$. Combining \eqref{eq:Dir_int},\eqref{eq:sqint},\eqref{sin2}, and \eqref{sin3}, we obtain 
\begin{align*}
&\quad \,\sum_{\delta_k(n)} \prod_{i = 0}^k\f{1}{\sqrt{(n_{i+1} - n_i) \vee 1}}  \\
&\le \f{2^{k/2} \sqrt{k+1} (n+k)^k}{n^{(k+1)/2}} \int_{\R^k} \prod_{i = 1}^k \Bigl(\f{\ind(s_i > 0)}{\sqrt {s_i}}\Bigr) \f{\ind(\sum_{i = 1}^k s_i < 1)}{\sqrt{1 - \sum_{i = 1}^k s_i}} \prod_{i = 1}^k ds_i  \le  \f{C^k (n+k)^k}{n^{(k+1)/2} \Gamma((k+1)/2)}.    
\end{align*}
Substituting this bound back into \eqref{eq:Ikbd1}, we obtain
\begin{align*}
\f{\Ee[(I_{k}(n,y))^2]}{q^2(n,y)} \le  C^k  y^k \f{(n+k)^k }{(2n + k)^{k}n^{k/2} \Gamma((k+1)/2) } \le\f{ C^k  y^k }{n^{k/2} \Gamma((k+1)/2) } ,
\end{align*}
as desired. 
\end{proof}

\begin{lemma} \label{lem:Ywhitenoise}
Given a space-time white noise $W$, one can couple the field of i.i.d. Brownian motions $\{B_r\}_{r \in \Z}$ with $W$ so that
\be \label{Ygamma series integral}\begin{aligned} 
    Y_\gamma(n,y \viiva m,x)  &= \sum_{k = 0}^\infty \gamma^k \int_{\R^k} \int_{\R^k}
    \prod_{i = 0}^k q(\lfloor t_{i + 1} \rfloor - \lfloor t_i \rfloor,y_{i +1} - y_i) \prod_{i = 1}^k W(dt_i,dy_i),
    \end{aligned}\ee
    where we define  
    $t_0 = m, t_{k + 1} = n, y_0 = x$, and $y_{k+1} = y$.
\end{lemma}
\begin{proof}
Given a space-time white noise $W$, we can define a field of i.i.d.\ two-sided Brownian motions $\{B_r\}_{r \in \Z}$ by
    \[ 
    B_r(y) = \begin{cases}
    W(\ind([r,r + 1] \times [0,y]))  & y \ge 0 \\
    -W(\ind([r,r + 1] \times [y,0])) &y < 0 
    \end{cases}
    \]
    Alternatively, we can use a single definition using the formal equality
    \be \label{Brdef}
    B_r(y) = \int_{0}^y \,dx \int_r^{r + 1}dt\, W(t,x). 
    \ee
      From the definition of $B_r$, we have the formal equality
    \be \label{Brderiv}
    d B_r(x) = \,dx \int_r^{r+1} dt \, W(t,x).
    \ee
    
    Now, with the $B_r$ defined in terms of $W$, we show that we can write $Y_\gamma$  as \eqref{Ygamma series integral}. Using \eqref{Brderiv}, the $k = 1$ term in \eqref{OCYchaos} can be written as 
    \begin{align*}
        &\quad \; \gamma \int_x^y \sum_{r = m}^n q(n - r,y - z)q(r - m,z - x) \,dB_r(z) \\
        &= \gamma \int_x^y \sum_{r = m}^n q(n - r,y - z)q(r - m,z - x) \,dz\int_r^{r + 1} dt \,W(t,z) \\
        &= \gamma \int_x^y \,dz \sum_{r = m}^n \int_r^{r + 1} dt\, q(n - \lfloor t \rfloor,y - z)q(\lfloor t \rfloor - m,z - x) W(t,z) \\
        &= \gamma \int_x^y\,dz \int_m^{n + 1}dt \,q(n - \lfloor t \rfloor,y - z)q(\lfloor t \rfloor - m,z - x)W(t,z) \\
        &= \gamma \int_{\R} \,dz \int_\R \,dt \,q(n - \lfloor t \rfloor,y - z)q(\lfloor t \rfloor - m,z - x)W(t,z),
    \end{align*}
    and this matches the $k = 1$ term of \eqref{Ygamma series integral}. The last line follows because the integrand is $0$ outside the bounds of integration on the previous line. The general case follows using the same reasoning and induction. 
\end{proof}

We prove an intermediate lemma for a scaled transition function. 
With $q$ as in \eqref{poi}, set 
\be \label{pNdef}
p_N(t,y \viiva s,x) = \sqrt N q\bigl(\fl{tN} - \fl{sN},(t-s)N + \sqrt N (y-x)\bigr)\qquad\text{and}\qquad p_N(t,y) = p_N(t,y \viiva 0,0).
\ee

\begin{lemma} \label{lem:moment_bd}
The following hold.
\begin{enumerate} [label=\rm(\roman{*}), ref=\rm(\roman{*})]  \itemsep=3pt
\item \label{itm:pNptwise} As $N \to \infty$, $p_N(t,y \viiva s,x) \to \rho(t -s,y-x )$, pointwise, for $x,y \in \R$ and $t > s$.
\item \label{itm:intconv} For each $t > 0$, $y \in \R$, $\alpha > 0$, and integer $M \ge 1$,
\be \label{eqn:intc}
\lim_{N \to \infty} \int_{-\infty}^{t\sqrt N + y} e^{\alpha |x|}p_N^M(t,y \viiva 0, x)\,dx = \int_\R e^{\alpha |x|}\rho^M(t,y - x)\,dx < \infty.
\ee
\end{enumerate}
\end{lemma}
\begin{proof}
\smallskip \noindent \textbf{Item \ref{itm:pNptwise}:} The pointwise convergence $p_N(t,y \viiva s,x) \to \rho(t-s,y-x)$ is a simple application of Stirling's approximation. We prove the $x = s = 0$ case to avoid clutter, but the general case is entirely similar.  
\be \label{stirlingpN}
\begin{aligned}    p_N(t,y \viiva s,x) &= \sqrt N e^{-N(t-s) + \sqrt N (y-x)} \Bigl(\f{(N(t-s) + \sqrt N (y-x))^{\lfloor tN \rfloor}}{(\lfloor tN \rfloor )!}\Bigr) \\
    &\sim \f{\sqrt N}{\sqrt{2\pi \lfloor tN \rfloor}} e^{-tN + \lfloor tN \rfloor + \sqrt N y} \Bigl(\f{tN + \sqrt N y}{\lfloor tN \rfloor}\Bigr)^{\lfloor tN \rfloor} \sim p(t,y),
\end{aligned}
\ee
where the last step follows from the Taylor expansion
\begin{align*}
\lfloor tN \rfloor \log \Bigl(\f{tN + \sqrt N y}{\lfloor tN \rfloor}\Bigr) &= \lfloor tN \rfloor \log \Bigl(1 + \f{tN + \sqrt N y - \lfloor tN \rfloor}{\lfloor tN \rfloor}\Bigr) \\
&= tN - \sqrt N y - \lfloor tN \rfloor - \f{\lfloor tN \rfloor}{2} \Bigl(\f{tN + \sqrt N y - \lfloor tN \rfloor}{\lfloor tN \rfloor}\Bigr)^2 + O(N^{-1/2}) \\
&=  tN - \sqrt N y - \lfloor tN \rfloor - \f{y^2}{2t} + O(N^{-1/2})
\end{align*}

\smallskip \noindent \textbf{Item \ref{itm:intconv}:} Recall the convention $p_N(t,y) = p_N(t,y \viiva 0,0)$.  Changing variables, \eqref{eqn:intc} is equivalent to
\[
\lim_{N \to \infty} \int_{-\infty}^{t\sqrt N} e^{\alpha|x + y|}p_N^2(t,-x)\,dx = \int_\R e^{\alpha|x + y|}\rho^2(t,-x)\,dx
\]
We prove this by showing separately that
\be \label{int1c}
\int_{-y}^{t\sqrt N} e^{\alpha (x + y)} p_N^M(t,-x)\,dx \to \int_{-y}^\infty e^{\alpha (x + y)}\rho^M(t,-x)\,dx,
\ee
and
\be \label{int2c}
\int_{-\infty}^{-y} e^{-\alpha (x + y)} p_N^M(t,-x)\,dx \to \int_{-\infty}^{-y} e^{-\alpha (x + y)} \rho^M(t,-x)\,dx.
\ee

First, by completing the square and changing variables, we obtain 
\begin{align*}
\int_{-y}^\infty e^{\alpha x}\rho^M(t,-x)\,dx &= \f{\sqrt{2t} e^{\alpha^2t/(2M)} }{\sqrt M (2\pi t)^{M/2}}  \int_{\sqrt{\f{M}{2t}}(-y - \f{\alpha t}{M})}^\infty e^{-u^2}\,du \\
&= \f{e^{\alpha^2t/(2M)} }{2\sqrt M (2\pi t)^{(M-1)/2}}\erfc\Bigl(\sqrt{\f{M}{2t}}(-y - \alpha t/M)\,\Bigr).
\end{align*}
On the other hand,
\begin{align*}
    &\quad \int_{-y}^{t\sqrt N } e^{\alpha x} p_N^M(t,-x)\,dx 
    = N^{M/2}\int_{-y}^{t\sqrt N } e^{\alpha x} e^{-M(tN -x\sqrt N\,)}\f{(tN - x\sqrt N)^{M\fl{tN}}}{(\fl{tN}!)^M}\,dx,
\end{align*}
which, upon the transformation $w = (tN - x\sqrt N)(\alpha/\sqrt N + M)$, we obtain
\begin{align} \label{int3c}
&\quad \; \f{N^{(M-1)/2}e^{\alpha t \sqrt N}}{(\lfloor tN \rfloor !)^M (\alpha/\sqrt N + M)^{M\lfloor tN \rfloor + 1}} \int_0^{(tN + y\sqrt N)(\alpha/\sqrt N + M)} e^{-w}w^{M \lfloor tN \rfloor}\,dw \\
&= \f{N^{(M-1)/2}e^{\alpha t \sqrt N}(M\lfloor tN\rfloor)!}{(\lfloor tN \rfloor !)^M (\alpha/\sqrt N + M)^{M\lfloor tN \rfloor + 1}} \f{\gamma(M \lfloor tN \rfloor + 1,(tN + y\sqrt N)(\alpha/\sqrt N + M))}{(M\lfloor tN\rfloor)!}
\end{align}
where $\gamma(s,x) = \int_0^x e^{-u} u^{s - 1}\,du$ is the lower incomplete gamma function. Tricomi \cite{Tricomi-1950} showed that as $a \to \infty$, the function $\gamma$ has the following asymptotic expansion that holds uniformly on compact subsets of $z$ (see also \cite{Temme-1975}):
\be \label{gammaexp}
\f{\gamma(a + 1,a + z(2a)^{1/2})}{\Gamma(a + 1)} \sim \f{1}{2}\erfc(-z) + o(1).
\ee
Inserting this asymptotic into \eqref{int3c} and using Stirling's approximation, we obtain
\begin{align*}
&\int_0^{t\sqrt N} e^{\alpha x}p_N^M(t,-x)\,dx \sim \f{N^{(M-1)/2}e^{\alpha t \sqrt N}(M\lfloor tN\rfloor)!}{2(\lfloor tN \rfloor !)^M (\alpha/\sqrt N + M)^{M\lfloor tN \rfloor + 1}}\erfc\Bigl(\sqrt{\f{M}{2t}}(-y - \alpha t/M)\,\Bigr) \\
&\sim \f{N^{(M-1)/2}e^{\alpha t\sqrt N}\sqrt{2\pi M \lfloor tN \rfloor}(M \lfloor tN \rfloor/e)^{M \lfloor tN \rfloor} }{2 (2\pi \lfloor tN \rfloor)^{M/2}} \\
&\qquad\qquad\qquad \times (\lfloor tN\rfloor/e)^{M \lfloor tN \rfloor} M (\alpha/\sqrt N + M)^{M \lfloor tN \rfloor} \erfc\Bigl(\sqrt{\f{M}{2t}}(-y - \alpha t/M)\,\Bigr) \\
&\sim \f{e^{\alpha t \sqrt N}}{2\sqrt M(2\pi t)^{(M-1)/2}}\Bigl(1 + \f{\alpha}{M \sqrt N}\Bigr)^{M \lfloor tN \rfloor} \sim \f{e^{\alpha^2t/(2M)} }{2\sqrt M (2\pi t)^{(M-1)/2}}\erfc\Bigl(\sqrt{\f{M}{2t}}(-y - \alpha t/M)\,\Bigr).
\end{align*}
The last step comes from 
\[
t\alpha \sqrt N - M\lfloor tN\rfloor  \log\Bigl(1 + \f{\alpha}{M\sqrt N}\Bigr) = t\alpha \sqrt N - M\lfloor tN\rfloor \Bigl(\f{\alpha}{M\sqrt N} - \f{\alpha^2}{2M^2 N} + o(N^{-1})\Bigr) = \f{\alpha^2 t}{2M} + o(1).
\]
This proves \eqref{int1c}. The proof of \eqref{int2c} is similar: the left-hand side is transformed into an incomplete gamma function via the transformation $w = (M-\alpha/\sqrt N)(Nt - x\sqrt N)$. In this case, we are left with a gamma function minus an incomplete gamma function, and the asymptotic expansion \eqref{gammaexp} gives us the needed asymptotics.
\end{proof}

We introduce the scaled   O'Connell-Yor polymer partition function, whose convergence to the fundamental solution of SHE is proved next. 
For $\beta > 0$ and a sequence $\beta_N$ such that $N^{1/4} \beta_N \to \beta$,  define a scaling factor
    \begin{align} 
    &\psi_N(s,t,x,y;\beta_N) =\sqrt N \exp\Bigl(-N\bigl(1+ \f{\beta_N^2}{2}\bigr)(t - s) - \sqrt N \bigl(1 + \f{\beta_N^2}{2}\bigr)(y - x)\Bigr) \label{psi}   
    \end{align}
     and the scaled partition function
  \be \label{ZNshort}
     \begin{aligned} 
Z_N(t,y\viiva s,x) &=  \psi_N(s,t,x,y;\beta_N) \OCY_{\beta_N}\bigl(\fl{tN},tN + y\sqrt N \,\bviiva \fl{sN},sN 
+ x\sqrt N \,\bigr) \tspa\ind\{x \le (t-s)\sqrt N + y\}\\
&= \sqrt N\,   Y_{\beta_N}\bigl(\lfloor tN\rfloor ,tN + y\sqrt N \bviiva \lfloor sN\rfloor, sN + y\sqrt N \tspb \bigr). \end{aligned}\ee
 We use representation \eqref{Ygamma series integral} in terms of white noise for $Z_N(t,y\viiva s,x)$ and then scale the white noise suitably to 
relate $Z_N(t,y\viiva s,x)$ to $Z_\beta(t,y\viiva s,x)$.  This produces for each $N$ a  coupling of   $Z_N(t,y\viiva s,x)$ and $Z_\beta(t,y\viiva s,x)$ on the probability space of the white noise. We show that in this coupling, their $L^2$ distance converges to zero.

For the next proofs, recall  the standard fact from analysis known as the generalized dominated convergence theorem: if $f_n \to f$ a.e., $|f_n| \le g_n\to g$ a.e., 
and $\int g_n \to \int g < \infty$, then $\int f_n \to \int f$.

\begin{theorem}\label{thm:OCYtoSHE}
 Fix $\beta > 0$ and a sequence $\beta_N$ such that $N^{1/4} \beta_N \to \beta$.  For each $N$ we have  a coupling of $Z_N$ and $Z_\beta$ on the probability space of the white noise so that this limit holds:  \[ \lim_{N\to\infty} \Ee\bigl[\,\bigl\lvert Z_N(t,y\viiva s,x) \,-\,  Z_\beta(t,y\viiva s,x) \bigr\rvert^2\,\bigr]=0\qquad\text{for each $s < t$ and $x,y \in \R$}.\]
In particular, the weak convergence $Z_N(t,y\viiva s,x) \Longrightarrow Z_\beta(t,y\viiva s,x)$ holds for each $s < t$ and $x,y \in \R$.
\end{theorem}

\begin{remark}
    We sketch here how our result is consistent with that in \cite[Theorem 1.2]{Nica-2021}. An independent proof follows. Ours gives the result for the four-parameter field, while \cite{Nica-2021} handles the two-parameter case. A change of coordinates is required to transfer between the two results, as shown in the discussion below. 
    We note that \cite[Theorem 1.2]{Nica-2021} is a more general result about partition functions for $d$ nonintersecting paths, while we only handle the $d = 1$ case.  There, the  semi-discrete partition function is rescaled by the Lebesgue volume of the path space $\pathsp_{(0,0),(n,x)}$: 
\[
\wt Z^{\rm sd}_\beta(n,x) = \f{n!}{x^n}\OCY_\beta(n,x\viiva 0,0).
\]
Theorem 1.2 of \cite{Nica-2021} states that for any sequence $\beta_N$ with $N^{1/4} \beta_N \to \beta$ as $N \to \infty$, we have the following convergence in distribution as $N \to \infty$: 
\[
\wt Z^{\rm sd}_{\beta_N} (\lfloor tN + x \sqrt N \rfloor,tN)\exp\Bigl(-\f{\beta_N^2}{2}tN \Bigr) \Longrightarrow \f{Z_\beta(t,x\viiva 0,0)}{\rho(t,x)}.
\]
Furthermore, as a process indexed by  $(t,x)\in(0,\infty)\times \R$, the convergence holds in the sense of finite-dimensional distributions, and there exists a coupling in which the convergence is in $L^p$ for any $p \ge 1$. Using Stirling's approximation, one can directly apply this result to show that, in the $s = x = 0$ case, the following convergence holds in $L^2(\Pp)$:
\be \label{ConvtoZb}
\sqrt N \exp\Bigl(-N\bigl(1 + \f{\beta_N^2}{2}\bigr)(t-s)\Bigr)\OCY_{\beta_N}\bigl(\lfloor tN + y\sqrt N \rfloor,tN\bviiva \lfloor sN + x\sqrt N \rfloor,sN\bigr) \Longrightarrow Z_\beta(t,y\viiva s,x).
\ee
It is reasonable to think this would extend to general $s < t$ and $x,y \in \R$, but some additional justification would be needed since the shift invariance does not immediately work through the floor functions. To see how the result in Theorem \ref{thm:OCYtoSHE} appears from \eqref{ConvtoZb}, replace $t$ with $t - \f{y}{\sqrt N}$ and $s$ with $s - \f{x}{\sqrt N}$, then replace $x$ with $-x$ and $y$ with $-y$ and use the reflection invariance of $Z_\beta$ (Theorem \ref{thm:Zinvar}\ref{itm:Zrefl}). To make this argument directly rigorous, one would need to show uniform convergence on compact sets, or change the parameterization and show that the chaos series still converges.  We emphasize here that our proof below is self-contained, uses different methods, and does not rely on the result of \cite{Nica-2021}, although the white noise coupling is the same.
\end{remark}

\begin{proof}[Proof of Theorem \ref{thm:OCYtoSHE}]
    With $Z_N(t,y \viiva s,x)$ from \eqref{ZNshort} and  for  a sequence $\beta_N$ with $N^{1/4} \beta_N  \to \beta$, Lemma \ref{lem:Ywhitenoise} implies
    \begin{align*}
    Z_N(t,y \viiva s,x) &= \sqrt N\tspb  Y_{\beta_N}\bigl(\lfloor tN\rfloor ,tN + y\sqrt N \bviiva \lfloor sN\rfloor, sN + y\sqrt N \bigr) \\
    &= \sqrt N \sum_{k = 0}^\infty \beta_N^k \int_{\R^k} \int_{\R^k} \prod_{i = 0}^k q(\lfloor t_{i + 1} \rfloor - \lfloor t_i \rfloor,y_{i +1} - y_i)\, W(dt_i,dy_i),
    \end{align*}
    where $t_0 = \lfloor sN \rfloor, t_{k+1} = \lfloor tN \rfloor, y_0 =sN + x\sqrt N$, and $y_k = tN + y\sqrt N$.
    Now, consider the transformation 
    \be \label{WNtrans}
    (t_i,y_i)_{1 \le i \le k} \mapsto \Phi_k((t_i,y_i)_{1 \le i \le k}) = \Bigl(\f{t_i}{N}, \f{y_i - t_i}{\sqrt N}    \Bigr)_{1 \le i \le k}.
    \ee
    This transformation alters the white noise, but multiplying by the square-root Jacobian term $N^{3k/4}$, we have the following distributional equality on the level of processes in $(s,x,t,y) \in \Rup$ (note that the transformation does not depend on the choice of $s,x,t,y$):  
    \begin{align}
    &\quad \; Z_N(t,y \viiva s,x) 
    \deq\sum_{k = 0}^\infty (N^{1/4}\beta_N)^k J_k^N(t,y \viiva s,x) \nonumber \\ 
    &:=  \sum_{k = 0}^\infty (N^{1/4}\beta_N)^k N^{(k+1)/2} \int_{\R^k}\int_{\R^k} \prod_{i = 0}^k q\bigl(\lfloor Nt_{i + 1} \rfloor - \lfloor N t_i \rfloor, N(t_{i +1} - t_i) + \sqrt N(y_{i + 1} - y_i)\bigr) \prod_{i = 1}^k W(dt_i,dy_i) \nonumber \\
    &= \sum_{k = 0}^\infty (N^{1/4}\beta_N)^k  \int_{\R^k} 
    \int_{\R^k}  \prod_{i = 0}^k p_N(t_{i + 1},y_{i + 1} \viiva t_i,y_i) \prod_{i = 1}^k W(dt_i\,dy_i) \label{prelimit_chaos}
    \end{align}
    where $p_N$ is defined in \eqref{pNdef}, and we define $t_0 = s, t_{k+1} = t, y_0 = x$, and $y_{k+1} = y$.
    We recall that  $N^{1/4} \beta_N \to \beta$. Since $q(n,y) = 0$ for $n < 0$ or $y < 0$, the integrand of the $k$th term in \eqref{prelimit_chaos} is supported on the set
    \begin{align*} 
     A_k(N,s,t,y,x):= \Bigl\{ 
     & sN + x \sqrt N \le t_i N + y_i \sqrt N \le t_{i + 1}N + y_{i + 1}\sqrt N\tspb, \; \\
     &\qquad\quad  \text{and}\;  \f{\fl{sN}}N \le  t_i \le \f{\lfloor t_{i + 1}N \rfloor + 1}{N}, \;  1 \le i \le k\Bigr\}.
    \end{align*} 

    The chaos series \eqref{prelimit_chaos} is the version of $Z_N(t,y \viiva s,x)$ that we couple with the SHE through the common white noise.  It is compared with 
      the chaos series \eqref{SHEchaos} of $Z_\beta$:
     \be \label{SHEchaos_rep}
    Z_\beta(t,y\viiva s,x) = \sum_{k = 0}^\infty \beta^k J_k(t,y \viiva s,x) := \sum_{k = 0}^\infty \beta^k \int_{\R^k}\int_{\R^k} \prod_{i = 0}^k \rho(t_{i + 1} - t_i,y_{i +1} - y_i) \prod_{i = 1}^k W(dt_i,d y_i),
    \ee
    where the integrand of the $k$th term is supported on the set $\Delta_k(t \viiva s)$ where $s < t_i < t_{i+1}$ for $1 \le i \le k$.
     We seek to show that for fixed $(s,x,t,y) \in \Rup$,   $\lim_{N \to \infty} \|Z_N(t,y \viiva s,x) - Z_\beta(t,y \viiva s,x)\|_{L^2(\Pp)}=0$. 
     We note that for any integer $K_0 \ge 0$,
\be
\begin{aligned} \label{L2decomp}
&\Bigl\|Z_N(t,y \viiva s,x) - Z_\beta(t,y \viiva s,x) \Bigr\|_{L^2(\Pp)} \\
 &\le \sum_{k = 0}^{K_0} \Bigl\| (N^{1/4} \beta_N)^k J_k^N(t,y \viiva s,x) - \beta^k J_k(t,y \viiva s,x)\Bigr\|_{L^2(\Pp)}  \\
 &\qquad\qquad + \Bigl\|\sum_{k = K_0 + 1}^\infty (N^{1/4} \beta_N)^k J_k^N(t,y \viiva s,x)  \Bigr\|_{L^2(\Pp)} + \Bigl\|\sum_{k = K_0 + 1}^\infty \beta^k J_k(t,y \viiva s,x)   \Bigr\|_{L^2(\Pp)}.
 \end{aligned}
 \ee
     Since the series for $Z_\beta$ is almost surely convergent in $L^2(\Pp)$, for any $\ve > 0$, there exists $K_0 \ge 0$ so that 
     \be \label{Zbsumsmall}
     \Bigl\|\sum_{k = K_0 + 1}^\infty \beta^k J_k(t,y \viiva s,x) \Bigr\|_{L^2(\Pp)} < \ve.
     \ee
      Then, by \eqref{eq:Ikfbd}, reversing the transformation \eqref{WNtrans}, and dividing by the $N^{3k/2}$ Jacobian term, we get
     \begin{align}
         &\quad \; \Bigl\|J_k^N(t,y \viiva s,x)  \Bigr\|_{L^2(\Pp)}^2 \le \int_{\R^k}\int_{\R^k} \prod_{i = 0}^k [p_N(t_{i+1},y_{i+1} \viiva t_i,y_i)]^2 \prod_{i = 1}^k dt_i dy_i  \nonumber \\
         &= N^{-k/2 + 1}\sum_{\delta_k(\lfloor tN \rfloor \viiva \lfloor sN\rfloor )}\int_{\Delta_k(tN + y \sqrt N \viiva sN + x \sqrt N)} \prod_{i = 0}^k q^2( n_{i+1} - n_i, y_{i+1} - y_i) \prod_{i = 1}^k dy_i \label{pN2_integral} \\
        &\overset{\eqref{IgammaL2bd}}{\le} C^k N^{-k/2 + 1} \f{((t-s)N + (y-x)\sqrt N)^{k} q^2(\lfloor tN \rfloor - \lfloor sN \rfloor \viiva (t-s)N + (y-x)\sqrt N)}{(\lfloor tN\rfloor - \lfloor sN \rfloor)^{k/2} \Gamma((k+1)/2)} \nonumber \\
        &\le \f{C^k [p_N(t,y \viiva s,x)]^2 }{\Gamma((k+1)/2)} \le \f{C^k}{\Gamma((k+1)/2)} \nonumber 
     \end{align}
     where the constant $C$ changes from line to line and depends on the fixed parameters $x,y$ and $s < t$, but not on $N$. The last inequality follows from the pointwise convergence $p_N(t,y \viiva s,x) \to \rho(t-s,y-x)$ (Lemma \ref{lem:moment_bd}\ref{itm:pNptwise}). Then, using orthogonality of each chaos for different values of $k$, there exists $K_0$ sufficiently large so that for all $N \ge 1$,
     \be \label{N_term_tail}
     \begin{aligned}
     \Bigl\|\sum_{k = K_0 + 1}^\infty (N^{1/4} \beta_N)^k J_k^N(t,y \viiva s,x)  \Bigr\|_{L^2(\Pp)}^2 &= \sum_{k = K_0 + 1}^\infty  (N^{1/4} \beta_N)^{2k} \|J_k^N(t,y \viiva s,x) \|_{L^2(\Pp)}^2  \\
     &\le \sum_{k = K_0 + 1}^\infty \f{C^k}{\Gamma((k+1)/2)} < \ve.
     \end{aligned}
     \ee
     Then, combining \eqref{L2decomp},\eqref{Zbsumsmall}, and \eqref{N_term_tail}, and recalling that $N^{1/4}\beta_N \to \beta$, the proof is complete once we show that, for each $k \ge 0$,
    \[
    \limsup_{N \to \infty} \Bigl\|J_k^N(t,y \viiva s,x) - J_k(t,y \viiva s,x)\Bigr\|_{L^2(\Pp)} = 0.
    \]
    When $k = 0$, this is simply the convergence of the nonrandom quantity $[p_N(t,y\viiva s,x)]^2$ to $\rho^2(t-s,y-x)$, which is Lemma \ref{lem:moment_bd}\ref{itm:pNptwise}. Thus, we take $k \ge 1$ in the sequel.
    We again use \eqref{eq:Ikfbd}. That is,
    \be \label{L2diffWN}
    \begin{aligned}
        &\quad \;\Bigl\|J_k^N(t,y \viiva s,x) - J_k(t,y \viiva s,x)\Bigr\|_{L^2(\Pp)}^2  \\
        &\le \int_{\R^k} \int_{\R^k} \Biggl(\prod_{i = 0}^k p_N(t_{i+1},y_{i+1} \viiva t_i,y_i) - \prod_{i = 0}^k \rho(t_{i+1} - t_i, y_{i+1}-y_i)\Biggr)^2 \prod_{i = 1}^k dt_i dy_i.
    \end{aligned}
    \ee
    Recall that $\prod_{i = 0}^k p_N(t_{i+1},y_{i+1} \viiva t_i,y_i)$ is supported on the set $A_k(N,s,t,y,x)$, while $\prod_{i = 0}^k \rho(t_{i+1} - t_i , y_{i+1} - y_i)$ is supported on the set where $t_{i+1} > t_i$ for all $i$. By Lemma \ref{lem:moment_bd}\ref{itm:pNptwise}, the integrand in \eqref{L2diffWN} converges to $0$ Lebesgue-a.e.   Expand the square, drop the cross term, and use   \eqref{eq:prodqbd} of Lemma \ref{lem:q2prod} to conclude that   the integrand in \eqref{L2diffWN} is bounded by a $k$-dependent constant times
    \begin{align} 
    &\quad \; \f{N^{k+1}}{\pi^{(k+1)/2}} e^{-2[(t-s)N +(y-x)\sqrt N]} 2^{2(\lfloor tN \rfloor - \lfloor sN \rfloor)} \prod_{i = 0}^k \f{((t_{i+1} - t_i)N +(y_{i+1} - y_i)\sqrt N)^{2(\lfloor t_{i+1}N \rfloor - \lfloor t_{i}N \rfloor)}}{[2(\lfloor t_{i+1}N \rfloor - \lfloor t_i N \rfloor)]! \sqrt{(\lfloor t_{i+1}N \rfloor - \lfloor t_i N \rfloor) \vee 1}}  \label{856} \\
&+\prod_{i = 0}^k \rho^2(t_{i+1} - t_i, y_{i+1}-y_i) \label{857}
    \end{align}
     A Stirling's approximation computation nearly identical to that in the proof of Lemma \ref{lem:moment_bd}\ref{itm:pNptwise}  shows that the term in \eqref{856} converges pointwise to the term in \eqref{857}. By the generalized dominated convergence theorem, it then suffices to show that the integral over $A_k(N,s,t,y,x)$  of the term in \eqref{856} converges as $N \to \infty$ to 
     \be \label{rhointcomp}
     \begin{aligned}
     &\quad \; \int_{\R^k} \int_{\R^k} \prod_{i = 0}^k \rho^2(t_{i+1} - t_i, y_{i+1}-y_i) \prod_{i = 1}^k dt_i dy_i \\
     &= \f{\sqrt{t-s}}{2^{k}\pi^{k/2}} \rho^2(t-s,y-x) \int_{\R^k} \ind_{B_k} \Biggl(\prod_{i = 1}^{k}\,dt_i\f{1}{\sqrt t_i}\Biggr) \f{1}{\sqrt{t - s - \sum_{i = 1}^k t_i}} < \infty,
     \end{aligned}
     \ee
     where 
    \[
    B_k = \{t_i > 0, 1 \le i \le k, \sum_{i = 1}^{k} t_i < t-s\}.
    \]
     The equality above comes  as follows.  To compute the integral on the left in \eqref{rhointcomp}, write the integrand as 
     \[
        \f{1}{2^{k+1}\pi^{(k+1)/2}}\prod_{i = 0}^k \f{1}{\sqrt{t_{i+1} - t_i}} \prod_{i = 0}^k \f{1}{\sqrt{\pi(t_{i+1} - t_i)}}e^{-\f{(y_{i+1} - y_i)^2}{t_{i+1} - t_i}},
     \]
     and recognize the second product as a product of transition probabilities for a diffusivity $\f{1}{\sqrt 2}$ Brownian motion. Hence, 
     \[
     \int_{\R^k} \int_{\R^k} \prod_{i = 0}^k \rho^2(t_{i+1} - t_i, y_{i+1}-y_i) \prod_{i = 1}^k dt_i dy_i  = \f{e^{-(y - x)^2/(t - s)}}{2^{k+1}\pi^{(k+1)/2}\sqrt{\pi(t-s)}}  \int_{\R^k}\prod_{i = 0}^k \f{\ind(t_{i+1} > t_i)}{\sqrt{t_{i+1} - t_i}} \prod_{i = 1}^k\,dt_i,  
     \]
     and one readily verifies that this agrees with \eqref{rhointcomp}.
     Next, reversing the transformation \eqref{WNtrans} just as in \eqref{pN2_integral} (and dividing by the  $N^{3k/2}$ Jacobian term), the integral over $A_k(N,s,t,y,x)$  of the term in \eqref{856} is equal to 
     \begin{align*}
         &\f{N^{-k/2 + 1}}{\pi^{(k+1)/2}} e^{-2[(t-s)N +(y-x)\sqrt N]} 2^{2(\lfloor tN \rfloor - \lfloor sN \rfloor)} \\
         &\qquad \qquad \times \sum_{\delta_k(\lfloor tN \rfloor \viiva \lfloor sN \rfloor)} \int_{\Delta_k(tN + y \sqrt N \viiva sN + x \sqrt N)} \prod_{i = 0}^k \f{(y_{i+1} - y_i)^{2(n_i - n_{i-1})}}{[2(n_{i+1} - n_i)]!\sqrt{(n_{i+1} - n_i)\vee 1}} \prod_{i = 1}^k dt_i \\
         &\overset{\eqref{eq:Dir_int}}{=}  \f{N^{-k/2}}{\pi^{(k+1)/2}} \f{2^{2(\lfloor tN \rfloor - \lfloor sN \rfloor)}[(\lfloor tN \rfloor - \lfloor sN \rfloor)!]^2 ((t-s)N + (y-x)\sqrt N)^k [p_N(t,y \viiva s,x)]^2}{[2(\lfloor tN \rfloor - \lfloor sN \rfloor) + k]!} \\
         &\qquad\qquad\qquad \times \sum_{\delta_k(\lfloor tN \rfloor \viiva \lfloor sN \rfloor)}\prod_{i = 0}^k \f{1}{\sqrt{(n_{i+1} - n_i) \vee 1}} \\
         \intertext{(next by Stirling's approximation and $p_N \to \rho$)}
         &\sim \f{N^{-(k-1)/2} \sqrt{t-s}}{2^{k} \pi^{k/2}}\rho^2(t-s, y-x) \sum_{\delta_k(\lfloor tN \rfloor \viiva \lfloor sN \rfloor)}\prod_{i = 0}^k \f{1}{\sqrt{(n_{i+1} - n_i) \vee 1}} 
         \\
         &= \f{N^{-(k-1)/2} \sqrt{t-s}}{2^k \pi^{k/2}}\rho^2(t-s, y-x) \sum_{\substack{m_i \ge 0 \\ \sum_{i = 1}^{k+1} m_i = \lfloor tN \rfloor - \lfloor sN \rfloor}}\prod_{i = 1}^{k+1} \f{1}{\sqrt{m_i \vee 1}} \\
         &=  \f{N^{-(k-1)/2} \sqrt{t-s}}{2^k \pi^{k/2}}\rho^2(t-s, y-x) \\
         &\qquad\qquad \times \int_{\R^k} \ind(t_i > 0, 1 \le i \le k, t_{k+1} = \lfloor tN \rfloor - \lfloor sN \rfloor - \sum_{i = 1}^k \lfloor t_i \rfloor \ge 0  )\prod_{i = 1}^{k+1} \f{1}{\sqrt{\lfloor t_i \rfloor \vee 1}} \prod_{i = 1}^k dt_i \\
         &= \f{\sqrt{t-s}}{2^k \pi^{k/2}} \rho^2(t-s,y-x) 
         \int\limits_{B_k(N)}\Biggl(\prod_{i = 1}^{k} dt_i  \sqrt{\f{N}{\lfloor t_i N \rfloor \vee 1}}\Biggr) \sqrt{\f{N}{(\lfloor tN \rfloor - \lfloor sN \rfloor - \sum_{i = 1}^k \lfloor t_iN \rfloor) \vee 1  }}, 
     \end{align*}
     where the last integration is over the set 
     \[
     B_k(N) = \Bigl\{t_i > 0,\; 1 \le i \le k,\; \sum_{i = 1}^k \lfloor t_i N \rfloor  \le \lfloor tN \rfloor - \lfloor sN \rfloor \Bigr\}.
     \]
      Comparing to \eqref{rhointcomp}, the proof is complete once we show that
     \begin{align}
     &\quad \;\lim_{N \to \infty} \int_{\R^k} \ind_{B_k(N)}\Biggl(\prod_{i = 1}^{k} dt_i  \sqrt{\f{N}{\lfloor t_i N \rfloor \vee 1}}\Biggr) \sqrt{\f{N}{(\lfloor tN \rfloor - \lfloor sN \rfloor - \sum_{i = 1}^k \lfloor t_iN \rfloor) \vee 1  }} \label{prelimitint}\\
     &= \int_{\R^k} \ind_{B_k} \Biggl(\prod_{i = 1}^{k}\,dt_i\f{1}{\sqrt t_i}\Biggr) \f{1}{\sqrt{t - s - \sum_{i = 1}^k t_i}}. \label{limit_int}
    \end{align}
    The proof is technical and lengthy, and is handled in Lemma \ref{lem:int_conv} at the end of the section. 
    \end{proof}

\begin{lemma} \label{lem:L1conv_bd}
Let $f:\Omega \times \R \to \R_{>0}$ be a jointly measurable 
function, independent of $\OCY_{\beta_N}$ and $Z_\beta$ such that, for some $\alpha > 0$, 
\be \label{fgrowth}
\Ee[f(x)] \le e^{\alpha|x|} \quad\forall x \in \R.
\ee
For a sequence $\beta_N$ with $N^{1/4}\beta_N \to \beta > 0$ and $Z_N(t,y\viiva s,x)$ defined as in \eqref{ZNshort}, under the coupling of Theorem \ref{thm:OCYtoSHE}, the following convergence holds 
for each choice of $y \in \R$ and $t > 0$:
\be\label{ZZf8}  
\lim_{N\to\infty} \Ee\biggl[\;\biggl\lvert 
\,\int_{-\infty}^{t\sqrt N + y} f(x) Z_N(t,y\viiva 0,x)\,dx \;-\;  \int_\R f(x) Z_\beta(t,y\viiva 0,x)\,dx \,
\biggr\rvert\;\biggr]  =0  . 
\ee
In particular, for a finite collection $\{f_i\}_{1 \le i \le k}$ of jointly measurable functions $f_i: \Omega \times \R \to \R_{>0}$ each satisfying \eqref{fgrowth} for some $\alpha > 0$,   the following weak convergence holds for finite-dimensional distributions of these  processes indexed by $y\in\R$: 
\[
\Biggl\{\int_{-\infty}^{t\sqrt N + y} f_i(x) Z_N(t,y \viiva 0,x)\,dx : y \in \R\Biggr\}_{1 \le i \le k} 
\overset{N \to \infty}{\Longrightarrow}
\Biggl\{\int_\R f_i(x) Z_\beta(t,y \viiva 0,x)\,dx: y \in \R\Biggr\}_{1 \le i \le k}.
\]
\end{lemma}
\begin{proof}
    We  can integrate over the whole space $\R$ in the left integral in \eqref{ZZf8}  because $Z_N(t,y \viiva 0,x)$ is defined to be zero for $x > t\sqrt N + y$. Then, by independence and the growth assumption on $f$, we have
    \begin{align} \label{Zbetal1bd}
    \Ee \Bigl[\int_\R f(x)|Z_N(t,y \viiva 0,x) - Z_\beta(t,y \viiva 0,x)|\,dx\Bigr]
    \le \int_\R  e^{\alpha|x|} \Ee[|Z_N(t,y \viiva 0,x) - Z_\beta(t,y \viiva 0,x)| ]\,dx.
    \end{align}
    By Theorem \ref{thm:OCYtoSHE}, we have, for each $t > 0$ and $x,y \in \R$,
    \[
    \Ee[|Z_N(t,y \viiva 0,x) - Z_\beta(t,y \viiva 0,x)| ] \le  \Ee[|Z_N(t,y \viiva 0,x) - Z_\beta(t,y \viiva 0,x)|^2 ]^{1/2} \overset{N \to \infty}{\longrightarrow} 0.
    \]
    By \eqref{SHEchaos},  \eqref{Ygamexp},  and the choice of scaling, $\Ee[Z_N(t,y\viiva 0,x)] = p_N(t,y \viiva 0,x)$ and $\Ee[Z_\beta(t,y \viiva 0,x)] = \rho(t,y-x)$. Thus, the integrand on the right-hand side of the integrand in \eqref{Zbetal1bd} is bounded above by 
    \[
      e^{\alpha|x|}(p_N(t,y-x) + \rho(t,y-x)).
    \]
    Lemma \ref{lem:moment_bd}\ref{itm:pNptwise}--\ref{itm:intconv} implies that $p_N(t,y-x) \to \rho(t,y-x)$ pointwise and that 
    \[
    \int_\R e^{\alpha|x|}(p_N(t,y-x) + \rho(t,y-x))\,dx \overset{N \to \infty}{\longrightarrow} \int_\R 2 e^{\alpha|x|} \rho(t,y-x)\,dx < \infty,
    \]
    so the generalized dominated convergence theorem completes the proof. The   finite-dimensional weak convergence holds because finite linear combinations also satisfy the limit in \eqref{ZZf8}, 
    so the Cram\'er-Wold theorem completes the proof. 
\end{proof}

We conclude this section by completing the unfinished business of the proof of Theorem \ref{thm:OCYtoSHE}.
 \begin{lemma} \label{lem:int_conv}
    The convergence of \eqref{prelimitint} to \eqref{limit_int} holds. 
    \end{lemma}
    \begin{proof}
    For this, we break the set $B_k(N)$ into two disjoint pieces, 
    \begin{align*}
    B_k^1(N) = \Bigl\{t_i > \f{k+1}{N}\;\; \forall\, 1 \le i \le k \;\text{and}\; t-s - \sum_{i = 1}^k t_i > \f{2}{N}  \Bigr\} \quad \text{and} \quad
    B_k^2(N) =  B_k(N) \setminus B_k^1(N).
    \end{align*}
    We use the dominated convergence theorem to show that the integral over $B_k^1(N)$ converges to the desired limit, and we argue separately that the integral over $B_k^2(N)$ goes to $0$. First,  observe that, Lebesgue a.e., 
    \be \label{892}
    \ind_{B_k^1(N)} \Biggl(\prod_{i = 1}^{k}   \sqrt{\f{N}{\lfloor t_i N \rfloor \vee 1}}\Biggr) \sqrt{\f{N}{(\lfloor tN \rfloor - \lfloor sN \rfloor - \sum_{i = 1}^k \lfloor t_i N \rfloor) \vee 1  }} \rightarrow \ind_{B_k} \Biggl(\prod_{i = 1}^{k}\f{1}{\sqrt t_i}\Biggr) \f{1}{\sqrt{t - s - \sum_{i = 1}^k t_i}}.
    \ee
    Observe that, since $x - 1 \le \lfloor x \rfloor \le x$,
    \[
    \f{\lfloor t_i N \rfloor \vee 1}{N} \ge t_i - \f{1}{N} \quad \text{and}\quad \f{\lfloor tN \rfloor - \lfloor sN \rfloor - \sum_{i = 1}^k \lfloor t_i N \rfloor}{N} \ge t - s - \sum_{i = 1}^k t_i - \f{1}{N},
    \]
    and thus, for $N$ large, 
    \be \label{tibd}
    t_i > \f{k+1}{N} \Longrightarrow t_i > \f{2}{N} \Longrightarrow  t_i - \f{1}{N}  > \f{t_i}{2} \Longrightarrow \sqrt{\f{\lfloor t_i N \rfloor \vee 1}{N}} \le \sqrt{\f{2}{t_i}}
    \ee
    and 
    \be \label{cbd}
    t - s - \sum_{i = 1}^k t_i > \f{2}{N} \Longrightarrow \sqrt{\f{N}{(\lfloor tN \rfloor - \lfloor sN \rfloor - \sum_{i = 1}^k \lfloor t_i N\rfloor) \vee 1  }} \le \sqrt{\f{2}{t - s - \sum_{i = 1}^k t_i}}.
    \ee
    Therefore,
    \begin{align*}
    \ind_{B_k^1(N)} \Biggl(\prod_{i = 1}^{k}   \sqrt{\f{N}{\lfloor t_i N \rfloor \vee 1}} \Biggl) \sqrt{\f{N}{(\lfloor tN \rfloor - \lfloor sN \rfloor - \sum_{i = 1}^k \lfloor t_i N\rfloor) \vee 1  }} \le \ind_{B_k} \Biggl(\prod_{i = 1}^k \sqrt{\f{2}{t_i}} \Biggr)\sqrt{\f{2}{t - s - \sum_{i = 1}^k t_i}} ,
    \end{align*}
    and the right-hand side is integrable over $\R^k$ (it is a constant multiple of the Dirichlet density). The dominated convergence theorem now implies the convergence of integrals of the functions in \eqref{892}. 
    
    We turn to showing the integral over $B_k^2(N)$ converges to $0$. Observe first that on the set $B_k(N)$, 
    \be \label{k+1bd}
        t - s - \sum_{i = 1}^{k} t_i \ge \f{\lfloor tN \rfloor - \lfloor sN \rfloor - \sum_{i = 1}^k \lfloor t_i N \rfloor}{N} - \f{k+1}{N} \ge -\f{k+1}{N}.
    \ee
    From the first inequality of \eqref{k+1bd}, we observe that, for all $N$, sufficiently large (depending on $t,s$),
    \be \label{impst}
    t_i \le \f{k+1}{N}\;\; \forall\;\; 1 \le i \le k \Longrightarrow t - s - \sum_{i = 1}^{k} t_i \ge \f{\lfloor tN \rfloor - \lfloor sN \rfloor}{N} - \f{k(k+1)}{N} - \f{k+1}{N} > \f{2}{N}.
    \ee
    Next, we break up the set $B_k^2(N)$  into $2^{k+1} - 2$ disjoint sets determined by whether $t_i \le \f{k+1}{N}$ for $1 \le i \le k$ and by whether $t - s - \sum_{i = 1}^k t_i \le \f{2}{N}$. The minus $2$ comes because $B_k^1(N)$ is one of these possible sets, and \eqref{impst} eliminates another possibility. Enumerate these sets as $\{B_k^{2,j}\}_{1 \le j \le 2^{k+1} - 2}$. We show that the integral over each $B_{k}^{2,j}$ converges to $0$. We do this by considering four separate cases for $B_{k}^{2,j}$. To avoid messy calculations, we use the shorthand notation
    \[
    I_k^{j}(N) := \int_{\R^k} \ind_{B_k^{2,j}(N)}\Biggl(\prod_{i = 1}^{k} dt_i  \sqrt{\f{N}{\lfloor t_i N \rfloor \vee 1}}\Biggr) \sqrt{\f{N}{(\lfloor tN \rfloor - \lfloor sN \rfloor - \sum_{i = 1}^k \lfloor t_i \rfloor) \vee 1  }}. 
    \]

    \smallskip \noindent \textbf{Case 1:  $2$ or more of the $t_i$ for $1 \le i \le k$ satisfy $t_i \le \f{k+1}{N}$:} Without loss of generality, we say that, for some $\ell \ge 2$, $t_i \le \f{k+1}{N}$  for $1 \le i \le \ell $, and $t_i > \f{k+1}{N}$  for $\ell + 1 \le i \le k$. For $\ell + 1 \le i \le k$, we use the bound in \eqref{tibd}. We also make use of the following bounds which hold in general:
    \be \label{trivest}
    \sqrt{\f{N}{\lfloor t_i N \rfloor \vee 1}} \le \sqrt N\quad \text{and}\quad \sqrt{\f{N}{(\lfloor tN \rfloor - \lfloor sN \rfloor - \sum_{i = 1}^k \lfloor t_i \rfloor) \vee 1 }} \le \sqrt N.
    \ee
    Observe also that on $B_k(N)$, for $1 \le i \le k$, and all $N$ sufficiently large,
    \[
    0 < t_i \le \f{\lfloor tN \rfloor - \lfloor sN \rfloor +1}{N} \le t-s + 1.
    \]
    Then, 
    \begin{align*}
        I_k^{j}(N) \le  \sqrt N \Biggl(\int_{0}^{(k+1)/N} \sqrt N \,dt\Biggr)^\ell \Biggl(\int_{0}^{t-s+1} \f{2}{\sqrt u}\,du \Biggr)^{k-\ell}  \le C(k,\ell) N^{-(\ell-1)/2} \to 0.
    \end{align*}

    \smallskip \noindent \textbf{Case 2: Exactly one of the $t_i$ for $1 \le i \le k$ satisfies $t_i \le \f{k+1}{N}$ and $t - s - \sum_{i = 1}^k t_i > \f{2}{N}$:}
    Without loss of generality, we will say $t_1 \le \f{k+1}{N}$. We start similarly to the last case, but instead use the bound \eqref{cbd} for the last term. In the following, the constant $C > 0$ depends on $t -s$ and $k$ and may change from line to line.
    \begin{align*}
    I_k^j(N) &\le \sqrt N \int_{\R^k}dt_1 \Biggl(\prod_{i = 2}^k\,dt_i \sqrt{\f{2}{t_i}}\Biggr)\; \sqrt{\f{2}{t - s - \sum_{i = 1}^k t_i}} \ind\Bigl(0 < t_1 < \f{k+1}{N},\; t_i > 0, 2 \le i \le k,\; \sum_{i = 1}^k t_i < t-s \Bigr)  \\
    &\le C \sqrt N \int_{\R^k}dt_1\Biggl(\prod_{i = 2}^k\,dt_i \sqrt{\f{1}{t_i}}\Biggr) \; \sqrt{\f{1}{1 - \sum_{i = 1}^k t_i}}\ind\Bigl(0 < t_1 < \f{k+1}{(t-s)N},\; t_i > 0, 2 \le i \le k, \;\sum_{i = 1}^k t_i < 1 \Bigr)  \\
    &\le C\sqrt N P\Bigl(0 < X_1 < \f{k+1}{(t-s)N}\Bigr) .
    \end{align*}
    where $P$ is the distribution of a random vector $(X_1,\ldots,X_{k+1})$ that is distributed according to the Dirichlet distribution with parameter vector $(1,\f{1}{2},\ldots,\f{1}{2})$. The next step follows from $X_1$ having a Beta distribution with parameters $(1,\f{k}{2})$.
    Thus,  for constants $C_1,C_2 > 0$ changing from term to term,
    \[
    \sqrt N P\Bigl(0 < X_1 < \f{k+1}{(t-s)N}\Bigr) = C_1\sqrt N \int_0^{C_2/N} (1-t)^{k/2 - 1}\,dt = C_1\sqrt N (1 - (1-C_2/N)^{k/2}) \le C N^{-1/2}.
    \]

    \smallskip \noindent \textbf{Case 3: $t_i > \f{k + 1}{N}$ for $1 \le i \le k$ and $t-s - \sum_{i = 1}^k t_i \le \f{2}{N}$:}
    We use \eqref{tibd} and \eqref{trivest} to get the estimate
    \[
    I_k^j(N) \le C \sqrt N \int_{\R^k} \ind_{B_k^{2,j}(N)} \prod_{i = 1}^k dt_i \sqrt{\f{1}{t_i}}
    \]
    for a constant $C$ depending on $k$.
    Next, consider the following change of variable:
    \be \label{chng_var}
    \wt t_1 = \f{k+1}{N} + t -s - \sum_{i = 1}^k t_i,\quad \wt t_i = t_i, \;\;2 \le i \le k.
    \ee
    On the set $B_{k}^{2,j}$, the assumption $t-s - \sum_{i = 1}^k t_i \le \f{2}{N}$ and \eqref{k+1bd} imply that $0 \le \wt t_1 \le \f{k+3}{N}$. Furthermore, 
    \[
        t - s - \sum_{i = 1}^k \wt t_i = t_1- \f{k+1}{N} > 0.
    \]
    In summary, the transformed vector lies in the set 
    \[
    \wt B_k^{2,j}(N) := \Bigl\{ 0 < \wt t_1 \le \f{k+3}{N},\quad \wt t_i > 0, 2 \le i \le k,\quad \sum_{i = 1}^k \wt t_i < t - s\Bigr\}
    \]
    Putting this all together, 
    \begin{align*}
    I_k^j(N) &\le C \sqrt N \int_{\R^k} \ind_{\wt B_k^{2,j}} \,d\wt t_1 \Biggl(\prod_{i = 2}^{k}\,d\wt t_i \f{1}{\sqrt{\wt t_i}} \Biggr) \sqrt{\f{1}{t - s - \sum_{i = 1}^k \wt t_i + \f{k + 1}{N}}} \\
    &\le  C \sqrt N \int_{\R^k} \ind_{\wt B_k^{2,j}} \,d\wt t_1 \Biggl(\prod_{i = 2}^{k}\,d\wt t_i \f{1}{\sqrt{\wt t_i}} \Biggr) \sqrt{\f{1}{t - s - \sum_{i = 1}^k \wt t_i}}. 
    \end{align*}
    The asymptotics of the integral can now be reduced to the computation of a beta probability, just as in the previous case. 
    
    \smallskip \noindent \textbf{Case 4: Exactly one of the $t_i$ for $1 \le i \le k$ satisfies $t_i \le \f{k+1}{N}$ and $t - s - \sum_{i = 1}^k t_i \le \f{2}{N}$.}
    Without loss of generality, we will say that $t_2 \le \f{k+1}{N}$. Then, using the bounds \eqref{trivest} for $i =1,2$ and the last factor, then \eqref{tibd} for $3 \le i \le k$, 
    \[
    I_k^j(N) \le C N^{3/2} \int_{\R^k} \ind_{B_k^{j,2}(N)} dt_1 dt_2 \prod_{i = 3}^k \f{dt_i}{\sqrt t_i}.
    \]
    Making the same change of variables \eqref{chng_var} as in the previous case,
    \[
    I_k^j(N) \le C N^{3/2} \int_{0}^{(k+3)/N}\,dt_1 \int_0^{(k+1)/N}\,dt_2 \Biggl(\int_0^{t-s + 1} \f{1}{\sqrt u}\,du 
 \Biggr)^{k-2} \le CN^{-1/2} \to 0,
    \]
    where the last $k-2$ integrals may be taken from $0$ to $t - s + 1$ for sufficiently large $N$ by the same reasoning as in Case 1. This concludes all cases.
 \end{proof}

\section{Proofs of the main theorems} \label{sec:proofs}

\subsection{Characterization and regularity of the Busemann process}
 We now turn to proving Theorem \ref{thm:SHEjumps}. To do this, we prove invariance of the $\KPZH_\beta$ for the SHE. Then, we use the uniqueness result from \cite{Janj-Rass-Sepp-22} (recorded as Theorem \ref{thm:SHE_Buse_unique}) to show that $\KPZH_\beta$ describes the Busemann process. Corollary \ref{cor:SHEjumps} gives the existence of discontinuities. We first prove an intermediate invariance result for the O'Connell-Yor polymer.

 \begin{proposition} \label{prop:OCYinvar}
Let $\beta > 0$, and let $\KH_\beta = \{\KH_\beta^\lambda\}_{\lambda \in \R}$ be the $\KPZH_\beta$.  Let $B_0,B_1,\ldots$ be a sequence of i.i.d.\ Brownian motions, independent of $\KH_\beta$ and defining the partition function \eqref{OCYpart}. Let $\OCY_\beta(n,y \viiva f)$ be as defined in \eqref{Zf}. Then, for $n \ge 0$ and $0 < \lambda_1 < \cdots < \lambda_k$,
we have this distributional equality on $C(\R)^k$: 
\be \label{OCYZinvar}
\Biggl\{\f{\OCY_\beta(n,\aabullet\viiva e^{\beta \KH_\beta^{\lambda_i}})}{\OCY_\beta(n,0\viiva e^{\beta \KH_\beta^{\lambda_i}})}\Biggr\}_{1 \le i \le k} \deq \{\exp\bigl(\beta \KH_\beta^{\lambda_i}(\aabullet)\bigr)\}_{1 \le i \le k}.
\ee
\end{proposition}
\begin{proof}
We prove this by induction. For $n = 0$, $\lambda>0$, and $y \in \R$,
\[
\OCY_\beta(0,y\viiva e^{\beta \KH_\beta^\lambda}) = \int_{-\infty}^y e^{\beta \KH_\beta^\lambda(x)}Z_\beta(0,y\viiva 0,x)\,dx = \int_{-\infty}^y \exp\bigl(\beta (\KH_\beta^\lambda(x) + B_0(x,y))\bigr)\,dx,
\]
and therefore, for $0 < \lambda_1 < \cdots < \lambda_k$,
\begin{align*}
     \Biggl\{\f{\OCY_\beta(0,y\viiva e^{\beta \KH_\beta^{\lambda_i}})}{\OCY_\beta(0,0\viiva e^{\beta \KH_\beta^{\lambda_i}})} : y\in\R \Biggr\}_{1 \le i \le k} 
     &= \Biggl\{\f{e^{\beta B_0(y)} \int_{-\infty}^y \exp\bigl(\beta(\KH_\beta^{\lambda_i}(x) - B_0(x))\bigr)\,dx}{\int_{-\infty}^0 \exp\bigl(\beta(\KH_\beta^{\lambda_i}(x) - B_0(x))\bigr)\,dx}: y\in\R\Biggr\}_{1 \le i \le k} \\
     &\!\!\!\overset{\eqref{Did}}{=} \bigl\{\exp\bigl(\beta D_\beta(B_0,\KH_\beta^{\lambda_i})(y)\bigr)\bigr\}_{1 \le i \le k}. 
\end{align*}
By Theorem \ref{thm:OCYinvar}, this has the same distribution as $\{\exp(\beta \KH_\beta^{\lambda_i})\}_{1 \le i \le k}$. 

Now, assume the invariance \eqref{OCYZinvar} holds for some $n \ge 0$. 
Then, 
\begin{align*}
\OCY_\beta(n + 1,y\viiva \exp(\beta\KH_\beta^\lambda)) &= \int_{-\infty}^y e^{\beta \KH_\beta^\lambda(x)}\OCY_\beta(n + 1,y\viiva 0,x)\,dx \\
&\overset{\eqref{Zcomp}}{=} \int_{-\infty}^y  \int_{x}^y \exp\bigl(\beta(\KH_\beta^\lambda(x) + B_{n + 1}(w,y))\bigr) \OCY_\beta(n,w|0,x)\,dw\,dx \\
&= \int_{-\infty}^y  \int_{-\infty}^w \exp\bigl(\beta(\KH_\beta^\lambda(x) + B_{n + 1}(w,y))\bigr) \OCY_\beta(n,w|0,x)\,dx\,dw \\
&\overset{\eqref{Zf}}{=} \int_{-\infty}^y \exp\bigl(\beta B_{n +1}(w,y)\bigr)\OCY_\beta(n,w\viiva \exp(\beta\KH_\beta^\lambda))\,dw.
\end{align*}
Then,
\begin{align*}
& \Biggl\{\f{\OCY_\beta(n + 1,y\viiva e^{\beta \KH_\beta^{\lambda_i}})}{\OCY_\beta(n + 1,0\viiva e^{\beta \KH_\beta^{\lambda_i}})}: y\in\R\Biggr\}_{1 \le i \le k} 
= \left\{\f{e^{\beta B_{n + 1}(y)}\ddd\int_{-\infty}^y \tfrac{\OCY_\beta(n,w\viiva e^{\beta \KH_\beta^{\lambda_i}})}{\OCY_\beta(n,0\viiva e^{\beta \KH_\beta^{\lambda_i}})} e^{- \beta B_{n +1}(w)}\,dw }
{\ddd\int_{-\infty}^0 \tfrac{\OCY_\beta(n,w\viiva e^{\beta \KH_\beta^{\lambda_i}})}{\OCY_\beta(n,0\viiva e^{\beta \KH_\beta^{\lambda_i}})} e^{- \beta B_{n +1}(w)}\,dw}: y\in\R\right\}_{1 \le i \le k} \\[0.5em]
&\deq \Biggl\{\f{e^{\beta B_{n + 1}(y))}\int_{-\infty}^y e^{\beta \KH_\beta^{\lambda_i}(w)} e^{- \beta B_{n +1}(w)}\,dw }
{\int_{-\infty}^0 e^{\beta \KH_\beta^{\lambda_i}(w)} e^{- \beta B_{n +1}(w)}\,dw}: y\in\R\Biggr\}_{1 \le i \le k}
\overset{\eqref{Did2}}{=} \bigl\{e^{\beta D(B_{n +1},\KH_\beta^{\lambda_i})} \bigr\}_{1 \le i \le k} \deq \{e^{\beta \KH_\beta^{\lambda_i}}\}_{1 \le i \le k}.
\end{align*}
The first distributional equality is the induction assumption and the second one Theorem \ref{thm:OCYinvar}.
\end{proof}
 
 Let the fundamental solution $Z_\beta$ of SHE  be defined as in \eqref{SHEchaos}, and recall the definition with initial data \eqref{ZZf}.

\begin{theorem} \label{thm:SHEinvar}
Let $\beta > 0$, and let $\KH_\beta$ be the $\KPZH_\beta$, defined on the probability space $(\Omega,\F,\Pp)$ and independent of the SHE Green's function $Z_\beta$. Let $t > 0$ and real $\lambda_1 < \cdots < \lambda_k$. Then,
\be \label{eqn:SHEinvar}
\Biggl\{\f{Z_\beta(t,\aabullet\viiva 0, e^{\beta \KH_\beta^{\lambda_i}})}{Z_\beta(t,0\viiva 0, e^{\beta \KH_\beta^{\lambda_i}})}\Biggr\}_{1 \le i \le k} \deq \{\exp(\beta \KH_\beta^{\lambda_i}(\aabullet))\}_{1 \le i \le k}.
\ee
\end{theorem}
 
\begin{proof}
For $N \in \N$ and $1 \le i \le k$, set $\mu_i^N = (\lambda_i + \f{\beta}{2}) N^{-1/4} + \beta^{-1} N^{1/4}$ and $\beta_N = N^{-1/4}\beta$. 
Let $N$ be large enough so that $\mu_i^N > 0$ for $i \in \{1,\ldots,k\}$. By \eqref{OCYZinvar}, for every $n \ge 1$,
\[
\Biggl\{\f{\OCY_{\beta_N}(n,\aabullet\viiva \exp(\beta_N \KH_{\beta_N}^{\mu_i^N}))}{\OCY_{\beta_N}(n,0\viiva \exp(\beta_N \KH_{\beta_N}^{\mu_i^N}))}\Biggr\}_{1 \le i \le k} \deq \{\exp(\beta_N \KH_{\beta_N}^{\mu_i^N}(\aabullet))\}_{1 \le i \le k}.
\]
Then, we have that 
\be \label{872}
\begin{aligned}
    &\quad \Biggl\{ \f{\exp(-(\sqrt N + \beta^2/2)y)\int_{-\infty}^{tN + y\sqrt N}\exp(\beta_N \KH_{\beta_N}^{\mu_i^N}(x))\OCY_{\beta_N}(tN, tN +y\sqrt N\viiva 0,x)\,dx}{\int_{-\infty}^{tN} \exp(\beta_N \KH_{\beta_N}^{\mu_i^N}(x))\OCY_{\beta_N}(tN,tN\viiva 0,x)\,dx}: y \in \R \Biggr\}_{1 \le i \le k} \\[0.5em]
    &\deq \Bigl\{\exp\bigl(\beta_N \KH_{\beta_N}^{\mu_i^N}(tN,tN + y\sqrt N)- (\sqrt N + \beta^2/2)y\bigr) : y \in \R \Bigr\}_{1 \le i \le k} \\[0.5em]
    &\deq \bigl\{\exp\bigl(\beta_N \KH_{\beta_N}^{\mu_i^N}(y \sqrt N)- (\sqrt N + \beta^2/2)y\bigr):y \in \R\bigr\}_{1 \le i \le k}\\
    &\deq \bigl\{\exp\bigl(\beta \KH_\beta^{\lambda_i}(y)\bigr):y \in \R \bigr\}_{1 \le i \le k}.
    \end{aligned}
\ee
where the second equality follows from shift invariance (Theorem \ref{thm:dist_invar}\ref{itm:shinv}), and the third equality follows from the scaling relations of Theorem \ref{thm:dist_invar}\ref{itm:KHscale}. For $t > 0$ and $x,y \in \R$, set
\[
\psi_N(t,x,y) = \sqrt N \exp\Bigl(-\bigl(N +\f{\beta^2  \sqrt N}{2}\bigr)t - \bigl(\sqrt N + \f{\beta^2}{2}\bigr)(y-x)\Bigr)
\]
so that, for our choice of $\beta_N$, $\psi_N(t,x,y) = \psi_N(0,t,x,y;\beta_N)$,
where the latter was defined in \eqref{psi}. 
By a change of variable from $x$ to $\sqrt Nx$, the first line of \eqref{872} is equal to
\[
\begin{aligned}
&\quad \Biggl\{ \f{\sqrt N \exp(-(\sqrt N + \beta^2/2)y)\int_{-\infty}^{t\sqrt N + y}\exp(\beta_N \KH_{\beta_N}^{\mu_i^N}(x\sqrt N ))\OCY_{\beta_N}(tN,tN + y\sqrt N|0,x\sqrt N)\,dx}{\sqrt N \int_{-\infty}^{t\sqrt N}  \exp(\beta_N \KH_{\beta_N}^{\mu_i^N}(x\sqrt N ))\OCY_{\beta_N}(tN,tN|0,x \sqrt N )\,dx}: y \in \R \Biggr\}_{1 \le i \le k} \\[0.5em]
&= \Biggl\{ \f{\int_{-\infty}^{t\sqrt N + y}e^{\beta_N \KH_{\beta_N}^{\mu_i^N}(x\sqrt N) - 
(\sqrt N + \beta^2/2) x}\psi_N(t,x,y) \OCY_{\beta_N}(tN,tN + y\sqrt N \viiva 0,x\sqrt N )\,dx}{\int_{-\infty}^{t\sqrt N}  e^{\beta_N \KH_{\beta_N}^{\mu_i^N}(x\sqrt N )- 
(\sqrt N + \beta^2/2)x}\psi_N(t,x,0)\OCY_{\beta_N}(tN,tN\viiva 0,x\sqrt N  )\,dx}: y \in \R \Biggr\}_{1 \le i \le k}\\[0.5em]
&\deq \Biggl\{ \f{\int_{-\infty}^{t\sqrt N + y}\exp(\beta \KH_\beta^{\lambda_i}(x)) \psi_N(t,x,y) \OCY_{\beta_N}(tN,tN + y\sqrt N |0,x\sqrt N )\,dx}{\int_{-\infty}^{t\sqrt N}  \exp(\beta \KH_\beta^{\lambda_i}(x))\psi_N(t,x,0)\OCY_{\beta_N}(tN,tN|0,x\sqrt N  )\,dx}: y \in \R \Biggr\}_{1 \le i \le k}.
\end{aligned}
\]
where the distributional equality follows from the scaling of Theorem \ref{thm:dist_invar}\ref{itm:KHscale}, similarly as is done in \eqref{872}. Lemma \ref{lem:L1conv_bd} implies that the above converges to, in the sense of finite dimensional distributions on
$C(\R,\R^k)$,
\be \label{873}
\Biggl\{\f{\int_\R \exp(\beta \KH_\beta^{\lambda_i}(x)) Z_\beta(t,y\viiva 0,x)\,dx}{\int_\R \exp(\beta \KH_\beta^{\lambda_i}(x)) Z_\beta(t,0\viiva 0,x)\,dx}:y \in \R\Biggr\}_{1 \le i \le k} 
= \Biggl\{\f{Z_\beta(t,y\viiva 0, e^{\beta \KH_\beta^{\lambda_i}})}{Z_\beta(t,0\viiva 0, e^{\beta \KH_\beta^{\lambda_i}})}:y \in \R\Biggr\}_{1 \le i \le k}.
\ee
In the application of the Lemma, $f_i(x) = e^{\beta F_\beta^{\lambda_i}(x)}$ for which the condition $\Ee[f(x)] \le e^{\alpha |x|}$ follows immediately.
Tightness follows because  the distribution of the process does not depend on $N$ by \eqref{872}.
 Then, by comparing \eqref{872} and \eqref{873}, for each $t \ge 0$,
\[
\Biggl\{\f{Z_\beta(t,\aabullet\viiva 0, e^{\beta \KH_\beta^{\lambda_i}})}{Z_\beta(t,0\viiva 0, e^{\beta \KH_\beta^{\lambda_i}})}\Biggr\}_{1 \le i \le k} \deq \{\exp(\beta \KH_\beta^{\lambda_i})\}_{1 \le i \le k}. \qedhere
\]
\end{proof}

\begin{corollary} \label{cor:SHEBuse=KPZH}
Let $\beta > 0$. Then, the following distributional equality holds as processes in $D(\R,C(\R))$:
\[
\{\bb^{(\beta \lambda) +}(0,0,0,\aabullet)\}_{\lambda \in \R} \deq \{\beta \KH_\beta^\lambda\}_{\lambda \in \R}.
\]
\end{corollary}
\begin{proof}
    The invariance of Theorem \ref{thm:SHEinvar} and the uniqueness of Theorem \ref{thm:SHE_Buse_unique}
    establish that for $\lambda_1 < \cdots < \lambda_k$, 
    \[
    \{\bb^{(\beta \lambda_i)+}(0,0,0,\aabullet)\}_{1 \le i \le k} \deq \{\beta \KH_\beta^{\lambda_i}\}_{1 \le i \le k}.
    \]
    The choice of factors of $\beta$ comes by comparing drifts, using Proposition \ref{prop:KPZH_cons}\ref{itm:KPZHBM} and Theorem \ref{thm:SHEbuse}\ref{itm:BuseBM}.  
    The equality of processes in the path space $D(\R,C(\R))$ follows by the uniqueness of Proposition \ref{prop:KPZH_cons}\ref{itm:KPZH_dist}. 
\end{proof}

\begin{proof}[Proof of Theorem \ref{thm:Busedrift}]
The description of the measures used in the theorem comes from Lemma \ref{lem:pQalt} and the definition of the finite-dimensional marginals of the $\KPZH_\beta$ in Proposition \ref{prop:KPZH_cons}\ref{itm:KPZH_dist}.  Uniqueness follows directly from Corollary \ref{cor:SHEBuse=KPZH} and the uniqueness in Theorem \ref{thm:SHE_Buse_unique}. In handling the factor of $\beta$, we recall that we define solutions to the KPZ equation as in \eqref{eqn:KPZeqn} as
\[
\kpzb(t,y\viiva s,f) = \f{1}{\beta}\log \int_\R e^{\beta f(x)} Z_\beta(t,y|s,x)\,dx. \qedhere
\]
\end{proof}

\begin{proof}[Proof of Theorem \ref{thm:jumps_every_edge}]
By Corollary \ref{cor:SHEjumps}, the $\KPZH_\beta$ is not almost surely continuous. Corollary \ref{cor:SHEBuse=KPZH} gives the equality of the Busemann process and the $\KPZH_\beta$. By Theorem \ref{thm:SHEbuse}\ref{itm:SHEBusedichot}, the set of discontinuities $\Lambda_\beta$ is countable and dense in $\R$ with probability $1$. The presence of the discontinuities for the process $\lambda \mapsto \KH_\beta^\lambda(x,y)$ is  Theorem \ref{thm:SHEbuse}\ref{itm:SHEBusemont} (originally proved in \cite{Janj-Rass-Sepp-22}). This completes the proof. 
\end{proof}

\begin{proof}[Proof of Theorem \ref{thm:SHEjumps}]
This is a direct consequence of Corollary \ref{cor:SHEBuse=KPZH} and Theorem \ref{thm:jumps_every_edge}. 
\end{proof}

\subsection{Limits as $\beta \nearrow \infty$ and $\beta \searrow 0$}
\label{sec:pf_KPZH_SH}

\begin{proof}[Proof of Theorem \ref{thm:SHlimit}]
We first prove the limit as $\beta \nearrow \infty$: Proposition \ref{prop:KPZH_cons}\ref{itm:KPZH_dist} implies that, for $\lambda_1 < \cdots<\lambda_k$, $(\KH_\beta^{\lambda_1},\ldots,\KH_\beta^{\lambda_k}) \sim \mu_{\beta}^{\lambda_1,\ldots,\lambda_k}$. By Lemma \ref{lem:pQalt}, we can describe this distribution as 
$
(\KH_\beta^{\lambda_1},\ldots,\KH_\beta^{\lambda_k}) \deq (\eta_\beta^1,\ldots,\eta_\beta^k),
$
where for independent Brownian motions $Y^1,\ldots,Y^k$ with drifts $\lambda_1,\ldots,\lambda_k$, $\eta_\beta^1=Y^1$, and for $2\le n \le k$, 
\begin{align*}
\eta_\beta^n(y)  &= Y^1(y) + \beta^{-1} \log \!\!\!\!\!
\int\limits_{-\infty < x_{n - 1} < \cdots <x_{1} < y}  \!\!\!\!\!\!\!\!\!\!\qquad  \prod_{i = 1}^{n-1}\exp\bigl[\beta (Y^{i + 1}(x_i) - Y^i(x_i)) \bigr] d x_{i} \\
&\qquad\qquad\qquad\qquad\qquad - \beta^{-1}\log  \!\!\!\!\! \int\limits_{-\infty < x_{n - 1} < \cdots <x_{1} < 0}   \!\!\!\!\!\!\!\!\!\!\qquad  \prod_{i = 1}^{n-1}\exp\bigl[\beta (Y^{i + 1}(x_i) - Y^i(x_i)) \bigr] d x_{i}.
\end{align*}
By the convergence of the $L^\beta$ norm as $\beta \nearrow \infty$, the zero-temperature limit $\beta\to\infty$ converts the polymer free energy into last-passage percolation. Therefore, on a single event of full probability, simultaneously for each $y \in \R$ and $n \in \{2,\ldots,k\}$,
\begin{align*}
\lim_{\beta \nearrow \infty} \eta_\beta^n(y) 
&= Y^1(y) + \sup_{-\infty < x_{n - 1} \le \cdots \le x_{1} \le y} \biggl\{\;\sum_{i = 1}^{n - 1} (Y^{i + 1}(x_i) - Y^i(x_i)) \biggr\}  \\
&\qquad\qquad\qquad\qquad\qquad-   \sup_{-\infty < x_{n - 1} \le \cdots \le x_{1} \le 0} \biggl\{\;\sum_{i = 1}^{n - 1} (Y^{i + 1}(x_i) - Y^i(x_i)) \biggr\}.
\end{align*}
Lemma \ref{lem:SHqueuerep} and Definition \ref{def:SH} imply that, in the sense of finite-dimensional distributions on $C(\R^k,\R)$, 
\[
(\eta_\beta^1(2\,\aabullet),\ldots,\eta_\beta^k(2\,\aabullet)) \overset{\beta \nearrow \infty}{\Longrightarrow} (G^{\lambda_1},\ldots,G^{\lambda_k}).
\]
Tightness holds because each component in the prelimit is a Brownian motion with a fixed drift. 

For each $\beta$, the process in Item \ref{KPZscal} has the same distribution as the process in Item \ref{tempscal} by the scaling relations of Theorem \ref{thm:dist_invar}\ref{itm:KHscale}.

Now, we prove the convergence as $\beta \searrow 0$. By Theorem~\ref{thm:lambda-F}, as processes in $y > 0$,
\be \label{FYXrep_expand}
\KH_\beta^{\lambda_2}(y) - \KH_\beta^{\lambda_1}(y) \deq \beta^{-1} \log(1 + X_{\lambda,\beta} Y_{\lambda,\beta}(y)) = \log \Biggl(1 + \f{\beta^{-1} X_{\lambda,\beta}Y_{\lambda,\beta}}{\beta^{-1}}\Biggr)^{\beta^{-1}},
\ee
where $\lambda := \lambda_2 - \lambda_1$, $X_{\lambda,\beta}$ has the Gamma distribution with shape $\lambda \beta^{-1}$ and rate $\beta^{-2}$, and 
\be \label{Ylbrep}
Y_{\lambda,\beta}(y) = \int_0^y \exp(\sqrt 2 \beta B(x) + \lambda \beta x)\,dx
\ee
where $B$ is a standard Brownian motion. For fixed $\lambda,y > 0$, couple the $Y_{\lambda,\beta}(y)$ together with a single Brownian motion $B$ using \eqref{Ylbrep}. Note that for $\beta < 1$, 
\[
\int_0^y \exp(\sqrt 2 \beta B(x) + \lambda \beta x)\,dx \le \int_0^y \exp(\sqrt 2 |B(x)| + \lambda x)\,dx, 
\]
and the right-hand side is finite almost surely. By dominated convergence, $Y_{\lambda,\beta}(y)$ converges almost surely to $y$ as $\beta \searrow 0$.
Next, the random variable $X_{\lambda,\beta}/\beta$ has mean $\lambda$ and variance $\lambda \beta$, so for any $\ve > 0$, by Chebyshev's inequality, 
\be \label{XYconv}
\lim_{\beta \searrow 0} \Pp\Bigl(|\beta^{-1}X_{\lambda,\beta} - \lambda| > \ve \Bigr) = 0.
\ee
Hence, there exists a coupling of copies of $X_{\lambda,\beta}$ (which we may keep independent of $Y_{\lambda,\beta}(y)$) so that $X_{\lambda,\beta} \to \lambda$ almost surely as $\beta \searrow 0$. In the product space, using \eqref{FYXrep_expand}, $\beta^{-1} \log(1 + X_{\lambda,\beta} Y_{\lambda,\beta}(y))$ converges almost surely to $\lambda y$. Therefore, for each $\ve > 0$ and $y > 0$,
\[
\limsup_{\beta \searrow 0} \Pp(|\KH_\beta^{\lambda_2}(y) - \KH_\beta^{\lambda_1}(y) - (\lambda_2 - \lambda_1)y| > \ve) = 0.
\]
The result also holds for $y < 0$  because $\{\KH_\beta^{\lambda}(y)\}_{\lambda \in \R} = \{-\KH_\beta^\lambda(y,0)\}_{\lambda \in \R} \deq \{-\KH_\beta^\lambda(0,-y)\}_{\lambda \in \R}$ (Theorem \ref{thm:dist_invar}\ref{itm:KHscale}). Now, let $\lambda_1 < \ldots < \lambda_k$ and $\{y_{i,j}: 2 \le i \le k, 1 \le j \le J_i\}$ be a finite collection of points in $\R$.
By a simple union bound, for each $\ve > 0$, we have 
\[
\limsup_{\beta \searrow 0} \Pp\Bigl(\sup_{2 \le i \le k, 1 \le j \le J_i}|\KH_\beta^{\lambda_i}(y_{i,j}) - \KH_\beta^{\lambda_1}(y_{i,j}) - (\lambda_i - \lambda_1)y_{i,j}| > \ve \Bigr) = 0.
\]
Since the marginal distribution of $\KH_\beta^{\lambda_1}$ does not change as $\beta \searrow 0$ (a Brownian motion with drift $\lambda_1$), it follows by Slutsky's Theorem that, in the sense of finite-dimensional distributions, 
\[
(\KH_\beta^{\lambda_1}(2\aabullet),\ldots,\KH_\beta^{\lambda_k}(2\aabullet)) \overset{\beta \searrow 0}{\Longrightarrow} (B(2\aabullet) + 2\lambda_1 \aabullet, B(2\aabullet) + 2\lambda_2 \aabullet,\ldots, B(2\aabullet) + 2\lambda_k \aabullet), 
\]
where $B$ is a standard Brownian motion. Convergence on $C(\R^k)$ follows because the marginal distribution of each component on the left-hand side is a Brownian motion with drift $\lambda_i$ and therefore is tight. 
\end{proof}

\begin{proof}[Proof of Corollary \ref{cor:KPZeqtofp}]
Given Theorem \ref{thm:SHlimit}, we follow a similar procedure to the proof of \cite[Corollary 1.9]{Wu-23}. The only needed change is that the joint distribution of the initial data changes with $T$.  Let $\Hh_\beta = \log Z_\beta$. Recalling the definition \eqref{eqn:KPZeqn} of solutions to the KPZ equation, we observe that 
\begin{align*}
    &\quad \; 2^{1/3} T^{-1/3} \Biggl[ \beta \kpzb\Bigl( \f{T t}{\beta^4} ,\f{2^{1/3} T^{2/3} y}{\beta^2};\f{Ts}{\beta^4}, F_\beta^{\alpha + 2^{1/3} T^{-1/3}\lambda_i}(\aabullet) - \alpha \aabullet\Bigr) + \f{T(t - s)}{24} -\f{2}{3}\log(\sqrt 2 T)\Biggr] \\
    &= 2^{1/3}T^{-1/3} \log \int_\R \f{1}{2^{1/3} T^{2/3}} \exp\Biggl(\beta \KH_\beta^{\alpha + \beta 2^{1/3} T^{-1/3}\lambda_i}(x) - \beta \alpha x + \Hh_\beta\Bigl(\f{Tt}{\beta^4},\f{2^{1/3} T^{2/3}y}{\beta^2} \Big | \f{Ts}{\beta^4},x\Bigr) + \f{T(t - s)}{24} \Biggr)\,dx \\
    &= 2^{1/3} T^{-1/3} \log \int_\R\beta^{-2} \exp\Biggl(\beta \KH_\beta^{\alpha + \beta 2^{1/3}T^{-1/3}\lambda_i}\Bigl(\f{2^{1/3} T^{2/3}}{\beta^2} x\Bigr) - \f{2^{1/3}T^{2/3} \alpha x}{\beta}  \\
    &\qquad\qquad\qquad\qquad\qquad\qquad + \Hh_\beta\Bigl(\f{Tt}{\beta^4},\f{2^{1/3} T^{2/3}y}{\beta^2} \Big | \f{Ts}{\beta^4},\f{2^{1/3} T^{2/3}x}{\beta^2}\Bigr) + \f{T(t - s)}{24}\Biggr)\,dx \\
    &\deq 2^{1/3} T^{-1/3} \log \int_\R  \exp\Biggl(\beta \KH_\beta^{\alpha + \beta 2^{1/3}T^{-1/3}\lambda_i}\Bigl(\f{2^{1/3} T^{2/3}}{\beta^2} x\Bigr) - \f{2^{1/3}T^{2/3} \alpha x}{\beta}  \\
&\qquad\qquad\qquad\qquad\qquad\qquad + \Hh_1\Bigl(Tt,2^{1/3} T^{2/3}y \Big | Ts,2^{1/3} T^{2/3}x\Bigr) + \f{T(t - s)}{24}\Biggr)\,dx \\
&= 2^{1/3} T^{-1/3} \log \int_\R  \exp\Biggl(2^{-1/3}T^{1/3}\Biggl[\beta 2^{1/3} T^{-1/3} \KH_\beta^{\alpha + \beta 2^{1/3}T^{-1/3}\lambda_i}\Bigl(\f{2^{1/3} T^{2/3}}{\beta^2} x\Bigr) - \f{2^{2/3}T^{1/3} \alpha x}{\beta}  \\
&\qquad\qquad\qquad\qquad\qquad\qquad + 2^{1/3}T^{-1/3}\Hh_1\Bigl(Tt,2^{1/3} T^{2/3}y \Big | Ts,2^{1/3} T^{2/3}x\Bigr) + \f{2^{1/3}T^{2/3}(t - s)}{24}\Biggr]\Biggr)\,dx \\
&= 2^{1/3} T^{-1/3} \log \int_\R \exp\Bigl(2^{-1/3} T^{1/3}[ F_T^i(x) + \Hh^T(t,y \viiva s,x)]   \Bigr)\,dx,
\end{align*}
where we note that we have added the $o(1)$ term $ -\f{2^{4/3}}{3T^{1/3}}\log(\sqrt 2 T)$ in the first expression. The distributional equality above is theorem \ref{thm:Zinvar}\ref{itm:Zscale}, and we define
\begin{align*}
F_i^T(x) &= \beta 2^{1/3} T^{-1/3} \KH_\beta^{\alpha + \beta 2^{1/3}T^{-1/3}\lambda_i}\Bigl(\f{2^{1/3} T^{2/3}}{\beta^2} x\Bigr) - \f{2^{2/3}T^{1/3} \alpha x}{\beta} \\
\Hh^T(t,y|s,x) &= 2^{1/3}T^{-1/3}\Hh_1\Bigl(Tt,2^{1/3} T^{2/3}y \Big | Ts,2^{1/3} T^{2/3}x\Bigr) + \f{2^{1/3}T^{2/3}(t - s)}{24}\\
h^T_i(t,y) &= 2^{1/3} T^{-1/3} \log \int_\R \exp\Bigl(2^{-1/3} T^{1/3}[ F_T^i(x) + \Hh^T(t,y \viiva s,x)]   \Bigr)\,dx.
\end{align*}

Note that $\{\KH_i^T\}_{1 \le i \le k}$ and $\Hh^T := \{\Hh^T(t,y|s,x): t > s, x,y \in \R\}$ are independent by assumption. 
We observe that $\KH_T^i$ is a Brownian motion with diffusion $\sqrt 2$ and drift $2\lambda_i$; hence its law does not depend on $T$. By \cite[Theorem 1.6]{Wu-23}, $\Hh^T$ converges to $\Ll$ in $C(\Rup,\R)$. By \cite{KPZ_equation_convergence,heat_and_landscape}, for each $i$, $h_i^T := \{h_i^T(t,x;F_i^T):t > s, x \in \R\}$ converges  in distribution on $C(\R_{>s},\R)$ to the KPZ fixed point $h_{\Ll}(t,y;s,G^{\lambda_i}) := \sup_{x \in \R}\{G^{\lambda_i}(x) + \Ll(x,s;y,t)\}$. Hence, this sequence is tight in $C(\R_{>s} \times \R, \R^k)$. All together, the sequence
\be \label{eqn:KPZeqtight}
\bigl(\{\KH_T^i\}_{1 \le i \le k},\Hh^T, \{h^T_i\}_{1 \le i \le k} \bigr)
\ee
is tight on $C(\R,\R^k) \times C(\Rup,\R) \times C(\R_{>s} \times \R,\R)$. Let 
\be \label{eqn:KPZeqlim}
(\{G^{\lambda_i}\}_{1 \le i \le k}, \Ll, \{g_i\}_{1 \le i \le k})
\ee
be a subsequential limit. We may write the first component $\{G^{\lambda_i}\}_{1 \le i \le k}$  and the second as $\Ll$ because we know the SH and DL are, respectively, the law of the limits of the first and second component. by the weak convergence of $h_i^T$, we also know that, marginally, for each $i$, $g_i \deq h_{\Ll}(\aabullet,\aabullet; s, G^{\lambda_i})$.  By Skorokhod representation (\cite[Thm.~11.7.2]{dudl}, \cite[Thm.~3.1.8]{EKbook}), there exists a coupling of \eqref{eqn:KPZeqtight} and \eqref{eqn:KPZeqlim} where, as $T \to \infty$, convergence holds in the sense of uniform convergence on compact sets. 
Now, we follow the procedure of \cite{Wu-23}. We observe that for fixed $t > s$ and $y \in \R$, with probability one,
\begin{align*}
h_{\Ll}(t,y;G^{\lambda_i}) &= \sup_{x \in \R}\{G^{\lambda_i}(x) + \Ll(x,0;y,t)\} \\
&=\lim_{M \to \infty} \sup_{|x| \le M} \{G^{\lambda_i}(x) + \Ll(x,0;y,t)\}  \\
&= \lim_{M \to \infty} \lim_{T \to \infty}2^{1/3} T^{-1/3} \log \int_{-M}^M \exp\Bigl(2^{-1/3} T^{1/3}[ F_T^i(x) + \Hh^T(t,y \viiva s,x)]   \Bigr)\,dx \\
&\le \lim_{M \to \infty} \lim_{T \to \infty} 2^{1/3} T^{-1/3} \log \int_{\R}\exp\Bigl(2^{-1/3} T^{1/3}[ F_T^i(x) + \Hh^T(t,y \viiva s,x)]   \Bigr)\,dx = g_i(t,y). 
\end{align*}
Hence, since we already established $h_{\Ll}(t,y;G^{\lambda_i}) \deq g_i(t,y)$, there exists an event of probability one on which, for $1 \le i \le k$, all $(t,y) \in \Q_{>s} \times \Q$, $h_{\Ll}(t,y;G^{\lambda_i}) \deq g_i(t,y)$. Equality on $\R_{>s} \times \R$ follows on this full probability event by continuity.

We next turn to the convergence of the Busemann process. It suffices to show that, for each $r \in \R$, the following distributional convergence holds, in the sense of uniform convergence on compact sets.
\begin{align*}
    &\Biggl\{2^{1/3}T^{-1/3}\Bigl[b_\beta^{\beta 2^{1/3} T^{-1/3} \lambda_i}\Bigl(\f{Ts}{\beta^4},\f{2^{1/3} T^{2/3} x}{\beta^2}; \f{Tt}{\beta^4}, \f{2^{1/3} T^{2/3} y}{\beta^2}   \Bigr) + \f{T(t - s)}{24}\Bigr]: (x,s;y,t) \in \R_r^4
   \Biggr\}  \\
   &\qquad\qquad\qquad\qquad \overset{T \to \infty}{\Longrightarrow} \{\mathcal B^{\lambda_i}(y,-t;x,-s):(x,s;y,t) \in \R_r^4  \}_{1 \le i \le k},
    \end{align*} 
    where we use the shorthand notation $\R_r^4 = \R \times \R_{> r} \times \R \times \R_{> r}$.
By the dynamic programming principle and additivity of the Busemann process (Theorem \ref{thm:SHEbuse}\ref{itm:SHEbuseadd},\ref{itm:SHEBuseevolve}) as well as the relation between Busemann process and the KPZH (Corollary \ref{cor:SHEBuse=KPZH}), for $s,t > r$,
\be \label{KPZstep}
\begin{aligned}
\bb^{(\beta \lambda)+}(s,x,t,y) &= \bb^{(\beta \lambda) +}(r ,0,t,y) - \bb^{(\beta \lambda)+}(r ,0,s,x) \\
&= \log \int_\R e^{\bb^{(\beta \lambda)+}(r ,0,r ,z)}Z_\beta(t,y\viiva r ,z)\,dz - \log \int_\R e^{\bb^{(\beta \lambda)+}(r ,0,r ,z)}Z_\beta(s,x\viiva r ,z)\,dz \\
&\deq \beta \kpzb(t,y \viiva r , F_{\beta}^{\lambda}) - \beta \kpzb(s,x;\viiva r , F_{\beta}^{\lambda}),
\end{aligned}
\ee
where the distributional equality holds as processes in $\lambda \times (x,s;y,t) \in \R \times \R_{r}^4$. 

Similarly, by the additivity and evolution of the Busemann process from Theorem \ref{thm:DL_Buse_summ}, along with the distributional equalities  $\Ll(x,s;y,t) \deq \Ll(y,-t;x,-s)$ (Lemma \ref{lem:DLsymm}) and the distributional equality between Busemann functions and the SH (Theorem \ref{thm:DL_Buse_summ}\ref{itm:SH_Buse_process}), for $s,t > r$,
\be \label{DLstep}
\begin{aligned}
\mathcal B^{\lambda +}(y,-t;x,-s) &= \mathcal B^{\lambda +}(y,-t;0,-r) - \mathcal B^{\lambda +}(x,-s;0,-r) \\
&=\sup_{z \in \R}\{\Ll(y,-t;z,-r) + \mathcal B^{\lambda +}(z,-r ;0,-r)\} \\
&\qquad\qquad -  \sup_{z \in \R}\{\Ll(x,-s;z,-r) + \mathcal B^{\lambda +}(z,-r;0,-r )\} \\
&\deq \sup_{z \in \R}\{G^{\lambda}(z) + \Ll(z,r;y,t)\} - \sup_{z \in \R}\{G^{\lambda}(z) + \Ll(z,r;x,s)\} \\
&= h_{\Ll}(t,y \viiva r,G^{\lambda}) - h_{\Ll}(s,x \viiva r ,G^{\lambda}),
\end{aligned}
\ee
where, again, the distributional equality holds as processes in $\lambda \times (x,s;y,t) \in \R \times \R_{r}^4$. Here, we have also used the independence of the Busemann process at time $-r$ and the DL for times less than $-r$ (Theorem \ref{thm:DL_Buse_summ}\ref{itm:indep_of_landscape}).
Comparing \eqref{KPZstep} to \eqref{DLstep} and using the first part of the theorem in the $\alpha = 0$ case, we get
\begin{align*}
&\quad \; \Biggl\{2^{1/3}T^{-1/3}\Bigl[b_\beta^{\beta 2^{1/3} T^{-1/3} \lambda_i}\Bigl(\f{Ts}{\beta^4},\f{2^{1/3} T^{2/3} x}{\beta^2}, \f{Tt}{\beta^4}, \f{2^{1/3} T^{2/3} y}{\beta^2}   \Bigr) + \f{T(t - s)}{24}\Bigr]: (x,s;y,t) \in \R_r^4 
   \Biggr\}_{1 \le i \le k} \\
   &\deq  \Biggl\{2^{1/3}T^{-1/3}\Bigl[h_{Z_{\beta}}\Bigl(\f{Tt}{\beta^4}, \f{2^{1/3} T^{2/3 x}}{\beta^2}  \viiva \f{T r}{\beta^4},\beta F_{\beta}^{2^{1/3} T^{-1/3} \lambda_i} \Bigr) + \f{T(t - r)}{24} - \f{2}{3} \log(\sqrt 2 T)  \\
   &\qquad -h_{Z_{\beta}}\Bigl(\f{Ts}{\beta^4}, \f{2^{1/3} T^{2/3 y}}{\beta^2}  \viiva \f{T r}{\beta^4},\beta F_{\beta}^{2^{1/3} T^{-1/3} \lambda_i} \Bigr) - \f{T(s - r)}{24} +  \f{2}{3} \log(\sqrt 2 T)     \Bigr]:  (x,s;y,t) \in \R_r^4 \Biggr\}_{1 \le i \le k} \\
   &\overset{T \to \infty}{\Longrightarrow} \Bigl\{h_{\Ll}(t,y \viiva r, G^{\lambda_i}) - h_{\Ll}(s,x \viiva r,G^{\lambda_i}): (x,s;y,t) \in \R_r^4 
   \Biggr\}_{1 \le i \le k} \\
   &\deq \{\mathcal B^{\lambda_i}(y,-t;x,-s):(x,s;y,t) \in \R_r^4  \}_{1 \le i \le k}. \qedhere 
\end{align*}
\end{proof}

\appendix
\section{Queues and the O'Connell-Yor polymer} \label{appx:queue}
Recall the transformations $\Qp_\beta,\Dp_\beta,\Rp_\beta$ defined in \eqref{QPDpRp}.
We state one of the main theorems from \cite{brownian_queues}. Their  theorem is stated for $\beta= 1$. The statement for general $\beta > 0$ follows from Lemma \ref{lem:pQscaling}  since 
\[
D_\beta(B,Y) = T_{\beta} D_1(T^2_{\beta^{-1}}(B,Y)),\qquad\text{and}\qquad R_\beta(B,Y) = T_\beta R_1(T^2_{\beta^{-1}}(B,Y)). 
\]
\begin{theorem} \cite[Theorem 5]{brownian_queues}, \cite[Theorem 2.1]{Matsumoto-Yor-2001b} \label{thm:ocy}
Let $B$ and $Y$ be independent two-sided Brownian motions with drift so that the drift of $Y$ 
is strictly larger than the drift of $B$. Let $\beta > 0$, and  let $Q_\beta,D_\beta,R_\beta$ be defined as in \eqref{QPDpRp}. Then, $(R_\beta(B,Y),D_\beta(B,Y)) \deq (B,Y)$, and for each $y \in \R$, $\{D_\beta(Y,B)(x),R_\beta(Y,B)(x): -\infty < x \le y \}$ is independent of $\{Q_\beta(x):x \ge y\}$.
\end{theorem}

\begin{lemma} \label{lem:int_Gamma}
Let $B$ be a standard two-sided Brownian motion, and let $\beta,\lambda > 0$. Then,
 \[
 \Bigl(\int_{-\infty}^0 e^{\sqrt 2 \beta B(x) + \lambda \beta x}\Bigr)^{-1} 
 \;\sim\; \text{\rm Gamma}(\lambda \beta^{-1},\beta^{-2}). \qedhere
 \]
\end{lemma}
\begin{proof}
Theorem 4.4 in \cite{Dufresne-1990} (also Equation 1.8.4(1) on page 612 4 of \cite{BM_handbook}) states that for $\gamma,\sigma > 0$, 
\[
\Bigl(\int_0^\infty e^{-\sigma B(x) - \gamma x }\,dx\bigr)^{-1} \sim \text{ \rm Gamma}(2\gamma \sigma^{-2},2 \sigma^{-2})
\]
(as a caution, we note that \cite{Dufresne-1990} parameterizes the Gamma distribution by shape and scale, so the result is stated there with scale $\sigma^2/2$). Observe that 
\[
\int_{-\infty}^0 e^{\sqrt 2 \beta B(x) + \lambda \beta x} = \int_0^\infty e^{\sqrt 2 \beta B(-x) - \lambda \beta x} \deq \int_0^\infty e^{-\sqrt 2 \beta B(x) - \lambda \beta x},
\]
and so the result follows upon substituting $\sigma = \sqrt 2 \beta$ and $\gamma = \lambda \beta$. 
\end{proof}

The remainder of this appendix contains the omitted proofs from Section \ref{sec:consistent_invar}, along with some additional lemmas. These follow similarly as for zero temperature in \cite{CGM_Joint_Buse} and \cite{Seppalainen-Sorensen-21b}, modulo the necessary inputs from Section \ref{sec:queue}. We first prove an intermediate lemma.

\begin{lemma} \label{lem:extint}
Let $n \ge 2$, and let $(B^1,Y^1,\ldots,Y^n)$ be such that the following operations are well-defined. Let $\beta > 0$. For $2 \le j \le n$, define $B^j = R_\beta(B^{j - 1},Y^{j - 1})$. Then, for $1 \le k \le n - 1$, 
\be
D_\beta^{(n +1)}(B^1,Y^1,\ldots,Y^n) = D_\beta^{(k + 1)}\bigl(D_\beta(B^1,Y^1),\ldots,D_\beta(B^k,Y^k), D_\beta^{(n - k + 1)}(B^{k+1},Y^{k+1},\ldots,Y^n)\bigr).
\ee
\end{lemma}
We note that the case $k= n - 1$ of Lemma \ref{lem:extint} gives us
\begin{equation} \label{intertwining}
    \Dp_\beta^{(n + 1)}(B^1,Y^1,\ldots,Y^n) = \Dp_\beta^{(n)}(\Dp_\beta(B^1,Y^1),\ldots,\Dp_\beta(B^n,Y^n)).
\end{equation}
\begin{proof}[Proof of Lemma \ref{lem:extint}]
Equation \eqref{fullid} gives us the statement for $n = 2$. Assume, by induction, that the statement is true for some $n - 1 \ge 2$. We will show the statement is also true for $n$. We first prove the case $k = 1$. Using \eqref{fullid} in the second equality below,
\begin{align*}
    &\;\;\;\Dp_\beta^{(2)}(\Dp_\beta(B^1,Y^1),\Dp_\beta^{(n)}(B^2,Y^2,\ldots,Y^n)) \\
    &= \Dp_\beta(\Dp_\beta(B^1,Y^1),\Dp_\beta(B^2,\Dp_\beta^{(n -1)}(Y^2,\ldots,Y^n))) \\
    &= \Dp_\beta(B^1,\Dp_\beta(Y^1,\Dp_\beta^{(n - 1)}(Y^2,\ldots,Y^n))) \\
    &= \Dp_\beta(B^1,\Dp_\beta^{(n)}(Y^1,\ldots,Y^n)) \\
    &= \Dp_\beta^{(n + 1)}(B^1,Y^1,\ldots,Y^n).
\end{align*}
Now, let $2 \le k \le n - 1$. Then, by definition of $D^{(k + 1)}$ and the induction assumption,
\begin{align*}
    &\Dp_\beta^{(k + 1)}(\Dp_\beta(B^1,Y^1),\ldots,\Dp_\beta(B^k,Y^k), \Dp_\beta^{(n - k + 1)}(B^{k + 1},Y^{k+1},\ldots,Y^{n})) \\
    = &\Dp_\beta(\Dp_\beta(B^1,Y^1),\Dp_\beta^{(k)}(\Dp_\beta(B^2,Y^2),\ldots,\Dp_\beta(B^k,Y^k),\Dp_\beta^{(n - k + 1)}(B^{k+1},Y^{k+1},,\ldots,Y^{n}))) \\
    = &\Dp_\beta(\Dp_\beta(B^1,Y^1),\Dp_\beta^{(n)}(B^2,Y^2,\ldots,Y^n)) = \Dp_\beta^{(2)}(\Dp_\beta(B^1,Y^1),\Dp_\beta^{(n)}(B^2,Y^2,\ldots,Y^n)).
\end{align*}
The lemma now follows from the $k = 1$ case.
\end{proof}

The \textit{multiline process} is a discrete-time Markov chain on the state space $\Y_n^{(a,\infty)}$ of \eqref{Yndef}, where $a \in \R$. The analogous process is defined in a discrete setting for particle systems in \cite{Ferrari-Martin-2007,Ferrari-Martin-2005,Ferrari-Martin-2009}, for  lattice last-passage percolation in \cite{CGM_Joint_Buse}, and in zero temperature BLPP in \cite{Seppalainen-Sorensen-21b}. 
Starting at time $m-1$ in state $\mathbf Y_{m - 1} =\mathbf Y =  (Y^1,Y^2,\ldots, Y^n) \in \Y_n^{(a,\infty)}$the time $m$ state is given as 
\begin{equation*} 
\mathbf Y_{m} = \overline{\mathbf Y} = (\overline Y^1,\overline Y^2,\ldots, \overline Y^n) \in \Y_n
\end{equation*}
is defined as follows. Let $B \in \CRpin$ satisfy
\[
\lim_{x \rightarrow -\infty} x^{-1}B(x) = a. 
\]
First, set $B^1 = B$,  and $\overline Y^1 = \Dp_\beta(Y^1,B^1)$. 
    Then, iteratively for $i = 2,3,\ldots,n$:
    \be \label{multiline process}
    B^i = \Rp_\beta(B^{i - 1},Y^{i - 1}), \qquad\text{and}\qquad
    \overline Y^i = \Dp_\beta(B^i,Y^i).
    \ee

\begin{lemma} \label{multiline process well-defined}
The mapping \eqref{multiline process} is well-defined on the state space $\Y_n^{(a,\infty)}$.
\end{lemma}
\begin{proof}
This follows from Lemma \ref{lem:OCYlim}: By induction, each $B^i$ satisfies
\[
\lim_{x \rightarrow -\infty}\frac{B^i(x)}{x} = a.
\]
Therefore, since $\mathbf Y \in \Y_n^{(a,\infty)}$, for $1 \le i \le n$,
\[
\limsup_{x \rightarrow \infty}\f{Y^i(x) - B^i(x)}{x} > 0. \qedhere
\]
\end{proof}

\begin{theorem} \label{multiline invariant distribution}
At each step of the evolution of the multiline process, take the driving function $B$  to be an independent standard, two-sided Brownian motion with drift $a \in \R$. For each $\bar \lambda = (\lambda_1,\ldots,\lambda_n) \in \R^n_{> 0}$ with $a < \lambda_1 < \cdots < \lambda_n$, the measure $\nu^{\bar \lambda}$ on $\Y_n^{(a,\infty)}$ is invariant for the multiline process \eqref{multiline process} 
\end{theorem}
\begin{proof}
Assuming that $\mathbf Y = (Y^1,\ldots, Y^n) \in \Y_n^{\R_{>0}}$ are i.i.d. Brownian motions with drifts $\lambda_1,\ldots,\lambda_n$, we must show that the same is true for $\overline Y^1,\ldots,\overline Y^n$.  By Theorem \ref{thm:ocy}, 
$
\overline Y^1 = \Dp_\beta(B^1,Y^1)
$
is a two-sided Brownian motion with  drift $\lambda_1$, independent of
$
B^2 = \Rp_\beta(B^1,Y^1),
$
which is a two-sided Brownian motion with drift $a$. Hence, the random paths $\overline Y^1, B^2,Y^2,\ldots, Y^n$ are mutually independent. We iterate this process as follows: Assume, for some $2 \leq k \leq n - 1$, that the random paths $\overline Y^1,\ldots,\overline Y^{k - 1},B^k,Y^k,\ldots, Y^n$ are mutually independent, where for $1 \leq i \leq k - 1$, $\overline Y^i$ is a Brownian motion with drift $\lambda_i$. Applying Theorem \ref{thm:ocy} again,
$
\overline Y^k = \Dp_\beta(B^k,Y^k)
$
is a two-sided Brownian motion with drift $\lambda_k$, independent of 
$
B^{k + 1} = \Rp_\beta(B^k,Y^k),
$
which is a two-sided Brownian motion with zero drift.
Since $(\overline Y^k,B^{k + 1})$ is a function of $(B^k,Y^k)$, it follows that $\overline Y^1,\ldots,\overline Y^k,B^{k + 1},Y^{k + 1},\ldots,Z^n$ are mutually independent, completing the proof.
\end{proof}

\begin{proof}[Proof of Lemma \ref{lem:OCY_consis} (Consistency of the measures)]    
It suffices to show that if $(\eta^1,\ldots,\eta^n)$ has distribution $\mu_\beta^{\lambda_1,\ldots, \lambda_n}$, then \[(\eta^1,\ldots,\eta^{i - 1},\eta^{i + 1},\ldots, \eta^n) \sim \mu_\beta^{\lambda_1,\ldots,\lambda_{i - 1},\lambda_{i + 1},\ldots,\lambda_n}.
\] 
Let $\mbf Y = (Y^1,\ldots,Y^n) \sim \nu^{\bar \lambda}$ and $\eta = \sDp^{(n)}_\beta(\mbf Y)$ so $\eta = (\eta^1,\ldots,\eta^n) \sim \mu_\beta^{\bar \lambda}$.

   For $i = n$, the statement is immediate from the definition of the map $\sDp_\beta^{(n)}$. Next, we show the case $i = 1$. For $2 \le j \le n$, by \eqref{intertwining}, we may write
\[
D_\beta^{(j)}(Y^1,\ldots,Y^j) = D_\beta^{(j - 1)}(D_\beta(\wt Y^1, Y^{2}),D_\beta(\wt Y^2, Y^{3}),\ldots,D_\beta(\wt Y^{j - 1},Y^j)),
\]
where $\wt Y^{1} = Y^1$, and for $i > 1$, $\wt Y^{i} = R(\wt Y^{i - 1},Y^{i})$. Then $(\eta^2,\ldots,\eta^n) = \sDp_\beta^{(n - 1)}(\hat Y^2,\ldots,\hat Y^n)$, where $\hat Y^i = D_\beta(\wt Y^{i - 1},Y^i)$ for $2 \le i \le n$. By Theorem \ref{multiline invariant distribution}, $\hat Y^2,\ldots,\hat Y^n$ are independent, completing the proof in the case $i = 1$. Using the definition of $D_\beta^{(j)}$ \eqref{hat D iterated}, for $i < j \le n$,
\begin{equation} \label{queue_iter}
D_\beta^{(j)}(Y^1,\ldots,Y^j) = D_\beta(D_\beta(Y^1,D_\beta(Y_2,\cdots D_\beta(Y^{i-1},D_\beta^{(j - i + 1)}(Y^i,\ldots,Y^{j}))\cdots).
\end{equation}
 We apply \eqref{intertwining}, just as in the $i = 1$ case, to obtain
\begin{align} \label{queue_iter_2}
&\;\;\;D_\beta^{(j - i + 1)}(Y^i,\ldots,Y^j)\nonumber \\
&= D_\beta^{(j - i)}(D_\beta(\wt Y^{i},Y^{i + 1}),\ldots,D_\beta(\wt Y^{j - 1},Y^j)) = D_\beta^{(j - i)}(\hat Y^{i + 1},\ldots,\hat Y^j),
\end{align}
where, $\wt Y^{i} = Y^i$, and for $j > i$, $\wt Y^{j} =R(\wt Y^{j- 1},Y^j)$. For $j > i$, we define $\hat Y^j = D_\beta(\wt Y^{j - 1},Y^j)$.
Then, by \eqref{queue_iter} and \eqref{queue_iter_2},  when $i < j \le n$,
\[D_\beta^{(j)}(Y^1,\ldots,Y^j) = D_\beta^{(j - 1)}(Y^{1},\ldots,Y^{i-1},\hat Y^{i+1},\ldots,\hat Y^{j} ),
\]
 and thus, 
\be \label{eqn:marg}
(\eta^1,\ldots,\eta^{i - 1},\eta^{i + 1},\ldots,\eta^n) = \sDp_\beta^{(n - 1)}(Y^1,\ldots,Y^{i - 1},\hat Y^{i + 1},\ldots,\hat Y^n).
\ee 
By Theorem \ref{multiline invariant distribution}, $\hat Y^{i + 1},\ldots,\hat Y^{n}$ are independent Brownian motions with drifts $\lambda_{i + 1},\ldots,\lambda_n$. These random paths are functions of $Y^{i},\ldots,Y^n$,so the paths functions $Y^1,\ldots,Y^{i - 1},\hat Y^{i + 1},\ldots,\hat Y^j$ are also independent. Thus, by \eqref{eqn:marg}, \[
(\eta^1,\ldots,\eta^{i - 1},\eta^{i + 1},\ldots, \eta^n) \sim \mu^{(\lambda_1,\ldots,\lambda_{i - 1},\lambda_{i + 1},\ldots,\lambda_n)}.  \qedhere
\]
\end{proof}

\begin{proof}[Proof of Theorem \ref{thm:OCYinvar}]
Let $\mathbf Y \sim \nu^{\bar \lambda}$. Let $\boldsymbol \eta = \sDp_\beta^{(n)}(\mathbf Y)$ so that $\boldsymbol\eta \sim \mu_\beta^{\bar \lambda}$. Then, for Brownian motion $B$, let $\Ss_{\beta}^B$ denote the mapping of a single evolution step of $\mathbf Y$ according to the multiline process \eqref{multiline process} and $\T_\beta^B$ denote the mapping of a single evolution step of $\boldsymbol\eta$ according to the Markov chain \eqref{eqn:OCYMC}. By definition of $\Dp_\beta^{(k)}$ and Equation \eqref{intertwining},
\begin{align*}
[\T_\beta^{B}(\boldsymbol\eta)]_k = &\Dp_\beta(\boldsymbol\eta^k,B) = \Dp_\beta(B^1,\Dp_\beta^{(k)}(Y^1,\ldots,Y^k)) = \Dp_\beta^{(k + 1)}(B^1,Y^1,\ldots,Y^k) \\
= &\Dp_\beta^{(k)}(\Dp_\beta(B^1,Y^1),\Dp_\beta(B^{2},Y^{2}),\ldots,\Dp_\beta(B^k,Y^k))  \\
= &\Dp_\beta^{(k)}([\Ss_\beta^{B}(\mathbf Y)]_1,[\Ss_{\beta}^{B}(\mathbf Y)]_{2 },\ldots,[\Ss_\beta^{B}(\mathbf Y)]_k) = [\sDp_\beta^{(n)}(\Ss_\beta^B(\mathbf Y))]_k.
\end{align*}
Therefore, $\T_\beta^B(\boldsymbol\eta) = \sDp_\beta^{(n)}(\Ss_\beta^B(\mathbf Z))$, and because $\boldsymbol\eta = \sDp^{(n)}(\mathbf Z)$, we have
\[
\T^B(\D^{(n)}(\mathbf Y)) = \D^{(n)}(\Ss^B(\mathbf Y)).
\]
Theorem \ref{multiline invariant distribution} implies that $\Ss_\beta^B(\mathbf Y) \overset{d}{=} \mathbf Y \sim \nu^{\bar \lambda}$. Therefore, $\T_\beta^B(\boldsymbol\eta) \overset{d}{=} \sDp_\beta^{(n)}(\mathbf Y) \sim \mu^{\bar \lambda}$.
\end{proof}

\section{Stationary horizon and   the directed landscape} \label{sec:SH}
The stationary horizon (SH) was first introduced by Busani in \cite{Busani-2021} and was later studied by Busani and the third and fourth authors in \cite{Seppalainen-Sorensen-21b,Busa-Sepp-Sore-22a,Busa-Sepp-Sore-22b}. We refer to those articles for a more complete description. Analogously to how the $\KPZH_\beta$ describes the jointly invariant measures for the O'Connell-Yor polymer and the KPZ equation, it was proved in \cite{Seppalainen-Sorensen-21b,Busa-Sepp-Sore-22a} that the SH describes the jointly invariant measures for Brownian last-passage percolation and the KPZ fixed point. Jointly invariant measures for the KPZ fixed point are made precise through the coupling with the directed landscape. See \cite{KPZfixed,Directed_Landscape,KPZ_equation_convergence,heat_and_landscape,Dauvergne-Virag-21,Wu-23} for more on the KPZ fixed point and directed landscape. We briefly describe the needed definition and facts about the SH, DL, and the KPZ fixed point here. 

The directed landscape (DL) is a random continuous function $\Ll:\Rup \to \R$. By convention, we switch the ordering of space-time coordinates to $\Ll(x,s;y,t)$ (in contrast to the ordering in $Z_\beta(t,y \viiva s,x)$). Given the DL, we can construct the KPZ fixed point started from time $s$ as 
\[
h_{\Ll}(t,y|s,\h) = \sup_{x \in \R}\{\h(x) + \Ll(x,s;y,t)\}, \qquad  t>s, \ y\in\R.
\]

SH is constructed with the zero-temperature counterparts of the mappings of Section \ref{sec:queue}. We denote these with the same letters but without the $\beta$ subscript. 
For functions that satisfy $Y(0) = B(0) = 0$ and $\limsup_{x \to -\infty} Y(x) - B(x) = -\infty$, define 
\be \label{Phialt}
D(B,Y)(y) = B(y) + \sup_{-\infty < x \le y }\{Y(x) - B(x)\} - \sup_{-\infty < x \le 0}\{Y(x) - B(x)\}. 
\ee
As in \eqref{hat D iterated}, iterate the mapping $D$ as follows:
\[
    D^{(1)}(Y) = Y, \quad\text{and}\quad
    D^{(n)}(Y^1,Y^{2},\ldots,Y^n) = D(Y^1,D^{(n - 1)}(Y^2,\ldots,Y^{n}))\quad \text{ for } n \geq 2. 
\]
A mapping $\sDp^{(n)}:\Y_n^\R \to \X_n^\R$ 
is defined as follows:  the image $\boldsymbol \eta = (\eta^1,\ldots, \eta^n)=\sDp^{(n)}(\mbf Z) \in \X_n$ is defined  for $\mathbf Y  = (Y^1,\ldots,Y^n)\in \Y_n^\R$  by 
\[
\eta^i = D^{(i)}(Y^1,\ldots,Y^i) \text{ for } 1 \le i \le n.
\]
For $\bar \lambda = (\lambda_1 < \cdots < \lambda_n)$, we define the measure $\mu^{\lambda_1,\ldots,\lambda_n}$ (again without the $\beta$ subscript) as 
\[
\mu^{\bar \lambda} =\nu^{\bar \lambda} \circ (\mathcal D^{(n)})^{-1}. 
\]

\begin{definition} \label{def:SH}
The stationary horizon $\{G_\mu\}_{\mu \in \R}$ is a process with paths in $D(\R,C(\R))$.
Its law is characterized as follows: For real numbers $\bar \lambda = (\lambda_1 < \cdots < \lambda_k)$, the $k$-tuple  $(G^{\lambda_1},\ldots,G^{\lambda_k})\in C(\R)^k$  has distribution  $\mu^{\bar \lambda} \circ (\wt T^k_2)^{-1}$, where $\wt T_2$ is the mapping $C(\R)^k \to C(\R)^k$ defined by
\[
\wt T_2^k(f_1,\ldots, f_k)(x) = (f_1(2x),\ldots,f_k(2x)). 
\]
\end{definition}
In this definition, we multiply by a factor of $2$ so that the marginal distributions are Brownian motions with diffusivity $\sqrt 2$. This is the correct parameterization for invariance under the KPZ fixed point.

\begin{lemma}[\cite{Busani-2021}, Theorem 1.2;  \cite{Seppalainen-Sorensen-21b}, Theorems 3.6(iii), 5.4] \label{lem:SHscal}
For $c > 0$ and $\nu \in \R$, \[
    \{cG_{c (\mu + \nu)}(c^{-2}x) - 2\nu x  : x\in \R\}_{\mu \in \R} \,\deq\, \{G_\mu(x): x \in \R\}_{\mu \in \R}.
    \]
\end{lemma}

\begin{lemma}[\cite{Seppalainen-Sorensen-21b}, Lemma 7.2, and see  Appendix D in \cite{Seppalainen-Sorensen-21a}] \label{lem:SHqueuerep}
For $\mbf Y = (Y^1,\ldots,Y^n)$, define
\[
A^{\mathbf Y}_{n}(x) = \sup_{-\infty < x_{n - 1} \leq \cdots \leq x_{1} \leq x} \Big\{\sum_{i = 1}^{n} Y^i(x_i) - Y^{i - 1}(x_{i})   \Big\}.
\]
Then, if $A^{\mbf Y}_n(0)$ is finite, for $n \geq 2$,
\begin{align*}
D^{(n)}(Y^1,Y^{2},\ldots, Y^n)(x) 
= &Y^1(x) + A^{\mathbf Y}_{n}(x) - A^{\mathbf Y}_{n}(0).
\end{align*}
\end{lemma}

The following states properties of the Busemann process for the DL from \cite{Busa-Sepp-Sore-22a}. For a single direction $\lambda$, these properties were previously established in \cite{Rahman-Virag-21}. 
\begin{theorem}\cite[Theorems 5.1--5.2]{Busa-Sepp-Sore-22a} \label{thm:DL_Buse_summ}
 On the probability space $(\Omega,\F,\Pp)$ of the directed landscape $\Ll$, there exists a process
\[
\{\mathcal B^{\lambda \sig}(p;q): \dir \in \R, \,  \sigg \in \{-,+\}, \, p,q \in \R^2\}
\]
satisfying the following properties.  All the  properties below hold on a single event of probability one, simultaneously for all directions $\dir \in \R$, signs $\sigg \in \{-,+\}$, and points $p,q \in \R^2$, unless otherwise specified.
\begin{enumerate} [label=\rm(\roman{*}), ref=\rm(\roman{*})]  \itemsep=3pt
\item{\rm(Continuity)} \label{itm:general_cts}  As an $\R^4 \to \R$ function,  $(x,s;y,t) \mapsto \mathcal B^{\lambda \sig}(x,s;y,t)$ is  continuous. 
 \item {\rm(Additivity)} \label{itm:DL_Buse_add} For all $p,q,r \in \R^2$, 
    $\mathcal B^{\lambda \sig}(p;q) + \mathcal B^{\lambda \sig}(q;r) = \mathcal B^{\lambda \sig}(p;r)$.   In particular, $\mathcal B^{\lambda \sig}(p;q) = -\mathcal B^{\lambda \sig}(q;p)$ and $\mathcal B^{\lambda \sig}(p;p) = 0$.
    \item {\rm(Backwards evolution as the KPZ fixed point)}\label{itm:Buse_KPZ_description} For 
    all $x,y \in \R$ and $s < t$,
    \be\label{W_var}
    \mathcal B^{\lambda \sig}(x,s;y,t) = \sup_{z \in \R}\{\Ll(x,s;z,t) + \mathcal B^{\lambda \sig}(z,t;y,t)\}.
    \ee
    \item {\rm(Independence)} \label{itm:indep_of_landscape} For each $T \in \R$, these processes are independent: 
    \begin{align*}
    &\{\mathcal B^{\lambda \sig}(x,s;y,t): \dir \in \R, \,\sigg \in \{-,+\}, \, x,y \in \R, \, s,t \ge T \} \\[4pt] 
    &\qquad\qquad\qquad \text{and } \ \ 
    \{\Ll(x,s;y,t): x,y \in \R,\, s < t \le T\}. 
    \end{align*}
    \item  {\rm(Distribution along a time level)}\label{itm:SH_Buse_process} For each $t \in \R$,  the following equality in distribution holds between random elements of the Skorokhod space $D(\R,C(\R))$:
\[
\{\mathcal B^{\lambda +}(\aabullet,t;0,t)\}_{\dir \in \R} \deq \bigl\{G_{\dir}(\aabullet) \bigr\}_{\dir \in \R},
\]
where $G$ is the stationary horizon.
    \end{enumerate}
\end{theorem}  

We also make use of the following symmetry of the directed landscape.
\begin{lemma}\cite[Proposition 14.1]{Dauvergne-Virag-21} \label{lem:DLsymm}   The directed landscape satisfies the following symmetry
\[
\{\Ll(x,s;y,t): (s,x,t,y) \in \Rup \} \deq \{\Ll(y,-t;x,-s): (s,x,t,y) \in \Rup \}. 
\]
\end{lemma}

\bibliographystyle{plain}
\bibliography{references_file}

\end{document}